\documentclass[11pt,reqno]{amsart}

\usepackage{amssymb,amsthm, amsmath,amstext,amsfonts,amscd,xcolor}
\usepackage{mathtools}
\usepackage{dsfont}
\usepackage{enumitem}
\usepackage[ansinew]{inputenc}
\usepackage{graphicx}
\usepackage[mathscr]{eucal}
\usepackage[colorlinks, 
    linkcolor={black!10!red},
	citecolor={black!20!blue},
	urlcolor={black!30!blue}
]{hyperref}
\usepackage[margin=1.5in]{geometry}

\newtheorem{theorem}{Theorem}[section]
\newtheorem*{Theorem*}{Theorem 1.1}
\newtheorem{Lemma}{Lemma}[section]
\newtheorem{Proposition}{Proposition}[section]
\newtheorem{Corollary}{Corollary}[section]
\newtheorem{remark}{Remark}[section]
\newtheorem{definition}{Definition}[section]
\newtheorem{Claim}{Claim}
\numberwithin{equation}{section}

\begin{document}
	
	\title[Arithmetic $k$-PDL for $k$ Q-P Long-Range Operators on $\mathbb{Z}^d$]{Arithmetic $k$-Polynomial Dynamical Localization for $k$ Power-Law Quasi-Periodic Long-Range Operators on $\ell^2(\mathbb{Z}^d)$}
	
	\author[Ao Cai]{Ao Cai}
	\address[Ao Cai]{
		School of Mathematical Sciences, Soochow University, Suzhou, Jiangsu, China.
	}
	\email{acai@suda.edu.cn}
	
	\author[Huihui Lv]{Huihui Lv}
	\address[Huihui Lv]{
		School of Mathematical Sciences, Soochow University, Suzhou, Jiangsu, China.
	}
	\email{20254007009@stu.suda.edu.cn}
	
	\author[Yuan Shan]{Yuan Shan}
	\address[Yuan Shan]{
		Department of Mathematics, Nanjing Audit University, Nanjing, China.
	}
	\email{shanyuan@nau.edu.cn}
	
	\author[Xueyin Wang]{Xueyin Wang}
	\address[Xueyin Wang]{
		Department of Mathematics, Texas A\&M University, College Station, TX, 77843, USA.
	}
	\email{xueyin@tamu.edu}

	\begin{abstract}
	We establish a criterion of arithmetic $k$-polynomial spectral localization and $k$-polynomial dynamical localization in expectation for quasi-periodic long-range operators on $\ell^2(\mathbb{Z}^{d})$ with power-law hopping based on the quantitative $C^{k}$-reducibility of the dual Schr\"odinger cocycle. As the application, we prove both localization properties for power-law long-range perturbations of the Almost Mathieu Operators with sufficiently large couplings and Diophantine frequencies. 
    \end{abstract}
	
	\maketitle
	
	\section{Introduction}
    In this paper, we study spectral and dynamical localization for the quasi-periodic long-range operator on $\ell^{2}(\mathbb{Z}^{d})$, 
	\begin{equation}\label{longrange}
		(L_{\lambda,V,\alpha,\theta} \psi)(n) = \sum_{m \in \mathbb{Z}^{d}} \widehat{V}_{m} \psi(n - m) + 2 \lambda\cos 2\pi(\theta + \langle n, \alpha \rangle) \psi(n), \quad n \in \mathbb{Z}^d, 
	\end{equation}
	where $\lambda>0$ is the coupling constant, $\widehat{V}_{m}$ denotes the $m$-th Fourier coefficient of $V \in C^{k}(\mathbb{T}^{d}, \mathbb{R})$, $\theta \in \mathbb{T}$ is the phase, and $\alpha \in \mathbb{T}^{d}$ is a frequency vector such that $(1, \alpha)$ is rationally independent. In particular, when $V(x)=2\cos 2\pi x$ and $d = 1$, the operator \eqref{longrange} reduces to the well-known Almost Mathieu Operator (AMO), 
	\begin{equation}\label{amo}
		(H_{\lambda,\cos,\alpha,\theta}\psi)(n)=\psi(n+1)+\psi(n-1)+2\lambda \cos 2\pi(\theta+n\alpha)\psi(n). 
	\end{equation}
    In the light of Aubry duality, $\{\lambda^{-1}L_{\lambda,V,\alpha,\theta}\}_{\theta\in\mathbb{T}}$ is dual to the following multi-frequency quasi-periodic Schr\"odinger operators $\{H_{\lambda^{-1},V,\alpha,x}\}_{x \in \mathbb{T}^{d}}$, 
	\begin{equation*}
		(H_{\lambda^{-1},V,\alpha,x}u)(n)=u(n+1)+u(n-1)+ \lambda^{-1}V(x+n\alpha)u(n), \quad n\in\mathbb{Z}, 
	\end{equation*}
	where $x\in\mathbb{T}^{d}$. For simplicity, we set $\lambda = 1$ and omit the dependence on $\lambda$ when no confusion arises. 
    
    Let $\{\delta_{n}\}_{n \in \mathbb{Z}^{d}}$ be the standard orthonormal basis of $\ell^{2}(\mathbb{Z}^{d})$. For a finitely supported initial state $\phi$ with $\|\phi\|_{\ell^{2}}=1$, define $\psi(t)=e^{-itL_{V,\alpha,\theta}}\phi$. To describe the dynamical behavior of $\psi(n,t)$, we use the position operator defined by
	\begin{equation*}
		(X \psi)(n) = n \psi(n).
	\end{equation*}
	The $p$-th moment ($p > 0$) of the position operator at time $t$ is then given by
	\begin{equation}\label{pmom}
		\langle \psi(t), |X|^{p} \psi(t) \rangle = \sum_{n \in \mathbb{Z}^{d}} |n|^{p} |\langle \delta_{n}, \psi(t) \rangle|^{2}, 
	\end{equation}
    understood in the quadratic form sense and allowed to be infinite. We establish arithmetic polynomial spectral localization and polynomial dynamical localization in expectation for this family.

	\subsection{Polynomial dynamical localization}
	Since the pioneering work of P. Anderson \cite{PhysRev.109.1492}, the phenomenon of localization has been extensively studied. We say that an operator $H$ exhibits \emph{spectral localization} if it has pure point spectrum, which implies all eigenfunctions belong to $\ell^{2}(\mathbb{Z}^{d})$. If, in addition, every eigenfunction decays exponentially, we say that $H$ exhibits \emph{Anderson localization}. 
	
	For quasi-periodic operators, the localization property depends not only on the potential but also on the arithmetic properties of the frequency $\alpha$ and the phase $\theta$.
	
	Recall that $\alpha \in\mathbb{T}^d$ is called {\it Diophantine} if there exist $\kappa>0$ and $\tau>d$ such that $\alpha \in {\rm DC}_d(\kappa,\tau)$, where
	\begin{equation}\label{dio1}
		\mathrm{DC}_d(\kappa,\tau) \coloneq \big\{\alpha \in\mathbb{T}^d: \inf_{j\in \mathbb{Z}}|\langle n,\alpha \rangle -j|
		> \frac{\kappa}{|n|^{\tau}},\quad \forall \, n\in\mathbb{Z}^d{\setminus}\{0\} \big\}.
	\end{equation}
	Here we denote $|n|=|n_1|+|n_2|+\cdots+|n_d|$ and $\langle n,\alpha \rangle=n_1\alpha_1+n_2\alpha_2+\cdots+n_d\alpha_d$. Denote $\mathrm{DC}_d=\cup_{\kappa>0} \mathrm{DC}_d(\kappa,\tau)$, which is of full Lebesgue measure. 
	
	A phase $\phi \in \mathbb{T}$ is called {\it Diophantine} with respect to $\alpha$ if it satisfies $\phi \in \Theta_\gamma^{\tau}$, where $\gamma>0$, $\tau>d$ and
	\begin{equation}\label{dio2}
		\Theta_\gamma^{\tau} \coloneq \big\{\phi \in \mathbb{T}: \inf_{j \in \mathbb{Z}}|2\phi-\langle n,\alpha \rangle -j|
		\geqslant \frac{\gamma}{(|n|+1)^{\tau}}, \quad \forall \, n \in \mathbb{Z}^d \big\}.
	\end{equation}
	Denote $\Theta^{\tau}=\bigcup_{\gamma>0} \Theta_\gamma^{\tau}$, which is of full measure in $\mathbb{T}$. A phase $\phi \in \mathbb{T}$ is called rational with respect to  $\alpha$ if $2\phi = \langle n_0,\alpha \rangle$ mod $\mathbb{Z}$ for some $n_0 \in \mathbb{Z}^d$.
	
	After decades of extensive study, the spectral behavior of the AMO is now well understood in terms of the arithmetic properties of $\alpha$ and $\theta$. The frequency and phase resonances play important roles in the spectral transitions, see \cite{MR2521117, avila2017sharp, Jitomirskaya1999, MR3779957,MR4756946,MR4836219,MR4686650}.
	
	For the analytic Schr\"odinger operator $H_{V,\alpha,x}$ in the positive Lyapunov exponent regime, Bourgain and Goldstein \cite{MR1815703} proved that for fixed $x \in \mathbb{T}^{d}$ with $d=1,2$, the operator $H_{V,\alpha,x}$ exhibits Anderson localization for almost every $\alpha \in \mathrm{DC}_{d}$ (excluding a frequency subset depending on $V$), see also \cite{MR2100420}. Their method relies on Green's function estimates and multi-scale analysis, originally developed for random Schr\"odinger operators. Fr\"ohlich and Spencer \cite{frohlich1983absence} first systematically introduced the multiscale analysis method, laying the crucial foundation for rigorous proofs of Anderson localization. 
	
	Related localization results have been established for Schr\"{o}dinger operators on the strip \cite{MR1796713}, for the long-range operator \eqref{longrange} with $\cos 2\pi(\cdot)$ replaced by an arbitrary analytic function \cite{MR2100420}, for operators with power-law hopping and analytic cosine type potentials \cite{shi2024green}, for Schr\"odinger operators with Gevrey smooth potentials \cite{MR2108112,MR3291922}, for operators on $\mathbb{Z}^{d}$ with $d \geqslant 2$ \cite{MR2346272, MR1947458, MR4108613}, and for Schr\"odinger operators with interactions \cite{MR3925103} or with background potentials of low complexity \cite{MR4546503}.
	
	Note that the above multi-scale analysis approach generally requires removing an unknown set of frequencies to eliminate the dependence of the large deviation exceptional set on energy. Consequently, this method usually cannot be applied to any fixed $\alpha \in \mathrm{DC}_d$ directly. A major breakthrough came from \cite{Jitomirskaya1999, MR1298941, MR1328253}, where Jitomirskaya proved Anderson localization for the supercritical AMO, with $\lambda>1$ in the normalization of $\eqref{amo}$, for all $\alpha \in \mathrm{DC}_{1}$ and $\theta \in \Theta^{\tau}$, by introducing Lagrange interpolation polynomials to analyze the distribution of resonances. Her method was later extended to one-dimensional long-range operators \eqref{longrange}, where the Fourier coefficients $\widehat{V}_{m}$ decay exponentially \cite{MR1908056}. 
    
	In our setting, the Fourier coefficients $\widehat{V}_{m}$ satisfy a polynomial decay bound for $V \in C^{k}(\mathbb{T}^{d}, \mathbb{R})$, making it natural to consider a polynomial version of spectral localization.
	
	\begin{definition}\label{def1.2}
		The family $\{L_{V,\alpha,\theta}\}_{\theta\in\mathbb{T}}$ is said to exhibit \emph{$s$-polynomial spectral localization} ($s$-PSL), if $L_{V,\alpha,\theta}$ has pure point spectrum  for all $\theta\in\Theta$ where $|\mathbb{T}\setminus \Theta|=0$, and every eigenfunction satisfies
		\begin{equation*}
			|\psi(n)| \leqslant C (1+|n|)^{-s}
		\end{equation*}
		for some constant $C > 0$. Moreover, if there is a precise arithmetic description for $\Theta$, we say $\{L_{V,\alpha,\theta}\}_{\theta\in\mathbb{T}}$ has \emph{arithmetic $s$-PSL}. 
	\end{definition}
	
	Spectral localization concerns only the eigenfunctions and does not involve time evolution. To study the time evolution of the moments defined in \eqref{pmom}, we introduce the notion of \emph{dynamical localization}. One standard definition is that the operator has pure point spectrum and, for every compactly supported initial state and every $p > 0$,
	\begin{equation}\label{dl}
		\sup_{t \in \mathbb{R}} \sum_{n \in \mathbb{Z}^{d}} |n|^{p} |\langle \delta_{n}, \psi(t) \rangle|^{2} < \infty.
	\end{equation}
    For a quasi-periodic operator family $\{L_{V,\alpha,\theta}\}_{\theta \in \mathbb{T}}$, the bound in \eqref{dl} depends on the phase $\theta$. Moreover, dynamical localization is known to be sensitive to $\theta$. More precisely, even for the AMO, dynamical localization fails for a Baire generic set of phases $\theta \in \mathbb{T}$ \cite{JitoSimonSC}.
	
	Therefore, it is natural to consider dynamical localization \emph{in expectation} to phase $\theta$. More precisely, we say that the family of operators $\{L_{V,\alpha,\theta}\}_{\theta\in \mathbb{T}}$ exhibits \emph{exponential dynamical localization in expectation} if 
	\begin{equation*}
		\int_{\mathbb{T}} \sup_{t \in \mathbb{R}}|\langle \delta_{n}, e^{itL_{V,\alpha,\theta}} \delta_{m} \rangle| \, \mathrm{d}\theta \leqslant Ce^{-\gamma|m-n|}
	\end{equation*}
	for some constant $\gamma>0$. 
	
	For the AMO with $\lambda > 1$, it is known that \eqref{amo} exhibits exponential dynamical localization in expectation for $\alpha \in \mathrm{DC}_1$ \cite{MR3079334, MR4216568}, where the decay rate satisfies $\gamma = \log \lambda - \varepsilon$. Their proof relies on the sharp phase transition result established in \cite{MR4756946}. For the long-range family $\{L_{\lambda,V,\alpha,\theta}\}_{\theta \in \mathbb{T}}$, Ge-You-Zhou \cite{ge2019exponential} proved that exponential dynamical localization in expectation holds for $\alpha \in \mathrm{DC}_{d}$, provided $\widehat{V}_m$ decays exponentially and $\lambda$ is sufficiently large. 
	
	In the setting where $V \in C^{k}(\mathbb{T}^{d}, \mathbb{R})$, to quantify the decay rate of localization for \eqref{longrange}, we introduce the following definition.
	
	\begin{definition}\label{def1.1}
		The family $\{L_{V,\alpha,\theta}\}_{\theta\in\mathbb{T}}$ is said to exhibit \emph{$s$-polynomial dynamical localization in expectation} ($s$-PDL), if $L_{V,\alpha,\theta}$ has pure point spectrum for all $\theta \in \Theta$ with $|\mathbb{T} \setminus \Theta| = 0$ and there exists a constant $C > 0$ such that  
		\begin{equation}\label{ineq1}
		    \begin{split}
		        \int_{\mathbb{T}}\sup_{t \in \mathbb{R}}|\langle\delta_n, e^{-itL_{V,\alpha,\theta}}\delta_m\rangle|\mathrm{d}\theta 
                & =\int_{\Theta}\sup_{t \in \mathbb{R}}|\langle\delta_n, e^{-itL_{V,\alpha,\theta}}\delta_m\rangle|\mathrm{d}\theta  \\ 
		          & \leqslant \frac{C}{(1+|n-m|)^{s}}.
		    \end{split}
		\end{equation}
		Moreover, if $\Theta$ admits a precise arithmetic description, we say that $\{L_{V,\alpha,\theta}\}_{\theta \in \mathbb{T}}$ has \emph{arithmetic $s$-PDL}.
	\end{definition}
	
	Our first result establishes a criterion for arithmetic $s$-PSL and $s$-PDL for long-range operators on $\ell^2(\mathbb{Z}^{d})$ based on the reducibility of the associated Schr\"odinger cocycle. Define
	\begin{equation*}
		S_{E}^{V}(x) \coloneq \begin{pmatrix}
			E - V(x) & -1 \\
			1 & 0
		\end{pmatrix} \in \mathrm{SL}(2, \mathbb{R}).
	\end{equation*}
	We say that the Schr\"odinger cocycle $(\alpha, S_{E}^{V})$ is \emph{$C^{k}$-reducible} if there exist $B \in C^{k}(2\mathbb{T}^{d}, \mathrm{SL}(2, \mathbb{R}))$ and a constant matrix $A \in \mathrm{SL}(2, \mathbb{R})$ such that
	\begin{equation*}
		B(x + \alpha)^{-1} S_{E}^{V}(x) B(x) = A.
	\end{equation*}
    If, in addition, the conjugacy and the constant normal form satisfy uniform quantitative estimates specified in both Theorem~\ref{APSL} and Theorem~\ref{pdl}, we say that the Schr\"odinger cocycle is \emph{quantitatively $C^{k}$-reducible}. These two theorems give the precise quantitative assumptions required for arithmetic polynomial spectral localization and arithmetic polynomial dynamical localization, respectively. Let $\rho(E)$ denote the rotation number of the cocycle $(\alpha, S_{E}^{V})$, see Sect.\ref{sect2.2} for the precise definition.
	
	We now state the main result in a concise form, the detailed version can be found in Theorem~\ref{APSL} and Theorem~\ref{pdl} respectively.  
	
	\begin{theorem}\label{main1}
		Let $\alpha \in {\rm DC}_d(\kappa,\tau)$ and $k>d$. Suppose that the Schr\"odinger cocycle $(\alpha, S_{E}^{V})$ is quantitatively $C^{k}$-reducible for all $E \in \Sigma_{V,\alpha}$ with $\rho(E) \in \Theta^{2\tau}$. Then 
		\begin{enumerate}[label=(\Alph*)]
			\item \label{item:psl} The operator $L_{V, \alpha, \theta}$ exhibits arithmetic $k$-PSL for every $\theta \in \Theta^{\tau}$. 
			\item \label{item:pdl} The family $\{L_{V, \alpha, \theta}\}_{\theta \in \mathbb{T}}$ exhibits arithmetic $k$-PDL.
		\end{enumerate}
	\end{theorem}
	
    Regarding spectral localization, Theorem~\ref{main1}\ref{item:psl} establishes $k$-PSL for every $\theta \in \Theta^{\tau}$ under the stated quantitative reducibility assumptions. This gives an explicit arithmetic full-measure set of localization phases, whereas \cite{MR4482246} establishes polynomial spectral localization for almost every phase. 
    
    Regarding dynamical localization, Theorem~\ref{main1}\ref{item:pdl} yields arithmetic $k$-PDL with the phase set $\Theta^{\tau}$, providing the arithmetic description not obtained in \cite{shan2025dynamical}, as noted in \cite[Remark~1.2]{shan2025dynamical}. For sufficiently small $C^{k'}$ potentials, our reducibility estimates allow $k=\lfloor\eta k'\rfloor$ for every fixed $0<\eta<1/4$ and sufficiently large $k'$. Thus, the proportion of regularity retained can be arbitrarily close to $1/4$, compared with $1/400$ in \cite{shan2025dynamical}. Our dynamical localization criterion then yields $k$-PDL without further loss of this remaining exponent.

    For sufficiently small potentials of high finite regularity, we establish the quantitative reducibility estimates required by the preceding theorem and obtain the following localization result.

    \begin{Corollary}\label{main2}
        Let $\alpha \in \mathrm{DC}_{d}(\kappa,\tau)$ and let $k \in \mathbb{N}$ satisfy $k>d$. There exists an integer $k'=k'(\tau,d,k)>4k$ and $\varepsilon'=\varepsilon'(\kappa,\tau,d,k)>0$ such that, if $V \in C^{k'}(\mathbb{T}^{d},\mathbb{R})$ and $\|V\|_{k'} \leqslant \varepsilon'$, then $L_{V,\alpha,\theta}$ has arithmetic $k$-PSL for every $\theta \in \Theta^{\tau}$, and the family $\{L_{V,\alpha,\theta}\}_{\theta \in \mathbb{T}}$ has arithmetic $k$-PDL.
    \end{Corollary}
    \begin{remark}
        Note that the restriction to a full-measure arithmetic subset of phases is, in general, unavoidable, since exceptional phases may exhibit singular continuous spectrum. For more quantitative growth estimates of \eqref{dl} that are uniform in $\theta\in\mathbb{T}$, see \cite{MR4288185,JPo,prema, liu2026quantum, MR4564259,  MR4604835}. 
    \end{remark}
    \begin{remark}
        The present paper focuses on arithmetic PSL and PDL in expectation for the long-range family $\{L_{V,\alpha,\theta}\}_{\theta \in \mathbb{T}}$. Strong ballistic transport for the Aubry dual quasi-periodic Schr\"odinger family $\{H_{V,\alpha,x}\}_{x \in \mathbb{T}^{d}}$ will be studied in a forthcoming paper, using the quantitative reducibility and localization estimates obtained here. 
    \end{remark}

	\subsection{Novelty of the proof}
	Let us first focus on arithmetic polynomial spectral localization. As previously mentioned, the proofs in \cite{MR2100420, MR1815703, MR1796713} are based on multi-scale analysis, which establishes localization but lacks an arithmetic characterization of the frequency $\alpha$. In contrast, the results in \cite{MR1908056, Jitomirskaya1999} provide arithmetic localization for both frequency $\alpha$ and phase $\theta$, but are confined to operators on $\ell^2(\mathbb{Z}^1)$.
	
	Recently, a new approach to proving localization via Aubry duality and reducibility was developed in \cite{avila2017sharp,  ge2020arithmetic,MR3512893}. A key advantage of this method is its robustness: it remains effective across different dimensions and does not critically depend on the smoothness of the potential, see also the survey \cite{MR3966844}. As a result, one can establish arithmetic localization for both higher-dimensional operators and potentials with low regularity.
	
	More precisely, Aubry duality links the reducibility energies $E$ of the Schr\"odinger cocycle to the phases of the dual long-range operator through the rotation number $\rho(E)$. The Fourier coefficients of the conjugating transformation $B$ yield a sufficiently ``good'' eigenfunction (see Definition \ref{def3.1}) for $L_{V,\alpha,\rho(E)}$, with decay rate $|n|^{-k}$. By applying shifts to this eigenfunction, one constructs a family of eigenfunctions. Following the strategy of \cite{ge2020arithmetic}, we introduce an auxiliary measure to demonstrate that this family forms a basis of $\ell^{2}(\mathbb{Z}^{d})$. This establishes Theorem \ref{main1}\ref{item:psl}.
	
	Regarding polynomial dynamical localization, we employ the eigenfunction expansion to reduce the problem to quantitative estimates for the eigenfunctions of $L_{V,\alpha,\rho(E)} u_E = E u_E$. Since uniform localization of eigenfunctions across the spectrum does not hold \cite{jitomirskaya2024sharp}, we decompose $\Theta^{\tau}$ into disjoint subsets $\Theta^{\tau}=\cup_i \Theta_i$ and establish stratified quantitative estimates for eigenfunctions corresponding to $E$ with $\rho(E) \in \Theta_i$. The two-center estimate in Theorem~\ref{goodeig}, together with the quantitative bound \eqref{A2}, balances the possible growth of the localization center. Combining this estimate with the measure bound for the arithmetic strata reduces the proof to the convergence of a series of the form $\sum_{i \geqslant 1} \gamma_i^{1-c}$, where $c<1$. Consequently, the full polynomial decay exponent $k$ is retained in the dynamical localization estimate, which proves Theorem~\ref{main1}\ref{item:pdl}. 

    The paper is organized as follows. Section~2 collects the preliminaries. Section~3 constructs good eigenfunctions from quantitative reducibility. Sections~4 and~5 establish the criteria for arithmetic polynomial spectral localization and polynomial dynamical localization in expectation, respectively. Section~6 proves quantitative $C^{k}$-reducibility and applies these criteria to finitely differentiable potentials. Appendices~A and~B contain the continuity argument and the analytic KAM estimates, respectively.

\section{Preliminaries}
	Recall that $\mathrm{sl}(2,\mathbb{R})$ is the set of $2\times2$ matrices of the form 
    \begin{equation*}
        \begin{pmatrix}
            x & y-z \\
            y+z & -x
        \end{pmatrix}
    \end{equation*}
    where $x,y,z \in \mathbb{R}$. It is well-known that $\mathrm{sl}(2,\mathbb{R})$ is isomorphic to $\mathrm{su}(1,1)$, the set of matrices of the form
    \begin{equation*}
        \begin{pmatrix}
            it & \nu \\
            \bar{\nu} & -it
        \end{pmatrix}
    \end{equation*}
	where $t\in\mathbb{R}, \nu \in \mathbb{C}$. The isomorphism between $\mathrm{sl}(2,\mathbb{R})$ and $\mathrm{su}(1,1)$ is given by $A \rightarrow MAM^{-1}$ where 
    \begin{equation*}
        M=\frac{1+i}{2}\begin{pmatrix}
            1 & i \\ 
            1 & -i
        \end{pmatrix} \in \mathrm{SL}(2,\mathbb{C}).
    \end{equation*}
	Direct calculation shows that 
    \begin{equation*}
        M\begin{pmatrix}
            x & y-z \\
            y+z & -x
        \end{pmatrix}M^{-1}=\begin{pmatrix}
            iz & x+iy \\
            x-iy & -iz
        \end{pmatrix}.
    \end{equation*}

\subsection{Conjugation and Reducibility}

Let
\begin{equation*}
    \mathcal{S}_h \coloneq \{x = (x_1, \dots, x_d) \in \mathbb{C}^d : |\Im x_i| < h, \ \text{for all } 1 \leqslant i \leqslant d\}.
\end{equation*}
For any bounded analytic (possibly matrix-valued) function $F(x)$ defined on $\mathcal{S}_h$, define the norm
\begin{equation*}
    |F|_h \coloneq \sup_{x \in \mathcal{S}_h} \|F(x)\|.
\end{equation*}
We denote by $C_h^\omega(\mathbb{T}^d,\mathbb{R})$ the space of real analytic functions on $\mathbb{T}^d$ admitting a bounded holomorphic extension to $\mathcal{S}_h$ and taking real values on $\mathbb{T}^d$. Similarly, $C_h^\omega(\mathbb{T}^d,\mathrm{sl}(2,\mathbb{R}))$ denotes the space of real analytic maps from $\mathbb{T}^d$ to $\mathrm{sl}(2,\mathbb{R})$ admitting a bounded holomorphic extension to $\mathcal{S}_h$ with values in $\mathrm{sl}(2,\mathbb{C})$. Finally, $C_h^\omega(\mathbb{T}^d,\mathrm{SL}(2,\mathbb{R}))$ denotes the space of real analytic maps from $\mathbb{T}^d$ to $\mathrm{SL}(2,\mathbb{R})$ admitting a bounded holomorphic extension to $\mathcal{S}_h$ with values in $\mathrm{SL}(2,\mathbb{C})$. We also define
\begin{equation*}
    C^\omega(\mathbb{T}^d,*) \coloneq \bigcup_{h>0}C_h^\omega(\mathbb{T}^d,*).
\end{equation*}
Let $C^k(\mathbb{T}^d,*)$ denote the space of $k$-times continuously differentiable $*$-valued functions. Define the $C^{k}$-norm
\begin{equation*}
    \|F\|_k \coloneq \sup_{|\beta| \leqslant k,\ x \in \mathbb{T}^d} \|\partial^{\beta} F(x)\|.
\end{equation*}   
By standard Cauchy estimate, for any $F \in C^{\omega}_{h}(\mathbb{T}^{d},*)$, 
\begin{equation}\label{norm}
    \|F\|_{k} \leqslant \frac{C(k,d)}{(\min\{h,1\})^k}|F|_{h}. 
\end{equation}

For $F \in C^k(\mathbb{T}^d, *)$ with $k \geqslant 0$, the Fourier coefficients of $F$ are defined by
\begin{equation*}
    \widehat{F}(n) \coloneq \int_{\mathbb{T}^d} F(x) e^{-2\pi i \langle n, x \rangle} \, \mathrm{d}x, \quad \forall n \in \mathbb{Z}^d.
\end{equation*}
By integrating by parts, for any $n \in \mathbb{Z}^{d}$,
\begin{equation}\label{decay}
    \|\widehat{F}(n)\| \leqslant C(k,d)\|F\|_k (1+|n|)^{-k}.
\end{equation}
Indeed, for $n \neq 0$, choosing $1 \leqslant j \leqslant d$ such that $|n_j| \geqslant |n|/d$ and integrating by parts $k$ times in the $x_j$ variable gives \eqref{decay}. 
	
For two cocycles $(\alpha,A_1)$, $(\alpha,A_2)\in \mathbb{T}^d  \times C^{\ast}(\mathbb{T}^d,\mathrm{SL}(2,\mathbb{R}))$ ($\ast$ represents $\omega$ or $k$),  we can say that they are $C^{*}$ conjugated if there exists $Z \in C^{*}(2\mathbb{T}^d,  \mathrm{SL}(2,\mathbb{R}))$, such that
    \begin{equation*}
        Z^{-1}(x+\alpha)A_1(x)Z(x)=A_2(x).
    \end{equation*}
Notably, we allow conjugacies on $2\mathbb{T}^d=\mathbb{R}^d/(2\mathbb{Z})^d$. 
	
We say that a finitely differentiable cocycle $(\alpha,A) \in \mathbb{T}^d \times C^k(\mathbb{T}^d,\mathrm{SL}(2,\mathbb{R}))$ is $C^{k_1}$(or $C^{k,k_1}$) almost reducible if there exist $B_j \in C^{k_1}(2\mathbb{T}^d,\mathrm{SL}(2,\mathbb{R}))$ and a constant matrix $\bar{A} \in \mathrm{SL}(2,\mathbb{R})$ such that
\begin{equation*}
    B_j^{-1}(x+\alpha)A(x)B_j(x) \to \bar{A}
\end{equation*}
in the $C^{k_1}$ topology. We say that $(\alpha,A)$ is $C^{k_1}$(or $C^{k,k_1}$) reducible if there exist $B \in C^{k_1}(2\mathbb{T}^d,\mathrm{SL}(2,\mathbb{R}))$ and $\bar{A} \in \mathrm{SL}(2,\mathbb{R})$ such that
\begin{equation*}
    B^{-1}(x+\alpha)A(x)B(x)=\bar{A}.
\end{equation*}

\subsection{Rotation number and degree}\label{sect2.2}
	
Throughout this subsection, we assume that the base rotation $x \mapsto x+\alpha$ is uniquely ergodic, this is the case, in particular, when $\alpha$ is rationally independent.

Suppose that $A \in C^0(\mathbb{T}^d,\mathrm{SL}(2,\mathbb{R}))$ is homotopic to identity. Then the projective skew-product $F_A:\mathbb{T}^d \times \mathbb{S}^1 \rightarrow \mathbb{T}^d \times \mathbb{S}^1$ defined by 
\begin{equation*}
	F_A(x,\omega) \coloneq \bigg(x+\alpha, \frac{A(x)\cdot \omega}{|A(x) \cdot \omega|}\bigg),
\end{equation*}
is also homotopic to the identity. Identifying $\mathbb{S}^1$ of the form $\mathbb{R}/\mathbb{Z}$, we choose a lift $\widetilde{F}_A:\mathbb{T}^d\times \mathbb{R}\rightarrow \mathbb{T}^d\times \mathbb{R}$ with  $\widetilde{F}_A(x,y)=(x+\alpha,y+\psi(x,y))$, where $\psi:\mathbb{T}^d \times \mathbb{R} \rightarrow \mathbb{R}$ is continuous and $\mathbb{Z}$-periodic in $y$. The function $\psi$ is called a lift of $A$. 

By its $\mathbb{Z}$-periodicity in $y$, $\psi$ induces a continuous function on $\mathbb{T}^d \times \mathbb{S}^1$, which is still denoted by $\psi$. Let $\mu$ be any $F_A$-invariant probability measure on $\mathbb{T}^d \times \mathbb{S}^1$. Since the projection of $\mu$ onto the first coordinate is invariant under the base rotation $x \mapsto x+\alpha$, the unique ergodicity of the base rotation implies that this projection is the Lebesgue measure. The number
\begin{equation}\label{rot1}
    \rho(\alpha,A) \coloneq \int_{\mathbb{T}^d \times \mathbb{S}^1}\psi(x,\omega) \, \mathrm{d} \mu(x,\omega) \bmod \mathbb{Z}
\end{equation}
is independent of the choices of the lift $\psi$ and the invariant measure $\mu$. It is called the \emph{fibered rotation number} of the cocycle $(\alpha,A)$, see \cite{johnson1982rotation} for more details.
   
	
For $(\alpha, Ae^{f}) \in C^{0}(\mathbb{T}^{d}, \mathrm{SL}(2, \mathbb{R}))$ homotopic to the identity, where $A \in \mathrm{SL}(2, \mathbb{R})$ is a constant matrix and $f \in C^0(\mathbb{T}^d, \mathrm{sl}(2, \mathbb{R}))$ is sufficiently small, we have
\begin{equation}\label{con}
    \|\rho(\alpha, Ae^{f})-\rho(\alpha, A)\|_{\mathbb{R}/\mathbb{Z}} \leqslant \tilde{c}\|f\|_0^{1/2},
\end{equation}
where $\tilde{c}=\tilde{c}(\|A\|)$.
    
Let 
\begin{equation*}
    R_{\phi} \coloneq \begin{pmatrix} 
    \cos 2\pi\phi & -\sin 2\pi\phi \\ 
    \sin 2\pi\phi & \cos 2\pi\phi 
    \end{pmatrix}.
\end{equation*}
Let $B \in C^0(2\mathbb{T}^d,\mathrm{SL}(2,\mathbb{R}))$. If $B$ is homotopic on $2\mathbb{T}^d$ to $x \mapsto R_{\langle n,x\rangle/2}$ for some $n \in \mathbb{Z}^d$, then we call $n$ the \emph{degree} of $B$ and denote it by $\deg B$. In particular, if $B$ is $\mathbb{T}^d$-periodic and is homotopic on $\mathbb{T}^d$ to $x \mapsto R_{\langle n,x \rangle}$, then its degree as a map on $2\mathbb{T}^d$ is $2n$. Furthermore, 
\begin{equation}\label{deg1}
	\deg(AB)=\deg A+\deg B.
\end{equation}

Throughout the paper, all reducibility conjugacies defined on $2\mathbb{T}^{d}$ are understood as admissible lifts of $\mathrm{PSL}(2,\mathbb{R})$-valued conjugacies. More precisely, if $B \in C^{k}(2\mathbb{T}^{d},\mathrm{SL}(2,\mathbb{R}))$ and $\deg B=n \in \mathbb{Z}^{d}$, then
\begin{equation}\label{admissibleparity}
	B(x+l)=(-1)^{\langle n,l \rangle}B(x), \quad l \in \mathbb{Z}^{d}.
\end{equation}
The same convention applies to complex conjugacies obtained from real conjugacies by constant complex changes of coordinates, and their degrees are inherited from the corresponding real conjugacies. All conjugacies produced by the KAM scheme below satisfy \eqref{admissibleparity}.
    
Note that the fibered rotation number remains invariant under real conjugacies that are homotopic to the identity map. Generally speaking, when the cocycle $(\alpha,A_1)$ is conjugated to $(\alpha,A_2)$ by $B \in C^0(2\mathbb{T}^d, \mathrm{SL}(2,\mathbb{R}))$, i.e. $B^{-1}(x +\alpha)A_1(x)B(x)=A_2(x)$, we have
\begin{equation}\label{rot2}
	\rho(\alpha,A_2)=\rho(\alpha,A_1)-\frac{\langle \deg B,\alpha\rangle}{2} \bmod \mathbb{Z}.
\end{equation}
	
For a Schr\"odinger cocycle, we write
\begin{equation*}
    S_E^V(x) \coloneq \begin{pmatrix}
                        E-V(x) & -1 \\
                        1 & 0
                      \end{pmatrix}.
\end{equation*}
In the following text, we abbreviate $\rho(\alpha,S_{E}^{V})=\rho(E)$ and $\rho(\alpha,A)=\rho(A)$. When $\rho(A)$ appears in absolute value, we take its representative in $[-1/2,1/2]$.

\begin{remark}\label{normalformconvention}
    Throughout the paper, normal-form statements and calculations are written for constant matrix $A$ with nonnegative trace, using the principal logarithm. For $\operatorname{tr}A \geqslant 0$, $|\rho(A)|\in[0,1/4]$ and $\|2\rho(A)\|_{\mathbb{R}/\mathbb{Z}}=2|\rho(A)|$. The negative-trace case is included by applying the same calculations and absolute rotation-number estimates to $-A$ and restoring the negative sign in the conjugacy equation, with $\rho(-A)=\rho(A)+\frac{1}{2}\bmod\mathbb{Z}$.
\end{remark}

\subsection{Hyperbolicity and integrated density of states}
For a cocycle $(\alpha,A)$, define
\begin{equation*}
    \begin{split}
        & A_n(x)=A(x+(n-1)\alpha) \cdots A(x), \quad n \geqslant 1, \\ 
        & A_{-n}(x)=A_n(x-n\alpha)^{-1}, \quad n \geqslant 1.
    \end{split}
\end{equation*}
We call the cocycle $(\alpha,A)$ {\it uniformly hyperbolic} if for every $x \in \mathbb{T}^d$, there exists a continuous decomposition $\mathbb{C}^2=E^s(x) \oplus E^u(x)$ such that for some constants $C>0$, $c>0$ and every $n \geqslant 0$,
	\begin{equation*}
	    \begin{split}
	        & |A_n(x)v| \leqslant Ce^{-cn}|v|, \quad v \in E^s(x), \\ 
            & |A_{-n}(x)v| \leqslant Ce^{-cn}|v|, \quad v \in E^u(x). 
	    \end{split}
	\end{equation*}
	This decomposition is invariant by the dynamics, which means that for any $x \in \mathbb{T}^d$, $A(x)E^*(x) = E^*(x + \alpha)$, for $*= s, u$. In the $C^0$ topology, the set of uniformly hyperbolic cocycles is an open set. Specifically, in the case of quasi-periodic Schr\"{o}dinger operators, the cocycle $(\alpha, S_E^V)$ is uniformly hyperbolic if and only if $E \notin \Sigma_{V,\alpha}$, or in other words, if the energy lies within a spectral gap \cite{johnson1986exponential}.
	
Let's consider the Schr\"{o}dinger operators
\begin{equation*}
    (H_{V,\alpha,x}u)(n)=u(n+1)+u(n-1)+V(x+n\alpha)u(n), \quad n \in \mathbb{Z}.
\end{equation*}
The integrated density of states (IDS) is the function $N_{V,\alpha} \colon \mathbb{R} \rightarrow [0,1]$ defined by 
	\begin{equation*}
	    N_{V,\alpha}(E)=\int_{\mathbb{T}^d}\mu_{V,\alpha,x}(-\infty,E]\mathrm{d}x, 
	\end{equation*}
where $\mu_{V,\alpha,x}=\mu^{\delta_0}_{V,\alpha,x}$ is the spectral measure of $H_{V,\alpha,x}$. 
	
	Choose $\rho(E) \in [0,1/2]$, it is well known that 
	\begin{equation}\label{ids}
		N_{V,\alpha}(E)=1-2\rho(E). 
	\end{equation}
	
	For any Borel set $\Delta \subseteq \mathbb{R}$, we also define 
	\begin{equation}\label{nv}
		N(\Delta)=\{N_{V,\alpha}(E) \colon E \in \Delta\}, \ \rho(\Delta)=\{\rho(E) \colon E \in \Delta\}. 
	\end{equation}

	\subsection{Analytic approximation}\label{AA}
	Let $f \in C^{k}(\mathbb{T}^{d},\mathrm{sl}(2,\mathbb{R}))$. By Zehnder \cite{zehnder1975generalized}, there is a sequence $\{f_{j}\}_{j \geqslant 1}$, $f_{j} \in C_{\frac{1}{j}}^{\omega}(\mathbb{T}^{d},\mathrm{sl}(2,\mathbb{R}))$ and a constant $C'(k,d)>0$, such that
    \begin{equation}\label{aa}
        \begin{split}
            &\| f_{j}-f \|_{k} \xrightarrow{j \rightarrow \infty} 0, \\
            &|f_{j}|_{\frac{1}{j}} \leqslant C' \|f\|_{k}, \\
            &|f_{j+1}-f_{j}|_{\frac{1}{j+1}} \leqslant C'j^{-k}\|f\|_{k}.
        \end{split}
    \end{equation}
	Furthermore, if $k \leqslant k'$ and $f \in C^{k'}$, the properties \eqref{aa} still hold with $k'$ instead of $k$. This implies that the sequence can be constructed from $f$ irrespective of its regularity (since $f_{j}$ is achieved by convolving $f$ with a map that does not depend on $k$).

\subsection{Aubry duality}

Assume that, for some $E \in \mathbb{R}$, the quasi-periodic Schr\"odinger equation
\begin{equation*}
    H_{V,\alpha,x}u=Eu,
\end{equation*}
admits a quasi-periodic Bloch wave of the form $u(n)=e^{2\pi i n\theta}\psi(x+n\alpha)$ for some $\psi \in C^*(\mathbb{T}^d, \mathbb{C})$ (where $*$ denotes $\omega$ or $k$) and $\theta \in \mathbb{T}$. Let $\hat{\psi}(m)$ denote the $m$-th Fourier coefficient of $\psi$. Then it is straightforward to verify that the sequence $\{\hat{\psi}(m)\}_{m \in \mathbb{Z}^d}$ is an eigenfunction of the following long-range operator:
\begin{equation*}
    (L_{V,\alpha,\theta} \hat{\psi})(n)=\sum_{m \in \mathbb{Z}^d} \widehat{V}_m \hat{\psi}(n-m)+2\cos2\pi(\theta+\langle n, \alpha \rangle) \hat{\psi}(n), \quad n \in \mathbb{Z}^d.
\end{equation*}
We refer to $L_{V,\alpha,\theta}$ as the \emph{dual operator} of $H_{V,\alpha,x}$.

\section{Good eigenfunctions of dual operators}
In this section, we construct good eigenfunctions for the long-range operator $L_{V, \alpha, \rho(E)}$ under the assumption of quantitative $C^k$-reducibility for Schr\"odinger cocycle $(\alpha, S_{E}^{V})$. Throughout the text, we always assume that the parameters satisfy $\kappa, \gamma \in (0,1)$ and $\tau>d$.

We begin by introducing the definition of a good eigenfunction.

\begin{definition}\label{def3.1}
    A normalized eigenfunction $u \in \ell^2(\mathbb{Z}^d)$ is called \emph{$(s, n^*, C, C_{|n^*|})$-good} if there exist $s \in \mathbb{Z}^+$, $C > 0$, $n^* \in \mathbb{Z}^d$, and $C_{|n^*|} \in [0,1]$ such that
    \begin{equation*}
        |u(n)| \leqslant C((1+|n|)^{-s}+C_{|n^*|}(1+|n+n^*|)^{-s}), \quad \forall n \in \mathbb{Z}^d.
    \end{equation*}
\end{definition}	

\begin{theorem}\label{goodeig}
    Let $\alpha \in {\mathrm{DC}}_d(\kappa, \tau)$. Suppose that $(\alpha, S_{E}^{V})$ is $C^k$-reducible for $E \in \Sigma_{V,\alpha}$ with $\rho(E) \in \Theta^{\tau}_{\gamma}$. That is, there exist $B \in C^k(2\mathbb{T}^d, \mathrm{SL}(2,\mathbb{R}))$ and $\bar{A} \in \mathrm{SL}(2,\mathbb{R})$ such that
    \begin{equation}\label{red}
        B^{-1}(x+\alpha)S_E^V(x)B(x)=\bar{A}
        =M^{-1}\exp\begin{pmatrix}
        it & \nu \\
        \bar{\nu} & -it
        \end{pmatrix}M.
    \end{equation}
    Moreover, suppose that there exist $\widetilde{B} \in C^k(2\mathbb{T}^d, \mathrm{SL}(2,\mathbb{R}))$, 
    $Y \in C^k(\mathbb{T}^d, \mathrm{sl}(2,\mathbb{R}))$, and $n^* \in \mathbb{Z}^d$ such that
    \begin{equation}\label{BRY}
        B(x)=\widetilde{B}(x) \, R_{\frac{\langle n^*, x \rangle}{2}} \, e^{Y(x)},
    \end{equation}
    with the following estimates:
    \begin{equation}\label{est}
        \begin{split}
            & |n^{*}| \leqslant \gamma^{-c_{1}}, \quad \|Y\|_k \leqslant (1+|n^{*}|)^{-c_{2}}, \\
            & \|\widetilde{B}\|_{0}\|\widetilde{B}\|_{k} \leqslant \gamma^{-c_3}, \quad |\deg \widetilde{B}| \leqslant \gamma^{-c_4}, \\
            & |\nu| \leqslant (1+|n^*|)^{-c_5}, \, \text{if }n^{*} \neq 0, \quad |\rho(\bar{A})| \geqslant \frac{\gamma }{16^{\tau}(1+|n^*|)^{\tau}}. 
        \end{split}
    \end{equation}
    where $c_{i}>0,i=1,2,3,4,5$, are fixed constants independent of $E$ and $\gamma$. Then the following hold.  
    \begin{enumerate}[label=(\Alph*)]
        \item \label{item:res2}
        $L_{V, \alpha, \rho(E)}$ admits a $(k, n^*, C, C_{|n^*|})$-good eigenfunction $u_{E}$, where
        \begin{equation}\label{CC}
            \begin{split}
                & C \leqslant C_1(k,d)\gamma^{-(c_{3}+c_{4}k)}, \\
                & C_{|n^*|} \leqslant \min \bigg\{1, \frac{100^{\tau}}{\gamma} (1+|n^{*}|)^{-\min(c_{2}, c_{5}-\tau)}\bigg\}.
            \end{split}
        \end{equation}
        Moreover, if 
        \begin{equation}\label{cond}
            c_{5}>\tau \quad \text{and}\quad c \coloneq \frac{k}{\min\{c_2, c_5-\tau\}}+2(c_3+c_4k)<1,
        \end{equation}
        then
        \begin{equation}\label{A2}
            C^{2}(1+C_{|n^{*}|}(1+|n^{*}|)^{k}) \lesssim_{k,d,\tau}\gamma^{-c}.
        \end{equation}
        
        \item \label{item:res1} In particular, $u_{E}$ satisfies 
        \begin{equation}\label{decayuE}
            |u_{E}(n)| \leqslant C_2(k,d)\gamma^{-(c_{3}+c_{4}k+c_{1}k)} (1+|n|)^{-k}, \quad n\in \mathbb{Z}^d. 
        \end{equation}   
    \end{enumerate}
\end{theorem}

\begin{remark}
    In Theorem~\ref{goodeig}, the vector $n^* \in \mathbb{Z}^{d}$ corresponds to the resonant site associated with the final step in reducibility. For further details, see Theorem~\ref{Ckreducibility}. 
\end{remark}
	
\begin{proof} 
By \eqref{rot2}, \eqref{red} and \eqref{BRY}, we have
\begin{equation}\label{rot3}
    \rho(\bar{A})=\rho(E)-\frac{\langle n_0+n^*, \alpha \rangle}{2} \bmod \mathbb{Z},
\end{equation}
where $n_0=\deg \widetilde{B}$. 

Since $E \in \Sigma_{V,\alpha}$, the constant cocycle $(\alpha,\bar{A})$ is not uniformly hyperbolic. The last estimate in \eqref{est} excludes the parabolic cases. Hence $\bar{A}$ is elliptic. The principal elliptic logarithm specified in Remark~\ref{normalformconvention} satisfies $\rho(\bar{A})t>0$.

We choose unitary matrix $U \in \mathrm{SL}(2,\mathbb{C})$ such that
    \begin{equation*}
        U M\bar{A} M^{-1}U^{-1}
        =\begin{pmatrix}
            e^{2\pi i\rho(\bar{A})} & c \\
            0 & e^{-2\pi i\rho(\bar{A})}
        \end{pmatrix}.
    \end{equation*}
Define $B_1(x) \coloneq B(x)M^{-1}U^{-1} \in C^{k}(2\mathbb{T}^d, \mathrm{SL}(2,\mathbb{C}))$, then
\begin{equation}\label{expansion}
    S_E^V(x) \, B_1(x)=B_1(x+\alpha)
    \begin{pmatrix}
        e^{2\pi i\rho(\bar{A})} & c \\
        0 & e^{-2\pi i\rho(\bar{A})}
    \end{pmatrix}.
\end{equation}
Note that $B_1$ is an admissible complex conjugacy induced by the real conjugacy $B$, and we define its inherited degree by $\deg B_1 \coloneq \deg B$. 

Let
\begin{equation*}
    B_1(x)=\begin{pmatrix} 
        b_{11}(x) & b_{12}(x) \\ 
        b_{21}(x) & b_{22}(x) 
    \end{pmatrix}.
\end{equation*}
By expanding \eqref{expansion}, we obtain the following system:
\begin{equation*}
    \begin{cases}
        (E-V(x)) \, b_{11}(x) = b_{21}(x)+b_{11}(x+\alpha) \, e^{2\pi i \rho(\bar{A})}, \\
        b_{11}(x)=b_{21}(x+\alpha) \, e^{2\pi i \rho(\bar{A})}.
    \end{cases}
\end{equation*}
Substituting the second equation into the first yields
\begin{equation}\label{dualfun}
    (E-V(x))b_{11}(x)=b_{11}(x-\alpha)e^{-2\pi i \rho(\bar{A})}+b_{11}(x+\alpha)e^{2\pi i \rho(\bar{A})}.
\end{equation}

Define
\begin{equation}\label{z11}
    z_{11}(x) \coloneq e^{-2\pi i \langle n_0+n^*, x \rangle/2} b_{11}(x).
\end{equation}
Applying \eqref{rot3} and taking the Fourier transform on both sides of \eqref{dualfun}, we obtain
\begin{equation}\label{eigenz}
    \sum_{m \in \mathbb{Z}^d} \widehat{V}_{m} \hat{z}_{11}(n-m)+2\cos 2\pi(\rho(E)+\langle n, \alpha\rangle)\hat{z}_{11}(n)=E\hat{z}_{11}(n).
\end{equation}
By the standing admissibility convention \eqref{admissibleparity} and $\deg B_{1}=n_{0}+n^{*}$, 
one has
\begin{equation*}
	\begin{split}
		z_{11}(x+l) & =e^{-\pi i\langle n_{0}+n^{*},x+l\rangle}b_{11}(x+l) \\
		& =e^{-\pi i\langle n_{0}+n^{*},x\rangle}(-1)^{\langle n_{0}+n^{*},l\rangle}(-1)^{\langle n_{0}+n^{*},l\rangle}b_{11}(x)=z_{11}(x).
	\end{split}
\end{equation*}
Therefore, $z_{11}$ is $\mathbb{T}^{d}$-periodic, and its Fourier coefficients are indexed by $\mathbb{Z}^{d}$.

In the following, we analyze the structure of $\hat{z}_{11}(n)$ and establish decay estimates for $|\hat{z}_{11}(n)|$. For convenience, we denote
\begin{equation*}
    \begin{split}
        & \widetilde{B}(x)M^{-1}= \begin{pmatrix} 
            B_{11}(x) & B_{12}(x) \\ 
            B_{21}(x) & B_{22}(x) 
        \end{pmatrix}, \\
        & Me^{Y(x)}M^{-1}= \begin{pmatrix} 
            Y_{11}(x) & Y_{12}(x) \\ 
            Y_{21}(x) & Y_{22}(x) 
        \end{pmatrix}, \quad 
        U^{-1}= \begin{pmatrix} 
            U_{11} & U_{12} \\ 
            U_{21} & U_{22} 
        \end{pmatrix}.
    \end{split}
\end{equation*}
Then $B_1(x)$ can be rewritten as
\begin{equation}\label{B1x}
    B_1(x) = \widetilde{B}(x)M^{-1} \cdot M R_{\langle n^*, x \rangle/2} M^{-1} \cdot M e^{Y(x)} M^{-1} \cdot U^{-1}.
\end{equation}
Since
    \begin{equation*}
        MR_{\langle n^*,x\rangle/2}M^{-1}
        =\begin{pmatrix}
            e^{\pi i\langle n^*,x\rangle} & 0 \\
            0 & e^{-\pi i\langle n^*,x\rangle}
        \end{pmatrix}.
    \end{equation*}
A direct computation gives  
\begin{equation*}
    \begin{split}
        z_{11}(x) &= e^{-2\pi i \langle n_0 + n^*, x \rangle/2} b_{11}(x) \\
        &= (U_{11} Y_{11}(x) + U_{21} Y_{12}(x)) B_{11}(x) e^{-\pi i \langle n_0, x \rangle} \\
        &\quad\ + (U_{11} Y_{21}(x) + U_{21} Y_{22}(x)) B_{12}(x) e^{-\pi i \langle n_0, x \rangle} e^{-2\pi i \langle n^*, x \rangle}. 
    \end{split}
\end{equation*}
By \eqref{decay}, \eqref{est}, and \eqref{B1x} we have 
\begin{equation*}
    \begin{split}
        |\hat{z}_{11}(n)| &\lesssim_{k,d}  \|U\| \|\widetilde{B}\|_{k}\|e^{-\pi i \langle n_0, x \rangle}\|_{k} (1+|n|)^{-k}\\
        &\quad\ + \|U_{11} Y_{21} + U_{21} Y_{22} \|_{k} \|\widetilde{B}\|_{k} \|e^{-\pi i \langle n_0, x \rangle}\|_{k} (1+|n+n^{*}|)^{-k} \\
        &\lesssim_{k,d} \|U\| \|\widetilde{B}\|_{k}(1+|n_{0}|)^{k} (1+|n|)^{-k} \\
        &\quad\ + \|U_{11} Y_{21} + U_{21} Y_{22} \|_{k} \|\widetilde{B}\|_{k} (1+| n_{0}|)^{k}(1+|n+n^{*}|)^{-k},
    \end{split}
\end{equation*}

Now we can construct the normalized eigenfunctions. Let
\begin{equation*}
    \mathcal{E}_{\gamma}^{\tau}=\{E \in \Sigma_{V,\alpha} \colon \rho(E) \in \Theta_{\gamma}^{\tau}\}. 
\end{equation*}
For any $E \in \mathcal{E}_{\gamma}^{\tau}$, we define
\begin{equation}\label{uEn}
    u_E(n)=\frac{\hat{z}_{11}(n)}{\|\hat{z}_{11}\|_{\ell^2}}, \quad n \in \mathbb{Z}^d.
\end{equation}
Then combining \eqref{eigenz} with \eqref{uEn} implies $\{u_E(n)\}_{n \in \mathbb{Z}^d}$ is a normalized eigenfunction of the long-range operator $L_{V, \alpha, \rho(E)}$. To get the quantitative estimate for $|u_{E}(n)|$, we need the following result:
\begin{Lemma}[{\cite[Lemma 4.2]{avila2017sharp}}]\label{lem3.1}
    For the matrix $B_1$ defined above, one has
    \begin{equation*}
        \|b_{11}\|_{L^{2}} \geqslant (2\|B_1\|_0)^{-1}.
    \end{equation*}
\end{Lemma}

Apply Lemma \ref{lem3.1} to $B_{1}$ in \eqref{B1x}, we obtain
\begin{equation*}
    \|\hat{z}_{11}\|_{\ell^2}=\|z_{11}\|_{L^{2}}=\|b_{11}\|_{L^2} \geqslant  (2\|B_1\|_0 )^{-1} \geqslant (6\|\widetilde{B}\|_0 )^{-1}.
\end{equation*}
Consequently, we have the estimate
\begin{equation*}
    \begin{split}
        |u_E(n)| & \lesssim_{k,d} \|U\| \|\widetilde{B}\|_{0}\|\widetilde{B}\|_{k} (1+|n_0|)^{k} (1+|n|)^{-k} \\
        & \quad +\|U_{11} Y_{21} + U_{21} Y_{22} \|_{k} \|\widetilde{B}\|_{0}\|\widetilde{B}\|_{k}(1+|n_0|)^{k}(1+|n+n^*|)^{-k}.
    \end{split}
\end{equation*}

We distinguish three cases according to the values of $\nu,t$ and $\rho(\alpha,\bar{A})$.

\medskip

\textbf{Case 1:} $4|\nu|<(1-\|Y\|_{k})|\rho(\bar{A})|$ and $\rho(\bar{A})t>0$. In this case, the following lemma shows that one may choose $U \in \mathrm{SL}(2,\mathbb{C})$ close to the identity matrix.
\begin{Lemma}[{\cite[Lemma 4.1]{ge2019exponential}}]\label{lem4.2}
Assume that
\begin{equation*}
    MAM^{-1}=\exp \begin{pmatrix} 
                    it & \nu \\ 
                    \bar{\nu} & -it 
                \end{pmatrix} \in \mathrm{SU}(1,1),
\end{equation*}
with $\mathrm{spec}(A)=\{e^{2\pi i\rho}, e^{-2\pi i\rho}\}$. If $4|\nu| \leqslant |\rho|$ and $\rho t > 0$, then there exists $U \in \mathrm{SL}(2, \mathbb{C})$ with $\|U-\mathrm{Id}\| \leqslant \frac{|\nu|}{|\rho|}$ such that
\begin{equation*}
    UMAM^{-1}U^{-1}=
    \begin{pmatrix} 
        e^{2\pi i\rho} & 0 \\ 
        0 & e^{-2\pi i\rho} 
    \end{pmatrix}.
\end{equation*}
\end{Lemma}

Applying Lemma~\ref{lem4.2} to the matrix $\bar{A}$, there exists $U \in \mathrm{SL}(2,\mathbb{C})$ such that
\begin{equation*}
    UM\bar{A}M^{-1}U^{-1}= 
    \begin{pmatrix}
        e^{2\pi i\rho(\bar{A})} & 0 \\
        0 & e^{-2\pi i\rho(\bar{A})}
    \end{pmatrix},
\end{equation*}
with $\|U-\mathrm{Id}\| \leqslant \frac{|\nu|}{|\rho(\bar{A})|}$. Moreover, $|U_{11}-1| \leqslant \frac{|\nu|}{|\rho(\bar{A})|}$ and $|U_{21}| \leqslant \frac{|\nu|}{|\rho(\bar{A})|}<\frac{1}{4}$. Hence,
\begin{equation*}
    \begin{split}
        |u_E&(n)| \lesssim_{k,d} \|\widetilde{B}\|_{0}\|\widetilde{B}\|_{k} (1+|n_0|)^{k} (1 + |n|)^{-k} \\
        & \quad +\|\widetilde{B}\|_{0}\|\widetilde{B}\|_{k} (1+|n_0|)^{k} \bigg(\|Y\|_{k} + \frac{4|\nu|}{|\rho(\bar{A})|} \bigg) (1 + |n + n^*|)^{-k},
    \end{split}
\end{equation*}
which shows that $u_E$ is a $(k, n^*, C, C_{|n^*|})$-good eigenfunction. 

\medskip

\textbf{Case 2:} $4|\nu|<(1-\|Y\|_{k})|\rho(\bar{A})|$ and $\rho(\bar{A})t<0$. This case cannot occur by the choice of the canonical elliptic logarithm.

\medskip

\textbf{Case 3:} $4|\nu| \geqslant (1-\|Y\|_{k}) \, |\rho(\bar{A})|$. Since the matrix $U$ chosen before the case distinction is unitary, one has $|U_{11}|, \, |U_{21}| \leqslant 1$. Therefore,
\begin{equation*}
    |U_{11}| \|Y_{21}\|_{k} + |U_{21}| \|Y_{22}\|_{k} \lesssim_{k,d} 1.
\end{equation*}
It follows that
\begin{equation*}
    \begin{split}
        |u_E(n)| & \lesssim_{k,d} \|\widetilde{B}\|_{0}\|\widetilde{B}\|_{k} (1+|n_0|)^{k} (1 + |n|)^{-k} \\
        & \quad +\|\widetilde{B}\|_{0}\|\widetilde{B}\|_{k} (1+|n_0|)^{k} (1 + |n + n^*|)^{-k},
    \end{split}
\end{equation*}
so $u_E$ is a $(k, n^*, C, C_{|n^*|})$-good eigenfunction.

\medskip

Combining \textbf{Cases 1--3}, the long-range operator $L_{V,\alpha,\rho(E)}$ admits a 
\begin{equation*}
    \bigg(k, \, n^*, \, \tilde{C}(k,d)\|\widetilde{B}\|_{0}\|\widetilde{B}\|_{k} (1+|n_0|)^{k}, \, \min\bigg\{1, \, \|Y\|_{k} + \frac{4|\nu|}{|\rho(\bar{A})|} \bigg\} \bigg)
\end{equation*}
-good eigenfunction. 

If $n^{*}=0$, the two localization centers coincide, and the required estimate follows directly from $C_{|n^{*}|} \leqslant 1 $, without using any bound on $\nu$. Hence it remains to consider $n^{*} \neq 0$. 

By \eqref{est}, 
\begin{equation*}
    \begin{split}
        & C \coloneq  \tilde{C}(k,d)\|\widetilde{B}\|_{0}\|\widetilde{B}\|_{k} (1+|n_0|)^{k} \leqslant C_1(k,d)\gamma^{-(c_{3}+c_{4}k)}, \\
        & C_{|n^*|} \coloneq \min\bigg\{1,\, \|Y\|_k + \frac{4|\nu|}{|\rho(\bar{A})|} \bigg\} \\
        &\quad \quad \leqslant \min\bigg\{1, (1+|n^{*}|)^{-c_{2}}+ \frac{4(1+|n^*|)^{-c_5+\tau}}{16^{-\tau}\gamma} \bigg\} \\
        &\quad\quad \leqslant \min \bigg\{1, \frac{100^{\tau}}{\gamma} (1+|n^{*}|)^{-\min(c_{2}, c_{5}-\tau)}\bigg\}.
    \end{split}
\end{equation*}
This proves \eqref{CC}. 

Estimate \eqref{A2} holds by considering the following two cases.

\textbf{Case 1:} If $\frac{100^{\tau}}{\gamma} (1+|n^*|)^{-\min\{c_2, c_5-\tau\}} \leqslant 1$, then $c_{5}>\tau$ and $\frac{k}{\min\{c_2, c_5-\tau\}}<1$ imply 
\begin{equation*}
    \begin{split}
        C_{|n^*|} (1+|n^*|)^{k} &\leqslant \frac{100^{\tau}}{\gamma} (1+|n^*|)^{-\min\{c_2, c_5-\tau\}}(1+|n^*|)^{k} \\
        &\lesssim_{\tau} \gamma^{-1} \gamma^{\frac{\min\{c_2, c_5-\tau\} - k}{\min\{c_2, c_5-\tau\}}} \\
        &\lesssim_{\tau} \gamma^{-\frac{k}{\min\{c_2, c_5-\tau\}}}. 
    \end{split}
\end{equation*}

\textbf{Case 2:} If $\frac{100^{\tau}}{\gamma} (1+|n^*|)^{-\min\{c_2, c_5-\tau\}} \geqslant 1$, then $c_{5}>\tau$ implies 
\begin{equation*}
    C_{|n^*|} (1+|n^*|)^{k} \leqslant (1+|n^*|)^{k} \lesssim_\tau \gamma^{-\frac{k}{\min\{c_2, c_5-\tau\}}}. 
\end{equation*}
Therefore, combining two cases with  \eqref{cond}, we obtain 
\begin{equation*}
    C^2\big( 1 + C_{|n^*|} (1+|n^*|)^{k} \big) \lesssim_{k,d,\tau} \gamma^{-2(c_3+c_4k)} \gamma^{-\frac{k}{\min\{c_2,c_5-\tau\}}} = \gamma^{-c}.
\end{equation*}
This proves \eqref{A2}.

To see \eqref{decayuE}, note that $C_{|n^{*}|} \leqslant 1$, one has
    \begin{equation*}
        \begin{split}
            |u_{E}(n)| & \lesssim_{k,d} \|\widetilde{B}\|_{0}\|\widetilde{B}\|_{k} (1+|n_0|)^{k} \big((1+|n|)^{-k}+(1+|n+n^{*}|)^{-k}\big) \\
            & \lesssim_{k,d} \gamma^{-(c_{3}+c_{4}k)} \big((1+|n|)^{-k}+(1+|n+n^{*}|)^{-k}\big) \\
            & \lesssim_{k,d} \gamma^{-(c_{3}+c_{4}k)}(1+|n|)^{-k}\bigg(1+\frac{(1+|n|)^{k}}{(1+|n+n^{*}|)^{k}}\bigg).
        \end{split}
    \end{equation*}
    According to $|n^{*}| \leqslant \gamma^{-c_{1}}$, it is easy to see that, for $|n| \geqslant (1+\gamma^{-c_{1}})$,
    \begin{equation*}
        \begin{split}
            \frac{1+|n|}{1+|n+n^{*}|} & \leqslant \frac{1+|n|}{1+|n|-|n^{*}|} \leqslant \frac{1+|n|}{1+|n|-|n|/(1+\gamma^{c_{1}})} \\ 
            & \leqslant \frac{1+\gamma^{c_{1}}}{\gamma^{c_{1}}} \leqslant 2\gamma^{-c_{1}}. 
        \end{split}
    \end{equation*}
    And for $|n| \leqslant (1+\gamma^{-c_{1}})$,
    \begin{equation*}
        \frac{1+|n|}{1+|n+n^{*}|} \leqslant 1+|n| \leqslant 3\gamma^{-c_{1}}.
    \end{equation*}
    Thus 
    \begin{equation*}
        |u_{E}(n)| \leqslant C_2(k,d)\gamma^{-(c_{3}+c_{4}k+c_{1}k)} (1+|n|)^{-k}.  
    \end{equation*}
    This proves \eqref{decayuE}. This concludes the proof of Theorem~\ref{goodeig}.	 
\end{proof}

\begin{Lemma}\label{branchlemma}
    Let $E \in \Sigma_{V,\alpha}$ with $\rho(E)\in \Theta_{\gamma}^{\tau}$, and let $u_E^{+}$ be the normalized eigenfunction of $L_{V,\alpha,\rho(E)}$ constructed in Theorem~\ref{goodeig}. Define
    \begin{equation}\label{uminus}
        u_E^{-}(n)=\overline{u_E^{+}(-n)},\quad n \in \mathbb{Z}^{d}.
    \end{equation}
    Then $u_E^{-}$ is a normalized eigenfunction of $L_{V,\alpha,-\rho(E)}$ with eigenvalue $E$. Moreover, $u_E^{-}$ satisfies the same one-site decay estimate as $u_E^{+}$. If $u_E^{+}$ is $(k,n^*,C,C_{|n^*|})$-good, then $u_E^{-}$ is $(k,-n^*,C,C_{|n^*|})$-good.
\end{Lemma}

\begin{proof}
    Since $V$ is real-valued, one has $\widehat{V}_{-m}=\overline{\widehat{V}_{m}}$. Taking the complex conjugate of
    \begin{equation*}
        \sum_{m \in \mathbb{Z}^{d}}\widehat{V}_{m}u_E^{+}(n-m)+2\cos 2\pi(\rho(E)+\langle n,\alpha\rangle)u_E^{+}(n)=E u_E^{+}(n),
    \end{equation*}
    and then replacing $n$ by $-n$, we obtain
    \begin{equation*}
        \sum_{m \in \mathbb{Z}^{d}}\widehat{V}_{m}u_E^{-}(n-m)+2\cos 2\pi(-\rho(E)+\langle n,\alpha\rangle)u_E^{-}(n)=E u_E^{-}(n).
    \end{equation*}
    Thus $u_E^{-}$ is a normalized eigenfunction of $L_{V,\alpha,-\rho(E)}$ with eigenvalue $E$. The decay estimates follow immediately from \eqref{uminus}.
\end{proof}

\section{Arithmetic polynomial spectral localization}

In this section, we prove Theorem~\ref{main1}\ref{item:psl}. More precisely, we establish the following result:

\begin{theorem}\label{APSL}
    Let $\alpha \in \mathrm{DC}_{d}(\kappa, \tau)$ and $k>d$. Suppose that, for every $\gamma \in (0,1)$ and every $E \in \Sigma_{V,\alpha}$ satisfying $\rho(E) \in \Theta_{\gamma}^{2\tau}$ the cocycle $(\alpha, S_{E}^{V})$ is $C^{k}$-reducible, and that \eqref{red}--\eqref{est} hold with the same $\gamma$ and with $\tau$ replaced by $2\tau$ in \eqref{est}. Assume moreover that
    \begin{equation}\label{fulluniformcond}
        2\tau(c_3+c_4k+c_1k)<k/2.
    \end{equation}
    Then the long-range operator $L_{V,\alpha,\theta}$ exhibits arithmetic $k$-PSL for every $\theta \in \Theta^{\tau}$.
\end{theorem}

By shifting the good eigenfunctions constructed in Theorem~\ref{goodeig}, we obtain a family of good eigenfunctions for the long-range operator $L_{V,\alpha,\theta}$. To establish their completeness, we introduce an auxiliary $\mathcal{R}$-measure.

\subsection{\texorpdfstring{$\mathcal{R}$-measure}{R-measure}}
Let $\alpha \in \mathrm{DC}_d(\kappa, \tau)$. For $(\alpha,V) \in \mathbb{T}^{d}\times C^{k}(\mathbb{T}^d, \mathbb{R})$, define 
\begin{equation*}
    \mathcal{E}^{2\tau} \coloneq \{E \in \Sigma_{V,\alpha} \colon \rho(E) \in \Theta^{2\tau}\}. 
\end{equation*}
Suppose that the cocycle $(\alpha, S_{E}^{V})$ is $C^k$-reducible for all $E \in \mathcal{E}^{2\tau}$. For any such $E$, there exist  $B_1 \in C^k(2\mathbb{T}^d,\mathrm{SL}(2,\mathbb{C}))$ and a constant matrix $\bar{A}\in \mathrm{SL}(2,\mathbb{R})$ such that 
\begin{equation}\label{conj3}
    B_1^{-1}(x+\alpha)S_E^V(x)B_1(x)=
    \begin{pmatrix} 
        e^{2\pi i\rho(\bar{A})} & 0 \\ 
        0 & e^{-2\pi i\rho(\bar{A})} 
    \end{pmatrix}.
\end{equation}
Here and below, the complex diagonalizing conjugacies are understood in the admissible sense: they are obtained from real conjugacies by the fixed matrix $M$ and by constant complex changes of coordinates. Their degrees are defined to be the degrees of the corresponding real conjugacies.

Writing $B_1(x)=\begin{pmatrix} 
                    b_{11}(x) & b_{12}(x) \\ 
                    b_{21}(x) & b_{22}(x) 
                \end{pmatrix}$ 
and defining $z_{11}(x)=e^{-\pi i \langle \deg B_1,x \rangle}b_{11}(x)$, the vector-valued function $u_E$ given by 
\begin{equation}\label{uE}
    u_E \colon \begin{cases}
                  \mathcal{E}^{2\tau} \longrightarrow \ell^{2}(\mathbb{Z}^{d}), \\
                  E \longmapsto \Big\{\frac{\hat{z}_{11}(n)}{\|\hat{z}_{11}\|_{\ell^2}}\Big\}_{n\in\mathbb{Z}^d},  
              \end{cases} 
\end{equation}
is a normalized eigenfunction of the long-range operator $L_{V,\alpha,\rho(E)}$. 
    
The conjugacy $B_1$ is not unique, so the vector $u_E$ is not expected to be uniquely determined as a vector. However, based on Ge-You \cite[Lemma~3.1]{ge2020arithmetic}, we provide demonstration of the following lemma that $u_E$ is uniquely determined up to multiplication by a complex number of modulus one. Consequently, all quantities $|u_E(n)|^2$ are well-defined.

\begin{Lemma}\label{welldefine} 
    For every $E \in \mathcal{E}^{2\tau}$, $u_E$ in \eqref{uE} is well-defined up to a unimodular constant. 
\end{Lemma}
\begin{proof}
    For $E \in \mathcal{E}^{2\tau}$. Suppose that two choices of diagonalizing conjugacies are given by
    \begin{equation}\label{Bconj1}
        B_1^{-1}(x+\alpha)S_E^V(x)B_1(x)=
        \begin{pmatrix}
            e^{2\pi i\rho(A_E)} & 0 \\
            0 & e^{-2\pi i\rho(A_E)}
        \end{pmatrix}. 
    \end{equation}
    \begin{equation}\label{Bconj2}
        \widetilde{B}_1^{-1}(x+\alpha)S_E^V(x)\widetilde{B}_1(x)=
        \begin{pmatrix}
            e^{2\pi i\rho(\tilde{A}_E)} & 0 \\
            0 & e^{-2\pi i\rho(\tilde{A}_E)}
        \end{pmatrix}, 
    \end{equation}
    
    We first reduce to the case $\deg B_1=\deg \widetilde{B}_1$. Define
    \begin{equation*}
        Q(x)=M R_{\langle \deg \widetilde{B}_1-\deg B_1,x\rangle/2}M^{-1}=
        \begin{pmatrix}
            e^{\pi i\langle \deg \widetilde{B}_1-\deg B_1,x\rangle} & 0 \\
            0 & e^{-\pi i\langle \deg \widetilde{B}_1-\deg B_1,x\rangle}
        \end{pmatrix}, 
    \end{equation*}
    then 
    \begin{equation*}
        \deg Q=\deg \widetilde{B}_1-\deg B_1, \quad \deg(B_1Q)=\deg \widetilde{B}_1.
    \end{equation*}
    Moreover, because $Q$ is diagonal in the above coordinates,
    \begin{equation*}
        \begin{split}
            & (B_1Q)^{-1}(x+\alpha)S_E^V(x)(B_1Q)(x) =Q^{-1}(x+\alpha)
            \begin{pmatrix}
                e^{2\pi i\rho(A_E)} & 0 \\
                0 & e^{-2\pi i\rho(A_E)}
            \end{pmatrix}Q(x) \\
            & \qquad \qquad \quad =\begin{pmatrix}
                e^{2\pi i(\rho(A_E)-\langle \deg \widetilde{B}_1-\deg B_1,\alpha\rangle/2)} & 0 \\
                0 & e^{-2\pi i(\rho(A_E)-\langle \deg \widetilde{B}_1-\deg B_1,\alpha\rangle/2)}
            \end{pmatrix}.
        \end{split}
    \end{equation*}
    Hence $B_1Q$ is again an admissible diagonalizing conjugacy.

    We now check that replacing $B_1$ by $B_1Q$ does not change the vector $u_E$ defined in \eqref{uE}. Write
    \begin{equation*}
        B_1(x)=\begin{pmatrix}
                   b_{11}(x) & b_{12}(x) \\
                   b_{21}(x) & b_{22}(x)
               \end{pmatrix}, \quad
        B_1(x)Q(x)=\begin{pmatrix}
                    d_{11}(x) & d_{12}(x) \\
                    d_{21}(x) & d_{22}(x)
                   \end{pmatrix}, 
    \end{equation*}
    then $d_{11}(x)=b_{11}(x)e^{\pi i\langle \deg \widetilde{B}_1-\deg B_1,x\rangle}$. Therefore,
    \begin{equation*}
        e^{-\pi i\langle \deg \widetilde{B}_1,x\rangle}d_{11}(x)=e^{-\pi i\langle \deg \widetilde{B}_1,x\rangle}b_{11}(x)e^{\pi i\langle \deg \widetilde{B}_1-\deg B_1,x\rangle}=e^{-\pi i\langle \deg B_1,x\rangle}b_{11}(x).
    \end{equation*}
    Thus the function $z_{11}$ and the vector $u_E$ are unchanged. Consequently, without loss of generality, we may assume $\deg B_1=\deg\widetilde B_1$. 

    By \eqref{rot2}, \eqref{Bconj1}, and \eqref{Bconj2}, we have
    \begin{equation}\label{rhoequal}
        \rho(A_E)=\rho(E)-\frac{\langle \deg B_1,\alpha\rangle}{2}=\rho(\tilde{A}_E) \bmod \mathbb{Z}.
    \end{equation}
    Combining \eqref{Bconj1} and \eqref{Bconj2}, we obtain
    \begin{equation}\label{Tequation}
        T(x+\alpha)\begin{pmatrix}
            e^{2\pi i\rho(A_E)} & 0 \\
            0 & e^{-2\pi i\rho(A_E)}
        \end{pmatrix}
        =
        \begin{pmatrix}
            e^{2\pi i\rho(A_E)} & 0 \\
            0 & e^{-2\pi i\rho(A_E)}
        \end{pmatrix}T(x), 
    \end{equation}
    where $T(x)=\widetilde B_1^{-1}(x)B_1(x)=
        \begin{pmatrix}
            t_{11}(x) & t_{12}(x) \\
            t_{21}(x) & t_{22}(x)
        \end{pmatrix}$.
    Thus
    \begin{equation*}
        t_{11}(x+\alpha)=t_{11}(x), \quad t_{22}(x+\alpha)=t_{22}(x),
    \end{equation*}
    \begin{equation*}
        t_{12}(x+\alpha)=e^{4\pi i\rho(A_E)}t_{12}(x), \quad t_{21}(x+\alpha)=e^{-4\pi i\rho(A_E)}t_{21}(x).
    \end{equation*}
    
Since $B_{1}$ and $\widetilde{B}_{1}$ are admissible and $\deg B_{1}=\deg\widetilde{B}_{1}$, their quotient $T(x)=\widetilde{B}_{1}^{-1}(x)B_{1}(x)$ is $\mathbb{T}^{d}$-periodic. Indeed, for every $l \in \mathbb{Z}^{d}$, by \eqref{admissibleparity}, 
\begin{equation*}
	\begin{split}
		T(x+l) & =\widetilde{B}_{1}^{-1}(x+l)B_{1}(x+l) \\
		& =\big((-1)^{\langle\deg B_{1},l\rangle}
		\widetilde{B}_{1}(x)\big)^{-1}
		(-1)^{\langle\deg B_{1},l\rangle}B_{1}(x)=T(x).
	\end{split}
\end{equation*}
Thus $T \in C^{k}(\mathbb{T}^{d},\mathrm{SL}(2,\mathbb{C}))$ and its Fourier expansion contains only frequencies in $\mathbb{Z}^{d}$. Therefore, the following Fourier coefficient argument is well-defined.
    
    Taking Fourier coefficients gives
    \begin{equation*}
        \big(e^{2\pi i\langle n,\alpha\rangle}-1\big)\hat{t}_{11}(n)=0, \quad \big(e^{2\pi i\langle n,\alpha\rangle}-1\big)\hat{t}_{22}(n)=0,
    \end{equation*}
    \begin{equation*}
        \big(e^{2\pi i\langle n,\alpha\rangle}-e^{4\pi i\rho(A_E)}\big) \widehat{t}_{12}(n)=0, \quad \big(e^{2\pi i\langle n,\alpha\rangle}-e^{-4\pi i\rho(A_E)}\big)\hat{t}_{21}(n)=0.
    \end{equation*}
    Since $\alpha$ is rationally independent, $t_{11}(x)=\hat{t}_{11}(0)$ and $t_{22}(x)=\hat{t}_{22}(0)$. Moreover, since $2\rho(A_E) \neq \langle l,\alpha\rangle \bmod \mathbb{Z}$ for any $l \in \mathbb{Z}^d$, we have $t_{12}(x) \equiv 0$, $t_{21}(x) \equiv 0$.
    
    Therefore, we have $\hat{t}_{11}(0), \hat{t}_{22}(0) \neq 0$ and 
    \begin{equation*}
    T(x)=\widetilde B_1^{-1}(x)B_1(x)=
        \begin{pmatrix}
            \hat{t}_{11}(0) & 0 \\
            0 & \hat{t}_{22}(0)
        \end{pmatrix}. 
    \end{equation*}    
    Write $\widetilde{B}_1(x)=\begin{pmatrix}
        \tilde{b}_{11}(x) & \tilde{b}_{12}(x) \\
        \tilde{b}_{21}(x) & \tilde{b}_{22}(x)
    \end{pmatrix}$, then $b_{11}(x)=\hat{t}_{11}(0)\tilde{b}_{11}(x)$. By the definition \eqref{uE} and normalization, we can get $u_E=\frac{\hat{t}_{11}(0)}{|\hat{t}_{11}(0)|}\tilde{u}_E$. Thus $u_E$ is well-defined up to a unimodular constant.
\end{proof}
    
We now define the function $E(\cdot) \colon \mathbb{T} \to \Sigma_{V,\alpha}$ by
\begin{equation}\label{etheta}
    E(\theta) = 
    \begin{cases}
        \rho^{-1}(\theta), & \theta \in [0,1/2], \\
        \rho^{-1}(1-\theta), & \theta \in [1/2,1].
    \end{cases}
\end{equation}
At a gap label where $\rho^{-1}$ is not single-valued, we choose one of the two spectral endpoints once and for all. With this convention, $E(\theta)$ is well-defined and single-valued on $\mathbb{T}$.

For $\theta \in \Theta^{2\tau}$, we always choose its representative in $[0,1]$ and define
\begin{equation}\label{utheta}
    u_{\theta}=
    \begin{cases}
        u_{E(\theta)}^{+}, & \theta \in [0,1/2], \\
        u_{E(\theta)}^{-}, & \theta \in [1/2,1].
    \end{cases}
\end{equation}
Indeed, if $\theta \in [0,1/2]$, then $\rho(E(\theta))=\theta$ and this follows from the definition of $u_{E(\theta)}^{+}$. If $\theta \in [1/2,1]$, then $\rho(E(\theta))=1-\theta$ and $-\rho(E(\theta))=\theta \bmod \mathbb{Z}$, so the claim follows from Lemma~\ref{branchlemma}.

For any fixed $\theta \in \Theta^{2\tau}$, we define $T^m \theta \coloneq \theta-\langle m,\alpha\rangle \bmod \mathbb{Z}, \ E_m(\theta)\coloneq E(T^m\theta)$, and $u_{m,\theta} \coloneq T_{-m}u_{T^m\theta}$, where $T_m$ is the shift operator defined by $(T_m u)(n)=u(n-m)$. Then, 
\begin{equation}\label{umthetaeig}
    L_{V,\alpha,\theta}u_{m,\theta}=E_m(\theta)u_{m,\theta}, \quad \|u_{m,\theta}\|_{\ell^2}=1.
\end{equation}
Moreover, by Lemma~\ref{welldefine} and Lemma~\ref{branchlemma}, $u_{m,\theta}$ is well-defined up to a unimodular constant. We now introduce $\mathcal{R}$-measure.

\begin{definition}\label{RM}
    For any fixed $\theta \in \Theta^{2\tau}$ and $n \in \mathbb{Z}^d$, the $\mathcal{R}$-measure $\nu_{\theta, \delta_n} \colon \mathfrak{B} \to \mathbb{R}$ is defined by
    \begin{equation*}
        \nu_{\theta, \delta_n}(B)=\sum_{m \in \mathcal{M}_\theta^B} |u_{m,\theta}(n)|^2
    \end{equation*}
    for all $B \in \mathfrak{B}$, where $\mathfrak{B}$ denotes the Borel $\sigma$-algebra on $\mathbb{R}$ and $\mathcal{M}_\theta^B=\{m \in \mathbb{Z}^d \colon E_m(\theta) \in B\}$.
\end{definition}
	
Denote by $\mu^{\mathrm{pp}}_{\theta,\delta_n}$ the pure point component of the spectral measure $\mu_{\theta,\delta_n}$ associated with the long-range operator $L_{V,\alpha,\theta}$, which is defined via the resolvent equation:
\begin{equation*}
    \langle \delta_n, (L_{V,\alpha,\theta}-z)^{-1} \delta_n \rangle =\int_{\mathbb{R}} \frac{1}{E-\bar{z}} \, \mathrm{d}\mu_{\theta,\delta_n}(E), \quad z \in \mathbb{C} \setminus \mathbb{R}. 
\end{equation*}
The following lemma establishes a relationship between the $\mathcal{R}$-measure $\nu_{\theta,\delta_n}$ and the pure point spectral measure $\mu^{\mathrm{pp}}_{\theta,\delta_n}$.

\begin{Lemma}\label{dominate}
    For any $n \in \mathbb{Z}^d$ and any Borel set $B \in \mathfrak{B}$, the following hold:
    \begin{enumerate}[label=(\Alph*)]
        \item \label{acont}
        $\nu_{\theta, \delta_n}(B) \leqslant
        \mu^{\mathrm{pp}}_{\theta, \delta_n}(\mathcal{E}^{2\tau} \cap B)$ for every $\theta \in \Theta^{2\tau}$.
        \item \label{leb}
        One has $\nu_{\theta,\delta_n}(\mathcal{E}^{2\tau})=1$ for a.e. $\theta \in \Theta^{2\tau}$. 
    \end{enumerate}
\end{Lemma}

Lemma~\ref{dominate} was originally proved for analytic potentials, see \cite[Lemma 3.2]{ge2020arithmetic}. Because the definition of the $\mathcal{R}$-measure $\nu_{\theta,\delta_n}$ in this paper differs from the one used in \cite{ge2020arithmetic}, we provide a complete proof here.

\begin{proof}
    For $\theta \in \Theta^{2\tau}$, the energies $\{E_m(\theta)\}_{m \in \mathbb Z^d}$ are mutually distinct. We first prove \ref{acont}. For any $\theta \in \Theta^{2\tau}$ and $m \in \mathbb{Z}^d$, let $P_m(\theta)$ be the spectral projection of
    $L_{V,\alpha,\theta}$ onto the eigenspace corresponding to $E_m(\theta)$. By \eqref{umthetaeig}, $u_{m,\theta}$ is a normalized eigenfunction of $L_{V,\alpha,\theta}$ with eigenvalue $E_m(\theta)$. By the Spectral Theorem, we have
\begin{equation*}
    \begin{split}
        \mu_{\theta,\delta_n}^{\mathrm{pp}}(\mathcal{E}^{2\tau} \cap B) \geqslant \sum_{m \in \mathcal{M}_\theta^B} \langle P_m(\theta)\delta_n, \delta_n \rangle \geqslant \sum_{m \in \mathcal{M}_\theta^B}|u_{m,\theta}(n)|^2=\nu_{\theta,\delta_n}(B).
    \end{split}
\end{equation*}

    Now we prove \ref{leb}. Recall that
    $|\mathbb{T} \setminus \Theta^{2\tau}|=0$. We extend the definition of $\nu_{\theta,\delta_n}$ on $\Theta^{2\tau}$ to $\mathbb{T}$ by setting
    \begin{equation*}
        \tilde{\nu}_{\theta}(B)=
        \begin{cases}
            \nu_{\theta,\delta_n}(B), & \theta \in \Theta^{2\tau}, \\
            0, & \theta \notin \Theta^{2\tau}.
        \end{cases}
    \end{equation*}
    For any $\gamma\in(0,1)$, applying Fubini's theorem gives
    \begin{equation}\label{finitegammaint}
        \begin{split}
            \int_{\mathbb{T}} \tilde{\nu}_\theta(\mathcal{E}_\gamma^{2\tau}) \, \mathrm{d}\theta
            & =\int_{\mathbb{T}}\sum_{m \colon T^m\theta \in \Theta^{2\tau}_\gamma}|u_{m,\theta}(n)|^2 \, \mathrm{d}\theta \\
            & =\sum_{m \in \mathbb{Z}^d}
            \int_{\Theta^{2\tau}_\gamma}|T_{-m}u_{T^m\theta}(n)|^2 \, \mathrm{d}T^m\theta \\
            & =\int_{\Theta^{2\tau}_\gamma}\sum_{m \in \mathbb{Z}^d}|u_{\theta}(n+m)|^2 \, \mathrm{d}\theta=|\Theta^{2\tau}_\gamma|.
        \end{split}
    \end{equation}
    Let $\gamma \rightarrow 0$, the monotone convergence theorem and \eqref{finitegammaint} imply
    \begin{equation}\label{fullmassint}
        \int_{\mathbb{T}}\tilde{\nu}_\theta(\mathcal{E}^{2\tau}) \, \mathrm{d}\theta=|\Theta^{2\tau}|=1.
    \end{equation}
    On the other hand, by \ref{acont},
    \begin{equation*}
        0 \leqslant \nu_{\theta,\delta_n}(\mathcal{E}^{2\tau})=\tilde{\nu}_\theta(\mathcal E^{2\tau}) \leqslant \mu^{\mathrm{pp}}_{\theta,\delta_n}(\mathcal{E}^{2\tau}) \leqslant \mu_{\theta,\delta_n}(\mathbb{R})=1
    \end{equation*}
    for any $\theta \in \Theta^{2\tau}$, while
    $\tilde{\nu}_\theta(\mathcal{E}^{2\tau})=0$ for
    $\theta \notin \Theta^{2\tau}$. Combining this with \eqref{fullmassint}, we obtain $\tilde{\nu}_\theta(\mathcal{E}^{2\tau})=1$ for a.e. $\theta \in \mathbb{T}$. Since $|\mathbb{T} \setminus \Theta^{2\tau}|=0$, this gives
    \begin{equation*}
        \nu_{\theta,\delta_n}(\mathcal{E}^{2\tau})=1
    \end{equation*}
    for a.e. $\theta \in \Theta^{2\tau}$.
\end{proof}

\subsection{Arithmetic spectral localization}\label{aspl}

In the following, we will show the arithmetic spectral localization, i.e. $\mu_{\theta,\delta_{n}}^{\mathrm{pp}}(\mathbb{R})=1$ for any $\theta \in \Theta_{\gamma}^{\tau}$ and $n \in \mathbb{Z}^{d}$.

\begin{Proposition}\label{APP}
	Let $\alpha \in \mathrm{DC}_d(\kappa,\tau)$ and $k>d$. Suppose that, for every $\gamma \in (0,1)$ and every $E \in \Sigma_{V,\alpha}$ satisfying $\rho(E) \in \Theta_{\gamma}^{2\tau}$, the cocycle $(\alpha,S_E^V)$ is $C^k$-reducible, and that \eqref{red}--\eqref{est} and \eqref{fulluniformcond} hold with the same $\gamma$ and with $\tau$ replaced by $2\tau$ in \eqref{est}. Then, $L_{V,\alpha,\theta}$ has pure point spectrum for every $\theta \in \Theta^\tau$. 
\end{Proposition}

\begin{proof}
By Theorem~\ref{goodeig}\ref{item:res1}, for any $E \in \Sigma_{V,\alpha}$ with $\rho(E) \in \Theta^{2\tau}$, we can construct a normalized eigenfunction $u_E$ of $L_{V,\alpha,\rho(E)}$ corresponding to $E$, with polynomial decay estimate \eqref{decayuE}. 

Fix $\theta \in \Theta^{\tau}$ and $n \in \mathbb{Z}^d$. Define
\begin{equation}\label{22}
    \begin{split}
        & \mathcal{T}_M \nu_{\theta, \delta_n}(B) = \sum_{|m| \leqslant M, m \in \mathcal{M}_\theta^B} |u_{m,\theta}(n)|^2, \\
        & \mathcal{R}_M \nu_{\theta, \delta_n}(B) = \sum_{|m|>M, m \in \mathcal{M}_\theta^B}|u_{m,\theta}(n)|^2.
    \end{split}
\end{equation}

The framework of our proof closely follows that of \cite{ge2020arithmetic,ge2021arithmetic}. More precisely, we utilize the homogeneity of $\Theta^{\tau}_{\gamma}$ (Lemma~\ref{hom}) to construct the approximation sequence, the uniformity (Lemma~\ref{uniform}) to control $\mathcal{R}_M \nu_{\theta, \delta_n}$, and the continuity (Lemma~\ref{continuous}) to control $\mathcal{T}_M \nu_{\theta, \delta_n}$. We first verify that all these conditions hold in the $C^{k}$-setting, and subsequently apply them to complete the proof.

\medskip

\noindent\textbf{Verification of Homogeneity.}

\begin{definition}
    For any $\theta \in \mathbb{R}$ and $B \subseteq \mathbb{R}$, we say that $\theta$ is $B$-homogeneous if $|(\theta-\sigma, \theta+\sigma) \cap B| \geqslant \sigma$ for every $0<\sigma \leqslant 1/2$.
\end{definition}

It is clear that $\theta$ is $B'$-homogeneous whenever it is $B$-homogeneous for some $B' \supseteq B$. The homogeneity of $\Theta^{\tau}_{\gamma}$ is independent of the operators and is adapted from \cite[Lemma 4.4]{ge2021arithmetic}. 

\begin{Lemma}\label{hom}
    Let
    \begin{equation*}
        S_{d,\tau}=\sum_{l \in \mathbb{Z}^d}\frac{1}{(1+|l|)^{\tau+d}}, \quad T_{d,\tau}=\sup_{R \geqslant 1}R^\tau\sum_{1+|l| \geqslant R}\frac{1}{(1+|l|)^{\tau+d}}.
    \end{equation*}
    Set $c(d,\tau)=16\max\{1,S_{d,\tau},T_{d,\tau}\}$, then every $\theta \in \Theta^\tau_\gamma$ is $\Theta^{\tau+d}_{\gamma/c(d,\tau)}$-homogeneous. 
\end{Lemma}

\begin{proof}
    We show that
    \begin{equation}\label{badgoal}
        |(\theta-\sigma,\theta+\sigma) \cap (\mathbb{T} \setminus \Theta^{\tau+d}_{\gamma/c(d,\tau)})| \leqslant \sigma.
    \end{equation}
    Since $|(\theta-\sigma,\theta+\sigma)|=2\sigma$ for $0<\sigma \leqslant 1/2$, \eqref{badgoal} implies the desired homogeneity.

    Let $\Theta_l=\big\{\theta \in \mathbb{T} \colon \|2\theta-\langle l,\alpha\rangle\|_{\mathbb{R}/\mathbb{Z}}<\frac{\gamma}{c(d,\tau)(1+|l|)^{\tau+d}}\big\}$, then $\mathbb{T} \setminus \Theta^{\tau+d}_{\gamma/c(d,\tau)}=\cup_{l \in \mathbb{Z}^d}\Theta_l$, and 
    \begin{equation}\label{badlength}
        |\Theta_l| \leqslant \frac{2\gamma}{c(d,\tau)(1+|l|)^{\tau+d}}.
    \end{equation}

    We first consider the case $0<\sigma<\gamma/4$. Suppose that $\Theta_l \cap (\theta-\sigma,\theta+\sigma) \neq \emptyset$. Then there exists $\theta' \in (\theta-\sigma,\theta+\sigma)$ such that
    \begin{equation*}
        \|2\theta'-\langle l,\alpha\rangle\|_{\mathbb{R}/\mathbb{Z}}<\frac{\gamma}{c(d,\tau)(1+|l|)^{\tau+d}}.
    \end{equation*}
    Since $\theta \in \Theta^\tau_\gamma$, we have $\frac{\gamma}{(1+|l|)^\tau} \leqslant \|2\theta-\langle l,\alpha\rangle\|_{\mathbb{R}/\mathbb{Z}}$. Hence
    \begin{equation*}
        \frac{\gamma}{(1+|l|)^\tau} \leqslant 2\sigma+\frac{\gamma}{c(d,\tau)(1+|l|)^{\tau+d}} \leqslant 2\sigma+\frac{\gamma}{c(d,\tau)(1+|l|)^{\tau}}.
    \end{equation*}
    It follows that $\frac{\gamma}{2(1+|l|)^\tau} \leqslant 2\sigma$. Thus
    \begin{equation}\label{klarge}
        1+|l| \geqslant \bigg(\frac{\gamma}{4\sigma}\bigg)^{1/\tau}.
    \end{equation}
    Combining \eqref{badlength} and \eqref{klarge}, we obtain 
    \begin{equation*}
        \begin{split}
            |(\theta-\sigma,\theta+\sigma) \cap (\mathbb{T} \setminus \Theta^{\tau+d}_{\gamma/c(d,\tau)})|
            & \leqslant \sum_{1+|l| \geqslant (\gamma/(4\sigma))^{1/\tau}}\frac{2\gamma}{c(d,\tau)(1+|l|)^{\tau+d}} \\ 
            & \leqslant \frac{2\gamma}{c(d,\tau)}\frac{4\sigma}{\gamma}T_{d,\tau}=\frac{8T_{d,\tau}}{c(d,\tau)}\sigma \leqslant \sigma. 
        \end{split}
    \end{equation*}

    It remains to consider $\sigma \geqslant \gamma/4$. In this case, by \eqref{badlength},
    \begin{equation*}
        \begin{split}
            |(\theta-\sigma,\theta+\sigma) \cap (\mathbb{T} \setminus \Theta^{\tau+d}_{\gamma/c(d,\tau)})|
            & \leqslant |\mathbb{T} \setminus \Theta^{\tau+d}_{\gamma/c(d,\tau)}| \\
            & \leqslant \sum_{l \in \mathbb{Z}^d}\frac{2\gamma}{c(d,\tau)(1+|l|)^{\tau+d}}=\frac{2\gamma}{c(d,\tau)}S_{d,\tau} \leqslant \frac{\gamma}{8} \leqslant \sigma.
        \end{split}
    \end{equation*}
    Therefore \eqref{badgoal} holds for all $0<\sigma \leqslant 1/2$. This proves the lemma. 
\end{proof}
\begin{remark}
    The exponent $\tau+d$ is the critical exponent produced by the $\mathbb{Z}^d$ counting. Indeed, the bad intervals have total tail
    \begin{equation*}
        \sum_{1+|l| \geqslant (\gamma/(4\sigma))^{1/\tau}}(1+|l|)^{-s} \lesssim \bigg(\frac{\gamma}{4\sigma}\bigg)^\frac{d-s}{\tau}.
    \end{equation*}
    Since $0<\sigma<\gamma/4$, the bad contribution is bounded by $O(\sigma)$ exactly when $s \geqslant \tau+d$. Thus $\tau+d$ is the sharp threshold for this homogeneous covering argument.
\end{remark}

\medskip

\noindent\textbf{Verification of the Uniformity.}
\begin{Lemma}\label{uniform}
	Assume that \eqref{fulluniformcond} holds. For any $\varepsilon>0$ and $\bar{\gamma} \in (0,\gamma/c(d,\tau))$, there exists $M_{0}=M_0(k,d,\tau,\bar{\gamma},n,\varepsilon)$ such that if $M>M_0$, then 
	\begin{equation*}
		\mathcal{R}_M\nu_{\theta,\delta_n}(\mathcal{E}^{2\tau})<\varepsilon, \quad \forall \theta \in \Theta^{2\tau}_{\bar{\gamma}}. 
	\end{equation*}
\end{Lemma}

\begin{proof}
    Since $\theta \in \Theta^{2\tau}_{\bar\gamma}$, for any $m,l \in \mathbb{Z}^d$, one has
    \begin{equation*}
        \begin{split}
            \|2T^m\theta-\langle l,\alpha\rangle\|_{\mathbb{R}/\mathbb{Z}}
            & =\|2\theta-\langle 2m+l,\alpha\rangle\|_{\mathbb{R}/\mathbb{Z}} \\
            & \geqslant \frac{\bar\gamma}{(|2m+l|+1)^{2\tau}} \geqslant \frac{(2|m|+1)^{-2\tau}\bar\gamma}{(|l|+1)^{2\tau}}, 
        \end{split}
    \end{equation*}
    which implies that $T^m\theta \in \Theta^{2\tau}_{\gamma_m}$, where $\gamma_m=(2|m|+1)^{-2\tau}\bar\gamma$. By Theorem~\ref{goodeig}\ref{item:res1} and Lemma~\ref{branchlemma}, for any $T^m\theta \in \Theta^{2\tau}_{\gamma_m}$, the normalized eigenfunction $u_{T^m\theta}$ satisfies
    \begin{equation*}
        |u_{T^m\theta}(l)| \leqslant C_2(k,d) \gamma_m^{-(c_{3}+c_{4}k+c_{1}k)}(1+|l|)^{-k}. 
    \end{equation*}
    Therefore, for fixed $n \in \mathbb{Z}^d$, 
    \begin{equation*}
        \begin{split}
            |u_{m,\theta}(n)| &=|T_{-m} u_{T^m\theta}(n)|=|u_{T^m\theta}(n+m)| \\
            & \leqslant C_2(k,d,\tau,\bar\gamma)(1+|m|)^{2\tau(c_{3}+c_{4}k+c_{1}k)}(1+|n+m|)^{-k} \\
            & \leqslant C_3(k,d,\tau,\bar{\gamma}, n)(1+|m|)^{-k+2\tau(c_{3}+c_{4}k+c_{1}k)}.
        \end{split}
    \end{equation*}
    Choosing $M_0=\bigg(\frac{C_3^2}{(k-d)\varepsilon}\bigg)^{\frac{1}{k-d}}$, we deduce from \eqref{fulluniformcond}, the definition of $\mathcal{R}_M \nu_{\theta,\delta_n}(\mathcal{E}^{2\tau})$, and the assumption $k > d$ that for any $M>M_0$,
    \begin{equation*}
        \mathcal{R}_M \nu_{\theta,\delta_n}(\mathcal{E}^{2\tau})=\sum_{|m|>M}|u_{m,\theta}(n)|^2 
        \leqslant C_3^2\sum_{j > M} j^{-(k-d+1)}<\varepsilon.
    \end{equation*}
\end{proof}

\medskip

\noindent\textbf{Verification of Continuity.} 
	
\begin{Lemma}[{\cite[Theorem 4.4]{ge2020arithmetic}, \cite[Lemma 4.3]{ge2021arithmetic}}]\label{continuous}
    For any $M>0$, $\varepsilon>0$, and $\bar{\gamma} \in (0,\gamma/c(d,\tau))$, there exists $\delta=\delta(\alpha,V,d,k, \bar{\gamma}, n, \varepsilon, M) > 0$ such that
    \begin{equation*}
        |\mathcal{T}_M\nu_{\theta,\delta_n}(\mathcal{E}^{2\tau})-\mathcal{T}_M\nu_{\theta',\delta_n}(\mathcal{E}^{2\tau})|<\varepsilon
    \end{equation*}
    for all $\theta, \theta' \in \Theta^{2\tau}_{\bar{\gamma}}$ with $\|\theta-\theta' \|_{\mathbb{R}/\mathbb{Z}}<\delta$. 
\end{Lemma}

The proof of Lemma~\ref{continuous} relies on quantitative estimates for the conjugation of the reducible cocycle $(\alpha, S_E^V)$ in the $C^k$-topology. Given the $C^k$-reducibility established in Section~\ref{Reducibility}, we provide a complete proof of Lemma~\ref{continuous} in Appendix~\ref{appendixA}.

We now complete the proof of Proposition~\ref{APP}. 
Fix any $\theta \in \Theta^{\tau}_\gamma$ and $\bar{\gamma} \in (0,\gamma/c(d,\tau))$. Let $G_n=\{\theta \in \Theta^{2\tau} \colon \nu_{\theta,\delta_n}(\mathcal{E}^{2\tau})=1\}$. By Lemma~\ref{dominate}\ref{leb}, $G_n$ has full Lebesgue measure in $\mathbb{T}$. Since $\Theta^{\tau+d}_{\gamma/c(d,\tau)} \subseteq \Theta^{2\tau}_{\bar{\gamma}}$, Lemma~\ref{hom} implies that $\theta$ is $\Theta^{2\tau}_{\bar{\gamma}}$-homogeneous. Then every neighborhood of $\theta$ intersects $G_n \cap \Theta_{\bar\gamma}^{2\tau}$. We can choose a sequence $\{\theta_j\}_{j=1}^\infty \subseteq G_n \cap \Theta^{2\tau}_{\bar{\gamma}}$ such that
\begin{equation}\label{lim}
    \theta_j \xrightarrow{j \to \infty} \theta \qquad \text{and} \qquad \nu_{\theta_j,\delta_n}(\mathcal{E}^{2\tau})=1.
\end{equation}
By Lemma~\ref{uniform}, for $\varepsilon=\bar{\gamma}$, there exists $M_0=M_0(k,d,\tau,\bar{\gamma},n,\varepsilon)>0$ such that, for any $M>M_0$,
\begin{equation}\label{unif}
    \mathcal{R}_M\nu_{\theta_j,\delta_n}(\mathcal{E}^{2\tau})<\bar{\gamma}.
\end{equation}
By Lemma~\ref{continuous}, for such an $M$, we have
\begin{equation}\label{limj}
    \lim_{j \to \infty} \mathcal{T}_M\nu_{\theta_j,\delta_n}(\mathcal{E}^{2\tau})=\mathcal{T}_M\nu_{\theta,\delta_n}(\mathcal{E}^{2\tau}).
\end{equation}
Therefore, by \eqref{lim}, \eqref{unif}, and \eqref{limj},
\begin{equation*}
    \begin{split}
        \nu_{\theta,\delta_n}(\mathcal{E}^{2\tau}) 
        & \geqslant \mathcal{T}_M\nu_{\theta,\delta_n}(\mathcal{E}^{2\tau})=\lim_{j \to \infty} \mathcal{T}_M\nu_{\theta_j,\delta_n}(\mathcal{E}^{2\tau}) \\
        & =\lim_{j \to \infty} \big( \nu_{\theta_j,\delta_n}(\mathcal{E}^{2\tau})-\mathcal{R}_M\nu_{\theta_j,\delta_n}(\mathcal{E}^{2\tau})\big) \geqslant 1-\bar{\gamma}.
    \end{split}
\end{equation*}
Let $\bar{\gamma} \to 0$, we obtain $\nu_{\theta,\delta_n}(\mathcal{E}^{2\tau}) \geqslant 1$. On the other hand, by Lemma~\ref{dominate}\ref{acont},
\begin{equation*}
    1 \leqslant \nu_{\theta,\delta_n}(\mathcal{E}^{2\tau}) \leqslant \mu^{\mathrm{pp}}_{\theta,\delta_n}(\mathcal{E}^{2\tau}) \leqslant \mu^{\mathrm{pp}}_{\theta,\delta_n}(\mathbb{R}) \leqslant \mu_{\theta,\delta_n}(\mathbb{R}) \leqslant 1,
\end{equation*}
which implies that $\mu^{\mathrm{pp}}_{\theta,\delta_n}(\mathbb{R})=1$ for any $\theta \in \Theta^{\tau}_\gamma$ and $n \in \mathbb{Z}^d$. Since $\Theta^{\tau}=\bigcup_{\gamma \in (0,1)} \Theta^{\tau}_\gamma$, Proposition~\ref{APP} follows. 
\end{proof}

\medskip

\begin{proof}[Proof of Theorem~\ref{APSL}]
    $L_{V,\alpha,\theta}$ has pure point spectrum for every $\theta \in \Theta^\tau$. Moreover, the proof of Proposition~\ref{APP} gives
    \begin{equation*}
        \nu_{\theta,\delta_n}(\mathcal{E}^{2\tau})=\sum_{m \in \mathbb{Z}^d}|u_{m,\theta}(n)|^2=1, \quad n \in \mathbb{Z}^d.
    \end{equation*}
    Since the energies $E_m(\theta)$ are mutually distinct for $\theta \in \Theta^\tau$, the corresponding normalized eigenfunctions are mutually orthogonal. Hence the family $\{u_{m,\theta}\}_{m \in \mathbb{Z}^d}$ is complete in $\ell^2(\mathbb{Z}^d)$. Let $\theta \in \Theta^\tau_\gamma$. For each fixed $m \in \mathbb{Z}^d$, one has $T^m\theta \in \Theta^\tau_{\gamma_m}, \ \gamma_m=(2|m|+1)^{-\tau}\gamma$. By Theorem~\ref{goodeig}\ref{item:res1} and Lemma~\ref{branchlemma},
    \begin{equation*}
        \begin{split}
            |u_{m,\theta}(n)| &=|T_{-m}u_{T^m\theta}(n)|=|u_{T^m\theta}(n+m)| \\ 
            & \leqslant C_1(k,d)\gamma_m^{-(c_{3}+c_{4}k+c_{1}k)} (1+|n+m|)^{-k} \\
            & \leqslant C_{4}(k,d,\tau,\gamma,m)(1+|n|)^{-k}.  
        \end{split}
    \end{equation*}
    This proves that long-range operator $L_{V,\alpha,\theta}$ exhibits arithmetic $k$-PSL for every $\theta \in \Theta^\tau$. 
\end{proof}

\section{Arithmetic polynomial dynamical localization}\label{spdl}

In this section, we prove Theorem~\ref{main1}\ref{item:pdl}. More precisely, we establish the following result:

\begin{theorem}\label{pdl}
    Let $\alpha \in \mathrm{DC}_d(\kappa, \tau)$ and $k>d$. Suppose that, for every $\gamma \in (0,1)$ and every $E \in \Sigma_{V,\alpha}$ satisfying $\rho(E) \in \Theta_{\gamma}^{2\tau}$, the Schr\"odinger cocycle $(\alpha, S_E^V)$ is $C^k$-reducible and \eqref{red}--\eqref{est} hold with the same $\gamma$ and with $\tau$ replaced by $2\tau$ in \eqref{est}. Assume moreover that \eqref{cond} holds with $\tau$ replaced by $2\tau$ and that \eqref{fulluniformcond} holds. Then the long-range operator family $\{L_{V, \alpha, \theta}\}_{\theta \in \mathbb{T}}$ has arithmetic $k$-PDL. 
\end{theorem}

By Proposition \ref{APP}, for every $\theta \in \Theta^{\tau}$, the operator $L_{V,\alpha,\theta}$ has pure point spectrum and the shifted eigenfunctions form an orthonormal basis. It remains to prove the following estimate.  

\begin{theorem}\label{thm5.2}
    Let $\alpha \in \mathrm{DC}_d(\kappa, \tau)$ and $k>d$. Suppose that the assumptions of Theorem~\ref{pdl} hold. Then for any $p,q \in \mathbb{Z}^{d}$, 
    \begin{equation}\label{ineq}
        \int_{\Theta^{\tau}} \sup_{t \in \mathbb{R}} |\langle \delta_p, e^{-itL_{V,\alpha,\theta}}\delta_q \rangle| \, \mathrm{d}\mu(\theta) \leqslant \frac{C}{(1+|p-q|)^{k}} 
    \end{equation}
    where $C=C(k,d,\tau)>0$ and $\mu$ denotes the Haar measure on $\mathbb{T}$.
\end{theorem}

\begin{remark}
    The domain of integration in \eqref{ineq} is $\Theta^{\tau}$ instead of $\mathbb{T}$ because $\Theta^{\tau}$ is a full-measure subset of $\mathbb{T}$. 
\end{remark}
	
\begin{proof}
Fix $\theta \in \Theta^{\tau}$. By Theorem~\ref{APSL}, $\{u_{m,\theta}\}_{m\in\mathbb{Z}^{d}}$ forms an orthonormal basis of $\ell^{2}(\mathbb{Z}^{d})$ and $L_{V,\alpha,\theta}u_{m,\theta}=E_m(\theta)u_{m,\theta}$, the spectral theorem gives
\begin{equation*}
    \begin{split}
        \langle \delta_p, e^{-itL_{V,\alpha,\theta}}\delta_q \rangle=\langle \delta_p, e^{-itL_{V,\alpha,\theta}} \sum_{m \in \mathbb{Z}^d} \langle \delta_q, u_{m,\theta} \rangle u_{m,\theta} \rangle 
        = \sum_{m \in \mathbb{Z}^d} e^{itE_m(\theta)} \overline{u_{m,\theta}(p)} u_{m,\theta}(q),
    \end{split}
\end{equation*}
which in turn yields
\begin{equation*}
    \int_{\Theta^{\tau}} \sup_{t \in \mathbb{R}}|\langle \delta_p, e^{-itL_{V,\alpha,\theta}}\delta_q \rangle| \, \mathrm{d}\mu(\theta) 
    \leqslant \int_{\Theta^{\tau}} \sum_{m \in \mathbb{Z}^d} |\overline{u_{m,\theta}(p)} u_{m,\theta}(q)| \, \mathrm{d}\mu(\theta). 
\end{equation*}

For any $i \in \mathbb{N}$, we inductively define 
\begin{equation}\label{i,t}
    \begin{split}
        \gamma_i &= 10^{-i}, \quad \Phi_i=\Theta^{\tau}_{\gamma_i}, \quad h_i=\lfloor \gamma_i^{-c_1} \rfloor, \quad C_i =C_1(k,d)\gamma_i^{-(c_{3}+c_{4}k)},\\
        C_{i,h} &= \min \bigg\{1, \frac{100^{2\tau}}{\gamma_{i}} (1+h)^{-\min(c_{2}, c_{5}-2\tau)}\bigg\}, \\
        \Phi^0_i &= \big\{\theta \in \Phi_i \colon u_{\theta} \text{ is } (k,0,C_i,C_{i,0})\text{-good}\big\}, \\ 
        \Phi^h_i &= \bigg(\Phi_i {\setminus} \bigcup_{j=0}^{h-1} \Phi^j_i\bigg) \bigcap \big\{\theta \colon u_{\theta} \text{ is } (k,n^*,C_i,C_{i,h}) \text{-good with } |n^*|=h\big\}.
    \end{split}
\end{equation}
\begin{Claim}\label{unit} 
    For any $i \in \mathbb{N}$, $\{\Phi^j_i\}_{j=0}^{h_i}$ are mutually disjoint and $\Phi_i=\bigcup_{h=0}^{h_i} \Phi^h_i$. 
\end{Claim}
\begin{proof}
It is clear that the sets $\{\Phi^j_i\}_{j=0}^{h_i}$ are mutually disjoint and $\bigcup_{j=0}^{h_i}\Phi^j_i \subseteq \Phi_i$ by construction. On the other hand, by Theorem~\ref{goodeig}\ref{item:res2} and Lemma~\ref{branchlemma}, for any $\theta \in \Phi_i$, $L_{V,\alpha,\theta}$ has a $(k, n^*, C_i, C_{i,|n^*|})$-good eigenfunction $u_{\theta}$, where $|n^*| \leqslant \lfloor\gamma_i^{-c_1}\rfloor=h_i$. Letting $h=|n^*|$, it follows that $u_{\theta}$ is $(k, n^*, C_i, C_{i,h})$-good. If $h=0$, then $\theta \in \Phi^0_i$. Suppose that $h \geqslant 1$. If $\theta \in \bigcup_{j=0}^{h-1} \Phi^j_i$, then $\theta \in \bigcup_{j=0}^{h_i}\Phi^j_i$ already holds. Otherwise, the definition of $\Phi^h_i$ gives $\theta \in \Phi^h_i$. Consequently, $\Phi_i \subseteq \bigcup_{j=0}^{h_i}\Phi^j_i$. 
\end{proof}

Recall that $u_{m,\theta}=T_{-m}u_{T^m\theta}$, one has $u_{m,T^{-m}\theta}=T_{-m}u_{\theta}$. Hence 
\begin{equation*}
    |u_{m,T^{-m}\theta}(n)|=|T_{-m}u_{\theta}(n)|=|u_{\theta}(n+m)|. 
\end{equation*}

Let $i_0$ be the smallest integer such that $\gamma_{i_0} \leqslant \frac{1}{(1+|p-q|)^{k}}$. By the invariance of the Haar measure and by replacing $m$ with $-m$ in the sum, this yields 
\begin{equation*}
    \begin{split}
        \int_{\Theta^{\tau}} \sum_{m \in \mathbb{Z}^d} & |\overline{u_{m,\theta}(p)} u_{m,\theta}(q)| \, \mathrm{d}\mu(\theta) \\
        & =\int_{\Theta^{\tau}} \sum_{m \in \mathbb{Z}^d} |\overline{u_{\theta}(p-m)} u_{\theta}(q-m)| \, \mathrm{d}\mu(\theta) \\ 
        & =\bigg(\int_{\mathbb{T} \setminus \Phi_{i_0}}+\sum_{i=1}^{i_0} \int_{\bigcup_{h=0}^{h_i} \Phi^h_i {\setminus} \Phi_{i-1}} \bigg) \sum_{m \in \mathbb{Z}^d} |\overline{u_{\theta}(p-m)} u_{\theta}(q-m)| \, \mathrm{d}\mu(\theta) \\
        & \eqcolon (\mathrm{I})+(\mathrm{II}).
    \end{split}
\end{equation*}
By the Cauchy-Schwarz inequality and the choice of $i_{0}$, we have 
\begin{equation}\label{I}
    (\mathrm{I}) \leqslant \mu(\mathbb{T} \setminus \Phi_{i_0}) \lesssim_{\tau}\gamma_{i_{0}} \lesssim_{\tau}\frac{1}{(1+|p-q|)^{k}}. 
\end{equation}

To estimate $(\mathrm{II})$, we apply Tonelli's theorem and Claim~\ref{unit} to deduce that
\begin{equation*}
    \begin{split}
        (\mathrm{II}) & = \, \sum_{i=1}^{i_0} \sum_{m \in \mathbb{Z}^d} \int_{\bigcup_{h=0}^{h_i} \Phi^h_i {\setminus} \Phi_{i-1}} |\overline{u_{\theta}(p-m)} u_{\theta}(q-m)| \, \mathrm{d}\mu(\theta) \\
        & \leqslant \, \sum_{i=1}^{i_0}\sum_{h=0}^{h_{i}} \sum_{m \in \mathbb{Z}^{d}} \int_{\Phi^{h}_{i} {\setminus} \Phi_{i-1}} |\overline{u_{\theta}(p-m)}u_{\theta}(q-m)| \, \mathrm{d}\mu(\theta). 
    \end{split}
\end{equation*}
By \eqref{i,t}, for any $\theta \in \Phi^h_i \setminus \Phi_{i-1}$, there exists $n^*=n^*(\theta) \in \mathbb{Z}^d$ with $|n^*(\theta)|=h$ such that 
    \begin{equation*}
        |u_{\theta}(n)| \leqslant C_i\big((1+|n|)^{-k}+C_{i,h}(1+|n+n^*|)^{-k}\big), \quad \forall n \in \mathbb{Z}^d.
    \end{equation*}
    Hence
    \begin{equation*}
        |\overline{u_{\theta}(p-m)}u_{\theta}(q-m)| \leqslant C_i^2\big((\mathrm{III})_m+(\mathrm{IV})_m+(\mathrm{V})_m+(\mathrm{VI})_m\big),
    \end{equation*}
    where
    \begin{equation*}
        \begin{split}
            (\mathrm{III})_m &=(1+|p-m|)^{-k}(1+|q-m|)^{-k}, \\
            (\mathrm{IV})_m &=C_{i,h}^2(1+|p-m+n^*|)^{-k}(1+|q-m+n^*|)^{-k}, \\
            (\mathrm{V})_m &=C_{i,h}(1+|p-m|)^{-k}(1+|q-m+n^*|)^{-k}, \\
            (\mathrm{VI})_m &=C_{i,h}(1+|q-m|)^{-k}(1+|p-m+n^*|)^{-k}.
        \end{split}
    \end{equation*}
Thus, 
\begin{equation}\label{siu}
    \begin{split}
        \sum_{m \in \mathbb{Z}^d} \int_{\Phi^h_i \setminus \Phi_{i-1}} &|\overline{u_{\theta}(p-m)} \, u_{\theta}(q-m)| \, \mathrm{d}\mu(\theta) \\
        &\leqslant C_i^2 \int_{\Phi^h_i \setminus \Phi_{i-1}} \sum_{m \in \mathbb{Z}^d} \big( (\mathrm{III})_m+(\mathrm{IV})_m+(\mathrm{V})_m+(\mathrm{VI})_m \big) \, \mathrm{d}\mu(\theta).
    \end{split}
\end{equation}

We need the following discrete convolution inequality.
\begin{Lemma}\label{convolution}
    Let $s_1, s_2 > 0$ be constants such that $s_1+s_2>d$. There exists a constant $C>0$, depending only on $s_1, s_2$, and $d$, such that for any $a \in \mathbb{Z}^d$, 
    \begin{equation*}
        \sum_{n \in \mathbb{Z}^d} \frac{1}{(1+|a-n|)^{s_1}(1+|n|)^{s_2}} \leqslant C (1+|a|)^{-s} \Gamma(|a|),
    \end{equation*}
    where $s=\min \{s_{1},s_{2},s_{1}+s_{2}-d\}$ and 
    \begin{equation*}
        \Gamma(|a|)=\begin{cases}
            \ln (2+|a|), & \text{if } \max\{s_{1},s_{2}\}=d, \\
            1, & \text{otherwise}.
        \end{cases}
    \end{equation*}
\end{Lemma}
\begin{proof}
For convenience, denote $R=|a|$. When $R \leqslant 2$, the inequality holds trivially by choosing $C$ large enough, so we assume $R>2$.  

We decompose $\mathbb{Z}^d$ into three disjoint regions: $\Omega_1=\{n \in \mathbb{Z}^d \colon |n| \leqslant R/2\}$, $\Omega_2=\{n \in \mathbb{Z}^d \colon |a-n| \leqslant R/2\} \setminus \Omega_1$, and $\Omega_3=\mathbb{Z}^d \setminus (\Omega_1 \cup \Omega_2)$. Denote
\begin{equation*}
    \sum_{n \in A}(\cdots) \coloneq \sum_{n \in A} \frac{1}{(1+|a-n|)^{s_1}(1+|n|)^{s_2}}, \quad A \subseteq \mathbb{Z}^{d}.
\end{equation*}

For $n \in \Omega_1$, since $|n| \leqslant R/2$, by the triangle inequality we have $|a-n| \geqslant |a|-|n| \geqslant R-R/2=R/2$.
Therefore,
\begin{equation*}
    \sum_{n\in \Omega_{1}}(\cdots) \leqslant \frac{1}{(1+R/2)^{s_1}} \sum_{|n| \leqslant R/2} \frac{1}{(1+|n|)^{s_2}}.
\end{equation*}
Approximating the sum with $\int_1^{R/2} r^{d-1-s_2} \, \mathrm{d}r$, we have
\begin{equation*}
    \sum_{|n| \leqslant R/2} \frac{1}{(1+|n|)^{s_2}} \lesssim_{s_{2},d}
    \begin{cases}
        1, & \text{if } s_2 > d, \\
        \ln R, & \text{if } s_2 = d, \\
        R^{d-s_{2}}, & \text{if } s_{2}<d,
    \end{cases}
\end{equation*}
Thus
\begin{equation*}
    \sum_{n \in \Omega_{1}}(\cdots) \lesssim_{s_{1},s_{2},d} 
    \begin{cases}
        R^{-s_{1}}, & \text{if } s_2 > d,\\
        R^{-s_{1}}\ln R, & \text{if } s_2 = d,\\
        R^{d-s_{2}-s_{1}}, & \text{if } s_{2}<d.
    \end{cases}
\end{equation*}

By symmetry (substituting $m=a-n$), we swap $s_1$ and $s_2$. We have
\begin{equation*}
    \sum_{n \in \Omega_{2}}(\cdots) \lesssim_{s_{1},s_{2},d} \begin{cases}
        R^{-s_{2}}, & \text{if } s_1 > d, \\
        R^{-s_{2}}\ln R, & \text{if } s_1 = d, \\
        R^{d-s_{2}-s_{1}}, & \text{if } s_{1}<d.
    \end{cases}
\end{equation*}

For $n \in \Omega_{3}$, we further split this into two parts. For $R/2<|n| \leqslant 2R$, the sum is estimated as follows,
\begin{equation*}
    \sum_{ \substack{n\in \Omega_{3}, \, R/2<|n|\leqslant 2R}} (\cdots) \lesssim_{s_{1},s_{2},d} \frac{R^{d}}{(1+R/2)^{s_{1}} (1+R/2)^{s_{2}}} \lesssim_{s_{1},s_{2},d} R^{d-s_{1}-s_{2}}.
\end{equation*}
For $|n|>2R$, we have $|a-n| \geqslant |n|-|a|=|n|-R>|n|/2$, thus
\begin{equation*}
    \sum_{\substack{n \in \Omega_{3}, \, |n|>2R}} (\cdots) \leqslant \sum_{|n|>2R} \frac{2^{s_{1}}}{(1+|n|)^{s_{1}+s_{2}}} \lesssim_{s_{1},s_{2},d} R^{d-s_{1}-s_{2}}.
\end{equation*}
Combining the estimates for three regions finishes the proof.
\end{proof}
   
By Lemma \ref{convolution} and $C_{i,h} \leqslant 1$, uniformly in $\theta \in \Phi_{i}^{h} \setminus \Phi_{i-1}$, we immediately have
\begin{equation}\label{3&4}
    \sum_{m \in \mathbb{Z}^d}\big((\mathrm{III})_{m}+(\mathrm{IV})_{m}\big) \leqslant \frac{2\tilde{C}(k,d)}{(1+|p-q|)^{k}}, 
\end{equation}
and 
\begin{equation}\label{5&6}
    \sum_{m \in \mathbb{Z}^{d}}\big((\mathrm{V})_{m}+ (\mathrm{VI})_{m}\big) \leqslant \frac{2C_{i,h}\tilde{C}(k,d)(1+h)^{k}}{(1+|p-q|)^{k}}.
\end{equation}
Substituting \eqref{3&4} and \eqref{5&6} into \eqref{siu} implies
\begin{equation*}
	\begin{split}
		(\mathrm{II}) & \leqslant \, \sum_{i=1}^{i_{0}}\sum_{h=0}^{h_{i}} 2C^{2}_{i}\mu(\Phi^{h}_{i} \setminus \Phi_{i-1})\tilde{C}(k,d)\frac{1+C_{i,h}(1+h)^{k}}{(1+|p-q|)^{k}} \\ 
        & \lesssim_{k,d}\sum_{i=1}^{i_{0}} C^{2}_{i}\mu(\mathbb{T} \setminus \Phi_{i-1})\frac{1+\sup_{0 \leqslant h \leqslant h_{i}}C_{i,h}(1+h)^{k}}{(1+|p-q|)^{k}}. 
	\end{split}
\end{equation*} 
By Theorem \ref{goodeig}\ref{item:res2}, replacing $\tau$ by $2\tau$,  
\begin{equation*}
	\sum_{i=1}^{\infty}C^{2}_{i}\mu(\mathbb{T} \setminus \Phi_{i-1})\big(1+\sup_{0 \leqslant h \leqslant h_{i}}C_{i,h}(1+h)^{k}\big) \lesssim_{k,\tau}\sum_{i=1}^{\infty} \gamma_{i}^{1-c}<\infty. 
\end{equation*}
Thus, 
\begin{equation}\label{II}
	(\mathrm{II}) \lesssim_{k,d,\tau} \frac{1}{(1+|p-q|)^{k}}. 
\end{equation}
Combining \eqref{I} with \eqref{II} yields
\begin{equation*}
    \begin{split}
        \int_{\Theta^{\tau}} \sum_{m \in \mathbb{Z}^d}| \overline{u_{m,\theta}(p)} u_{m,\theta}(q)| \, \mathrm{d}\mu(\theta) \leqslant C(1+|p-q|)^{-k}, 
    \end{split}
\end{equation*}
where $C=C(k,d,\tau)>0$. Estimate \eqref{ineq} follows immediately. This finishes the proof. 
\end{proof}

\section{\texorpdfstring{$C^k$}{Ck}-reducibility of quasi-periodic cocycles}\label{Reducibility}

In this section, we prove the quantitative $C^k$-reducibility for $\mathrm{SL}(2,\mathbb{R})$-valued cocycle with Diophantine frequency in the local perturbation regime. As a corollary, we prove the long-range operator with polynomially decaying hopping $L_{V,\alpha,\theta}$ has arithmetic $k$-PSL and $k$-PDL.

\subsection{\texorpdfstring{$C^k$}{ck}-almost reducibility.}

For any $0<r'<r$, we denote
\begin{equation*}
	\epsilon_0'(r,r')=\frac{c}{(2\|A\|)^{\tilde{D}}}(r-r')^{D\tau}.
\end{equation*}
For any $m \geqslant 2, m \in \mathbb{Z}^{+}$, we define the sequences 
\begin{equation}\label{epsN}
	\epsilon_m=\frac{c/2}{(2\|A\|)^{\tilde{D}}m^{D\tau}}, \quad N_m=\frac{2|\ln\epsilon_m|}{\frac{1}{m}-\frac{1}{m^{1+s}}},
\end{equation}
where $s>0$. It is obvious that there exists $m_{0}=m_{0}(D,\tau,s)$ such that for any $m \geqslant m_0$, 
\begin{equation}\label{ms}
	\epsilon_{m} \leqslant \epsilon_{0}'\bigg(\frac{1}{m},\frac{1}{m^{1+s}}\bigg). 
\end{equation}
Let $l_j=L^{(1+s)^{j-1}}$ and $r_j=1/{l_j}$ for $j \in \mathbb{N}^+$, where $L=L(\kappa,\tau,d,\|A\|,\sigma,D,\tilde{D},s) \geqslant \max\{(2\|A\|)^{\tilde{D}}/c,m_0\}$ is an integer. Because $l_j$ may not be an integer, we replace $l_j$ with $\lfloor l_j\rfloor+1$ when necessary. 

The following $C^{k}$-almost reducibility theorem originates from \cite{MR4512234,cai2024quantitative}. We provide the whole proof because of the new definition of $\epsilon_{m}$ and our improved estimates.
\begin{theorem}\label{Ckalmost}
    Let $\alpha \in \mathrm{DC}_d(\kappa, \tau)$, $\sigma \in (0,1/6)$, $s \in (0,1-\frac{7}{2}\sigma)$ and $A \in \mathrm{SL}(2, \mathbb{R})$. Suppose that $f \in C^{k'}(\mathbb{T}^d, \mathrm{sl}(2, \mathbb{R}))$ with $k'>\frac{2}{\sigma}(1+s+\sigma)\tau+1$. Let $D>0$ be such that
    \begin{equation}\label{Drange}
        \max \bigg\{\frac{2(k'-1)}{(4-5\sigma)\tau}, \frac{2}{\sigma}\bigg\}<D<\frac{k'-1}{(1+\sigma+s)\tau}.
    \end{equation}
    Then there exists $\epsilon_1^*=\epsilon_1^*(\kappa,\tau,d,k',\|A\|,\sigma,s,D)>0$ such that if 
	\begin{equation}\label{pertur1}
		\|f\|_{k'} \leqslant \epsilon_1^*,
	\end{equation}
	then for each $j \geqslant 1$, let $\{f_{l_j}\}_{j \geqslant 1}$ be the analytic approximation sequence of $f$ given in Section~\ref{AA}, we have:
	\begin{enumerate}[label=$(\Alph*)_{j}$]
		\item \label{item:comega} ($C^{\omega}$-conjugacy) There exist $B_{l_j} \in C^{\omega}_{r_{j+1}}(2\mathbb{T}^d, \mathrm{SL}(2, \mathbb{R}))$, $f'_{l_j} \in C^{\omega}_{r_{j+1}}(\mathbb{T}^d, \mathrm{sl}(2, \mathbb{R}))$, and $A_{l_j} \in \mathrm{SL}(2, \mathbb{R})$ such that
		      \begin{equation*}
		      	B_{l_j}^{-1}(x+\alpha) Ae^{f_{l_j}(x)} B_{l_j}(x)=A_{l_j} e^{f'_{l_j}(x)},
		      \end{equation*}
		      with the estimates
		      \begin{equation*}
		      	\begin{split}
		      		&\|B_{l_j}\|_0, |B_{l_j}|_{r_{j+1}} \leqslant \epsilon_{l_j}^{-\sigma/2}, \quad |\deg B_{l_j}| \leqslant 8l_j|\ln{\epsilon_{l_j}}|,\\
                    &\|A_{l_j}\| \leqslant 2 \|A\|, \quad |f'_{l_j}|_{r_{j+1}} \leqslant \epsilon_{l_j}^{2-\frac{7}{2}\sigma}, 
		      	\end{split}
		      \end{equation*}
		      
		\item \label{item:ck} ($C^{k}$-conjugacy) Moreover, for any $k< \frac{k'-\sigma D\tau-1}{1+s}, k \in \mathbb{N}$, there exists $\overline{f}_{l_j} \in C^{k}(\mathbb{T}^d, \mathrm{sl}(2, \mathbb{R}))$ such that
		      \begin{equation*}
		      	B_{l_j}^{-1}(x+\alpha) Ae^{f(x)} B_{l_j}(x) = A_{l_j} e^{\overline{f}_{l_j}(x)},
		      \end{equation*}
		      with the estimates 
		      \begin{equation*}
		      	\|\overline{f}_{l_j}\|_{k}\leqslant 4l_{j}^{(1+s)k-k'+1+\sigma D\tau}, \quad \|\overline{f}_{l_j}\|_0 \leqslant \epsilon_{l_{j+1}}.
		      \end{equation*}
              
        \item \label{item:nre} {(Non-resonant step)} If for any $n \in \mathbb{Z}^d$ with $0<|n| \leqslant N_{l_j}$,  
            \begin{equation*}
                \|2\rho(A_{l_{j-1}}) - \langle n,\alpha \rangle\|_{\mathbb{R}/\mathbb{Z}} \geqslant \epsilon_{l_j}^\sigma,
            \end{equation*}
            then the conjugation $B_{l_{j}}$ takes the form $B_{l_j}=B_{l_{j-1}} \cdot \widetilde{B}_{l_{j-1}}$, with the estimates 
            \begin{equation*}
                \begin{split}
                    & |\widetilde{B}_{l_{j-1}}-\mathrm{Id}|_{r_{j+1}} \leqslant \epsilon_{l_j}^{1-\frac{7}{2}\sigma}, \quad \deg \widetilde{B}_{l_{j-1}}=0, \\ 
                    & \|A_{l_j}-A_{l_{j-1}}\| \leqslant 2\|A_{l_{j-1}}\|\epsilon_{l_j}. 
                \end{split}
            \end{equation*}
        
		\item \label{item:re} {(Resonant step)}
        If there exists $n^*_{l_j} \in \mathbb{Z}^d$ with $0<|n^*_{l_j}| \leqslant N_{l_j}$ such that
		      \begin{equation*}
		      	\|2\rho(A_{l_{j-1}})-\langle n^*_{l_j}, \alpha \rangle\|_{\mathbb{R}/\mathbb{Z}}<\epsilon_{l_j}^\sigma,
		      \end{equation*}
		      then the conjugation $B_{l_{j}}$ takes the form $B_{l_j} = B_{l_{j-1}} \cdot \widetilde{B}_{l_{j-1}}$, where
		      \begin{equation*}
		      	\widetilde{B}_{l_{j-1}}(x)=P_{l_j}^{-1} \cdot e^{Y_{l_j}(x)} \cdot R_{\langle n^*_{l_j}, x \rangle/2} \in C^{\omega}_{r_{j+1}}(2\mathbb{T}^d, \mathrm{SL}(2,\mathbb{R})),
		      \end{equation*}
		      with $P_{l_j} \in \mathrm{SL}(2,\mathbb{R})$ and $Y_{l_j} \in C^{\omega}_{r_{j+1}}(\mathbb{T}^d, \mathrm{sl}(2,\mathbb{R}))$ satisfying the estimates
		      \begin{equation*}
		      	\begin{split}
		      		& \deg \widetilde{B}_{l_{j-1}}=n^*_{l_j}, \quad \|P_{l_j}\| \leqslant 4\|A_{l_{j-1}}\|^{1/2} \kappa^{-1/2}|n^*_{l_j}|^{\tau/2}, \\
		      		& |Y_{l_j}|_{r_{j+1}}\leqslant  4 \|A_{l_{j-1}}\|^{1/2} \kappa^{-1/2} |n^*_{l_j}|^{\tau/2} \epsilon_{l_j}^{1/2}.
		      	\end{split} 
		      \end{equation*}
		    Moreover, there exist $t_j \in \mathbb{R}$ and $\nu_j \in \mathbb{C}$ such that
		      \begin{equation*}
		      	A_{l_j}=M^{-1} \exp\begin{pmatrix} it_j & \nu_j \\ \bar{\nu}_j & -it_j \end{pmatrix} M,
		      \end{equation*}
            with 
		      \begin{equation*}
		      	\begin{split}
		      	    &|t_{j}| \leqslant 4\epsilon_{l_j}^{\sigma}, \quad \|2\rho(A_{l_j})\|_{\mathbb{R}/\mathbb{Z}} \leqslant 4 \epsilon_{l_j}^\sigma, \\
                    &|\nu_j| \leqslant 64\|A_{l_{j-1}}\| \kappa^{-1} |n^*_{l_j}|^{\tau} \epsilon_{l_j} e^{-2\pi |n^*_{l_j}| r_j}.
		      	\end{split}
		      \end{equation*}
    \end{enumerate}
\end{theorem}

\begin{remark}
    For simplicity, we let $A_{l_{0}}=A$ and $B_{l_{0}}=\mathrm{Id}$.
\end{remark}

\begin{proof}
It follows from the assumptions $0<\frac{7}{2}\sigma+s<1$ and $k'-1>\frac{2}{\sigma}(1+\sigma+s)\tau$ that $D$ satisfying \eqref{Drange} exists.

We use induction to prove \ref{item:comega}, \ref{item:nre} and \ref{item:re}.

{\bf First step:} We choose $\epsilon_1^*$ to be sufficiently small such that $C'\epsilon_1^*<\epsilon_{l_1}$. By \eqref{pertur1}, we have
\begin{equation}\label{C'}
    C'\|f\|_{k'} \leqslant C'\epsilon_1^*<\epsilon_{l_1}. 
\end{equation}
Then, applying \eqref{aa} and \eqref{ms}, we obtain
\begin{equation*}
    |f_{l_1}|_{r_1} \leqslant \epsilon_{l_1} \leqslant \epsilon_0'(r_{1}, r_{2}),
\end{equation*}
which satisfies the assumption of Theorem \ref{KAM}.
By Theorem \ref{KAM}, there exist  $B_{l_1} \in C^\omega_{r_2}(2\mathbb{T}^d,\mathrm{SL}(2,\mathbb{R}))$, $A_{l_1} \in \mathrm{SL}(2,\mathbb{R})$ and $f_{l_1}' \in C^\omega_{r_2}(\mathbb{T}^d,\mathrm{sl}(2,\mathbb{R}))$ such that 
\begin{equation}\label{1stB}
	B_{l_{1}}^{-1}(x+\alpha) Ae^{f_{l_1}(x)} B_{l_1}(x)=A_{l_1} e^{f'_{l_1}(x)}. 
\end{equation}
Let $\widetilde{B}_{l_{0}}=B_{l_{1}}$. It is obvious that $B_{l_{1}}=B_{l_{0}} \cdot \widetilde{B}_{l_{0}}$. We distinguish two cases.

If the first step is non-resonant, then
\begin{equation*}
    |\widetilde{B}_{l_{0}}-\mathrm{Id}|_{r_{2}}<\epsilon_{l_{1}}^{1-\frac{7}{2}\sigma}, \quad \deg\widetilde{B}_{l_0}=0, \quad \|A_{l_{1}}-A\|\leqslant 2\|A\|\epsilon_{l_{1}}, \quad |f_{l_1}'|_{r_2} \leqslant \epsilon_{l_1}^{2-\frac{7}{2}\sigma}.
\end{equation*}
This proves $(C)_{1}$. 

If the first step is resonant, then there exist $P_{l_{1}} \in \mathrm{SL}(2,\mathbb{R})$, $Y_{l_{1}} \in C^{\omega}_{r_{2}}(\mathbb{T}^{d}, \mathrm{sl}(2,\mathbb{R}))$, $0<|n_{l_{1}}^{*}| \leqslant N_{l_{1}}$ with 
\begin{equation*}
	\begin{split}
        \|P_{l_1}\| \leqslant 4\|A\|^{1/2}\kappa^{-1/2}|n^*_{l_1}|^{\tau/2}, \quad |Y_{l_1}|_{r_{2}} \leqslant 4\|A\|^{1/2}\kappa^{-1/2}|n^*_{l_1}|^{\tau/2} \epsilon_{l_1}^{1/2}, 
	\end{split} 
\end{equation*}
such that $\widetilde{B}_{l_{0}}(x)=P_{l_1}^{-1} \cdot e^{Y_{l_1}(x)} \cdot R_{\langle n^*_{l_1}, x \rangle/2}$ and 
\begin{equation*}
    \deg B_{l_1}=\deg \widetilde{B}_{l_{0}}=n_{l_{1}}^{*}, \quad |f_{l_1}'|_{r_2} \leqslant \epsilon_{l_1}^{100}. 
\end{equation*}

Moreover, there exist $t_1 \in \mathbb{R}$ and $\nu_1 \in \mathbb{C}$ satisfying 
\begin{equation*}
	|t_{1}| \leqslant 4\epsilon_{l_1}^{\sigma}, \quad |\nu_1| \leqslant 64\|A\| \kappa^{-1} |n^*_{l_1}|^{\tau} \epsilon_{l_1} e^{-2\pi |n^*_{l_1}| r_1}.
\end{equation*}
such that 
	\begin{equation*}
		A_{l_1}=M^{-1}\exp\begin{pmatrix} 
              it_1 & \nu_1 \\ 
              \bar{\nu}_1 & -it_1 
            \end{pmatrix}M, \quad \|2\rho(A_{l_1})\|_{\mathbb{R}/\mathbb{Z}} \leqslant 4\epsilon_{l_1}^\sigma.
	\end{equation*}
This proves $(D)_{1}$.

Take the worst scenario into consideration, we have 
\begin{equation*}
    \begin{split}
        & \|B_{l_1}\|_0, |B_{l_1}|_{r_2} \leqslant 8\|A\|^{1/2}\kappa^{-1/2}|n^{*}_{l_1}|^{\tau/2} e^{\pi|n^{*}_{l_1}|r_2} \leqslant \epsilon_{l_1}^{-\sigma/2}, \\  
        & |\deg B_{l_1}| \leqslant N_{l_{1}} \leqslant 4l_1|\ln \epsilon_{l_1}|, \\ 
        & \|A_{l_1}\| \leqslant 2\|A\|,\quad |f'_{l_{1}}|_{r_{2}} \leqslant \epsilon_{l_{1}}^{2-\frac{7}{2}\sigma}. 
    \end{split}
\end{equation*}
This proves $(A)_{1}$.

{\bf Inductive step:} Assume that $(A)_{j}, (C)_{j}, (D)_{j}$ hold from the first step to $j$-th step, we will show $(A)_{j+1}, (C)_{j+1}, (D)_{j+1}$ hold for $(j+1)$-th step. By inductive assumption, we already have
\begin{equation*}
	B_{l_j}^{-1}(x+\alpha)Ae^{f_{l_j}(x)} B_{l_j}(x)=A_{l_j}e^{f'_{l_j}(x)}, 
\end{equation*}
then
\begin{equation*}
  \begin{split}
     B^{-1}_{l_j}(x+\alpha)Ae^{f_{l_{j+1}}(x)}B_{l_j}(x)
     & =A_{l_j} e^{f_{l_j}'(x)}B^{-1}_{l_j}(x)e^{-f_{l_j}(x)}e^{f_{l_{j+1}}(x)}B_{l_j}(x) \\
     & \eqcolon A_{l_j} e^{\tilde{f}_{l_j}(x)}.
  \end{split}
\end{equation*}
Thus, by $2-\frac{7}{2}\sigma>\frac{k'-1}{D\tau}-\sigma>1+s$, \eqref{aa}, $(A)_{j}$ and Baker-Campbell-Hausdorff formula, 
\begin{equation*}
  \begin{split}
     |\tilde{f}_{l_j}|_{r_{j+1}}
     & =|\ln(e^{f_{l_j}'}B^{-1}_{l_j}e^{-f_{l_j}}e^{f_{l_{j+1}}}B_{l_j})|_{r_{j+1}} \\ 
     & \leqslant 2(|f_{l_j}'|_{r_{j+1}}+|B^{-1}_{l_j}|_{r_{j+1}}|f_{l_{j+1}}-f_{l_j}|_{r_{j+1}}|B_{l_j}|_{r_{j+1}}) \\ 
     & \leqslant 2(\epsilon_{l_j}^{2-\frac{7}{2}\sigma} + \epsilon_{l_{j}}^{-\sigma} 2\epsilon_{l_{1}}l_{j}^{-(k'-1)}) \leqslant \epsilon_{l_{j+1}}  \leqslant \epsilon'_0(r_{j+1},r_{j+2}). 
  \end{split}
\end{equation*}
Now we apply Theorem \ref{KAM} for the cocycle $(\alpha,A_{l_j}e^{\tilde{f}_{l_j}})$, there exist $\widetilde{B}_{l_j} \in C^\omega_{r_{j+2}}(2\mathbb{T}^d,\mathrm{SL}(2,\mathbb{R}))$, $A_{l_{j+1}} \in \mathrm{SL}(2,\mathbb{R})$ and $f_{l_{j+1}}' \in C^\omega_{r_{j+2}}(\mathbb{T}^d,\mathrm{sl}(2,\mathbb{R}))$ such that 
\begin{equation*}
	\widetilde{B}_{l_j}^{-1}(x+\alpha) A_{l_j}e^{\tilde{f}_{l_j}(x)} \widetilde{B}_{l_j}(x)=A_{l_{j+1}} e^{f'_{l_{j+1}}(x)},  
\end{equation*}
which gives 
\begin{equation}\label{indeuctiveB}
	B_{l_{j+1}}^{-1}(x+\alpha) Ae^{f_{l_{j+1}}(x)} B_{l_{j+1}}(x)=A_{l_{j+1}} e^{f'_{l_{j+1}}(x)},  
\end{equation}
where $B_{l_{j+1}}=B_{l_j} \cdot \widetilde{B}_{l_j}$. We again distinguish two cases.

If $(j+1)$-th step is non-resonant, then Theorem \ref{KAM}(A) implies
\begin{equation*}
    |\widetilde{B}_{l_{j}}-\mathrm{Id}|_{r_{j+2}}<\epsilon_{l_{j+1}}^{1-\frac{7}{2}\sigma}, \ \deg \widetilde{B}_{l_j}=0, \ \|A_{l_{j+1}}-A_{l_{j}}\|\leqslant 2\|A_{l_{j}}\|\epsilon_{l_{j+1}}, \ |f_{l_{j+1}}'|_{r_{j+2}} \leqslant \epsilon_{l_{j+1}}^{2-\frac{7}{2}\sigma}.
\end{equation*}
This proves $(C)_{j+1}$.

If $(j+1)$-th step is resonant, then Theorem \ref{KAM}(B) implies that there exist $P_{l_{j+1}} \in \mathrm{SL}(2,\mathbb{R})$, $Y_{l_{j+1}}\in C^{\omega}_{r_{j+2}}(\mathbb{T}^{d}, \mathrm{sl}(2,\mathbb{R}))$, $0<|n_{l_{j+1}}^{*}|\leqslant N_{l_{j+1}}$ with 
\begin{equation*}
	\begin{split}
        \|P_{l_{j+1}}\| \leqslant 4\|A_{l_{j}}\|^{1/2}\kappa^{-1/2} |n^*_{l_{j+1}}|^{\tau/2}, \quad |Y_{l_{j+1}}|_{r_{j+2}} \leqslant  4\|A_{l_j}\|^{1/2}\kappa^{-1/2} |n^*_{l_{j+1}}|^{\tau/2} \epsilon_{l_{j+1}}^{1/2}, 
	\end{split} 
\end{equation*}
such that $\widetilde{B}_{l_{j}}(x)=P_{l_{j+1}}^{-1} \cdot e^{Y_{l_{j+1}}(x)} \cdot R_{\langle n^*_{l_{j+1}}, x \rangle/2}$ and 
\begin{equation*}
    \deg \widetilde{B}_{l_{j}}=n_{l_{j+1}}^{*}, \quad |f_{l_{j+1}}'|_{r_{j+2}} \leqslant \epsilon_{l_{j+1}}^{100}. 
\end{equation*}
Thus by $D^{-1}< \sigma/2$,
\begin{equation*}
    |\widetilde{B}_{l_{j}}|_{r_{j+2}}\leqslant 8\|A_{l_j}\|^{1/2}\kappa^{-1/2}|n^{*}_{l_{j+1}}|^{\tau/2} e^{\pi|n^{*}_{l_{j+1}}|r_{j+2}}\leqslant \epsilon_{l_{j+1}}^{-\sigma/4}.
\end{equation*}

Moreover, there exist $t_{j+1} \in \mathbb{R}$ and $\nu_{j+1} \in \mathbb{C}$ satisfying 
\begin{equation*}
	|t_{j+1}| \leqslant 4\epsilon_{l_{j+1}}^{\sigma}, \quad |\nu_{j+1}| \leqslant 64\|A_{l_j}\| \kappa^{-1} |n^*_{l_{j+1}}|^{\tau} \epsilon_{l_{j+1}} e^{-2\pi |n^*_{l_{j+1}}| r_{j+1}}.
\end{equation*}
such that 
	\begin{equation*}
		A_{l_{j+1}}=M^{-1} \exp\begin{pmatrix} 
        it_{j+1} & \nu_{j+1} \\ 
        \bar{\nu}_{j+1} & -it_{j+1} \end{pmatrix} M, \quad \|2\rho(A_{l_{j+1}})\|_{\mathbb{R}/\mathbb{Z}} \leqslant 4 \epsilon_{l_{j+1}}^\sigma.
	\end{equation*}
This proves $(D)_{j+1}$.

\begin{Lemma}\label{adjacentres}
    Let $0<\lambda<\min\{\sigma/2-D^{-1},D^{-1}\}$. For any adjacent resonant steps $j_{i}$ and $j_{i+1}$, we have $\epsilon_{l_{j_{i+1}}}<\epsilon_{l_{j_{i}}}^{\frac{\sigma D}{1+\lambda D}}$.
\end{Lemma}
\begin{proof}
    Theorem \ref{KAM} implies $\|2\rho(A_{l_{j_i}})\|_{\mathbb{R}/\mathbb{Z}} \leqslant 4\epsilon_{l_{j_{i}}}^{\sigma}$. Since every step between $j_i$ and $j_{i+1}$ is non-resonant, we also have $\|2\rho(A_{l_{j_{i+1}-1}})\|_{\mathbb{R}/\mathbb{Z}} \leqslant 8\epsilon_{l_{j_{i}}}^{\sigma}$. By the resonant condition at the step $j_{i+1}$, there exists $0<|n^{*}|\leqslant N_{l_{j_{i+1}}}$ such that $\|2\rho(A_{l_{j_{i+1}-1}})-\langle n^{*},\alpha\rangle\|_{\mathbb{R}/\mathbb{Z}}<\epsilon_{l_{j_{i+1}}}^{\sigma}$. However, by $\alpha\in \mathrm{DC}_{d}(\kappa,\tau)$ and \eqref{epsN}, we have 
    \begin{equation*}
        \|\langle n^{*},\alpha\rangle\|_{\mathbb{R}/\mathbb{Z}}\geqslant \frac{\kappa}{N_{l_{j_{i+1}}}^{\tau}}> 10\epsilon_{l_{j_{i+1}}}^{D^{-1}+\lambda}.
    \end{equation*}
    This yields a contradiction if $\epsilon_{l_{j_{i+1}}} \geqslant \epsilon_{l_{j_{i}}}^{\frac{\sigma D}{1+\lambda D}}$.
\end{proof}

To show $(A)_{j+1}$, we again consider two cases. If $(j+1)$-th is non-resonant, it follows from inductive assumption that
\begin{equation*}
    \|B_{l_{j+1}}\|_{0},|B_{l_{j+1}}|_{r_{j+2}} \leqslant 2|B_{l_{j}}|_{r_{j+1}} \leqslant 2\epsilon_{l_{j}}^{-\sigma/2} \leqslant \epsilon_{l_{j+1}}^{-\sigma/2}.
\end{equation*}
If $(j+1)$-th is resonant, we trace back to the resonant step $j^{*}$ which is closest to $(j+1)$. Since each step between $j^{*}$ and $(j+1)$ is non-resonant, Lemma \ref{adjacentres} implies
\begin{equation*}
    \|B_{l_{j+1}}\|_{0}, |B_{l_{j+1}}|_{r_{j+2}}\leqslant 2 |B_{l_{j^{*}}}|_{r_{j^{*}+1}} |\widetilde{B}_{l_{j}}|_{r_{j+2}} \leqslant 2\epsilon_{l_{j^{*}}}^{-\frac{\sigma}{2}} \epsilon_{l_{j+1}}^{-\frac{\sigma}{4}}  \leqslant 2\epsilon_{l_{j+1}}^{-\frac{1+\lambda D}{2D}-\frac{\sigma}{4}}\leqslant \epsilon_{l_{j+1}}^{-\frac{\sigma}{2}}.
\end{equation*}
If there is no such $j^*$, then 
\begin{equation*}
    \|B_{l_{j+1}}\|_0, |B_{l_{j+1}}|_{r_{j+2}} \leqslant 2|\widetilde B_{l_j}|_{r_{j+2}} \leqslant 2\epsilon_{l_{j+1}}^{-\sigma/4} \leqslant \epsilon_{l_{j+1}}^{-\sigma/2}. 
\end{equation*}
In both cases, we have 
\begin{equation*}
    \begin{split}
        & |\deg B_{l_{j+1}}|\leqslant 8l_{j}|\ln\epsilon_{l_{j}}|+N_{l_{j+1}} \leqslant 8l_{j+1}|\ln \epsilon_{l_{j+1}}|, \\ 
        & \|A_{l_{j+1}}\| \leqslant 2\|A\|,\quad |f'_{l_{j+1}}|_{r_{j+2}}\leqslant \epsilon_{l_{j+1}}^{2-\frac{7}{2}\sigma}. 
    \end{split}
\end{equation*}
This proves $(A)_{j+1}$. Thus we prove that \ref{item:comega}, \ref{item:nre} and \ref{item:re} hold for all $j \in \mathbb{N}$.

Finally, let's prove \ref{item:ck}. The direct calculation shows
\begin{equation*}
  \begin{split}
     B^{-1}_{l_{j}}(x+\alpha) Ae^{f(x)} B_{l_{j}}(x)
     & =A_{l_{j}}e^{f_{l_{j}}'(x)}B^{-1}_{l_{j}}(x)e^{-f_{l_{j}}(x)}e^{f(x)}B_{l_{j}}(x) \\
     & \eqcolon A_{l_{j}} e^{\overline{f}_{l_{j}}(x)}.
  \end{split}
\end{equation*} 
Fix $j$. Taking logarithms on the branch through $\mathrm{Id}$, we have
\begin{equation*}
    \overline{f}_{l_{j}}=B_{l_{j}}^{-1}\ln\big(e^{B_{l_{j}}f'_{l_{j}}B_{l_{j}}^{-1}}e^{-f_{l_{j}}}e^{f}\big)B_{l_{j}}.
\end{equation*}
By \eqref{aa} and $|B_{l_{j}}f'_{l_{j}}B_{l_{j}}^{-1}|_{r_{j+1}} \leqslant \epsilon_{l_{j}}^{2-\frac{9}{2}\sigma}$, the three-factor Baker-Campbell-Hausdorff map is uniformly Lipschitz in its third argument along the analytic approximations. Applying the Cauchy estimate \eqref{norm} to the logarithmic increments and taking a telescoping sum along the analytic approximations \eqref{aa} from $l_j$ to $+\infty$, we get 
\begin{equation*}
    \|\overline{f}_{l_{j}}-f'_{l_{j}}\|_{k} \leqslant C(k,d)l_{j}^{(1+s)k} \epsilon_{l_{j}}^{-\sigma} 2\epsilon_{l_{1}}l_{j}^{-(k'-1)}  \leqslant l_{j}^{(1+s)k+\sigma D\tau-(k'-1)}. 
\end{equation*}
Similarly, 
\begin{equation*}
     \|f_{l_{j}}'\|_{k} \leqslant C(k,d)(l_{j+1})^{k}|f_{l_{j}}'|_{r_{j+1}}  \leqslant C(k,d)(l_{j})^{(1+s)k}\epsilon_{l_{j}}^{2-\frac{7}{2}\sigma} \leqslant l_{j}^{(1+s)k-(2-\frac{7}{2}\sigma)D\tau}. 
\end{equation*}
Thus, by $2-\frac{7}{2}\sigma>\frac{k'-1}{D\tau}-\sigma>1+s$ again and $k<\frac{k'-\sigma D\tau-1}{1+s}$, we have 
\begin{equation*}
    \begin{split}
        \|\overline{f}_{l_{j}}\|_{k} & \leqslant 2(\|f_{l_{j}}'\|_{k} + \|\overline{f}_{l_{j}}-f_{l_{j}}'\|_{k}) \\ 
        & \leqslant 2(l_{j}^{(1+s)k-(2-\frac{7}{2}\sigma)D\tau}+l_{j}^{(1+s)k+\sigma D\tau-(k'-1)}) \leqslant 4l_{j}^{(1+s)k-k'+1+\sigma D\tau}.
    \end{split}
\end{equation*}
And
\begin{equation*}
  \begin{split}
     \|\overline{f}_{l_{j}}\|_0
     & \leqslant 2(\|f_{l_{j}}'\|_0 + \|B^{-1}_{l_{j}}\|_0\|f-f_{l_{j}}\|_0\|B_{l_{j}}\|_0) \\
     & \leqslant 2(\epsilon_{l_{j}}^{2-\frac{7}{2}\sigma}+\epsilon_{l_{j}}^{-\sigma} 2\epsilon_{l_{1}} l_{j}^{-(k'-1)}) \leqslant \epsilon_{l_{j}}^{\frac{k'-1}{D\tau} -\sigma} \leqslant \epsilon_{l_{j}}^{1+s} \leqslant \epsilon_{l_{j+1}}. 
  \end{split}
\end{equation*}
This proves $(B)_{j}$.
\end{proof}

\subsection{\texorpdfstring{$C^k$}{ck}-reducibility}
By repeatedly applying the almost reducibility theorem (Theorem \ref{Ckalmost}), we establish quantitative $C^{k}$-reducibility for $\mathrm{SL}(2,\mathbb{R})$-valued cocycles whose rotation number is Diophantine, in the regime of small perturbations.

\begin{theorem}\label{Ckreducibility}
    Let $\alpha \in \mathrm{DC}_d(\kappa, \tau)$, $\sigma \in (0,1/9)$, $s \in (0,1-\frac{7}{2}\sigma)$,  and $A \in \mathrm{SL}(2, \mathbb{R})$. Suppose that $f \in C^{k'}(\mathbb{T}^d, \mathrm{sl}(2, \mathbb{R}))$ and $\rho(\alpha,Ae^f) \in \Theta_\gamma^{2\tau}$, where $k'>\frac{18}{\sigma^2}(1+s+\sigma)\tau+1$. Let $D>0$ be a constant satisfying
    \begin{equation*}
        \max\bigg\{\frac{2(k'-1)}{(4-5\sigma)\tau},\frac{18}{\sigma^2}\bigg\}<D<\min\bigg\{\frac{k'-1}{(1+\sigma+s)\tau},\frac{k'-20}{8\sigma\tau}\bigg\}.
    \end{equation*} 
    Then for any $k \leqslant \frac{k'-8\sigma D\tau-4}{4(1+s)^2}, k \in \mathbb{N}$, there exists $\epsilon^*_2=\epsilon^*_2(\kappa,\tau,d,k',\|A\|,\sigma,s,D)>0$ such that if 
	\begin{equation*}
		\|f\|_{k'} \leqslant \epsilon^*_2,
	\end{equation*}
    there exist $\bar{A} \in \mathrm{SL}(2,\mathbb{R})$, $B \in C^{k}(2\mathbb{T}^d, \mathrm{SL}(2,\mathbb{R}))$, and $\widetilde{B} \in C^{k}(2\mathbb{T}^d, \mathrm{SL}(2, \mathbb{R}))$, $Y \in C^{k}(\mathbb{T}^d, \mathrm{sl}(2,\mathbb{R}))$, $n^* \in \mathbb{Z}^d$ such that $B$ takes the form
    \begin{equation*}
        B(x)=\widetilde{B}(x)\cdot R_{\frac{\langle n^*,x \rangle}{2}}\cdot e^{Y(x)}
    \end{equation*} 
    and     
    \begin{equation*}
		B^{-1}(x+\alpha)Ae^{f(x)}B(x)=\bar{A}=M^{-1}\exp\begin{pmatrix} 
                    it & \nu \\ 
                    \bar{\nu} & -it 
                \end{pmatrix}M
	\end{equation*}
	with the following estimates 
    \begin{equation*}
		\begin{split}
			& |n^{*}| \leqslant \gamma^{-\frac{2}{\sigma(k'-1)}},  \quad \|Y\|_k \leqslant (1+|n^{*}|)^{-\frac{k'}{2}}, \\
            & \|\widetilde{B}\|_{0}\|\widetilde{B}\|_{k} \leqslant \gamma^{-\frac{2\tau k'}{\sigma^2(k'-1)^2}}, \quad |\deg \widetilde{B}| \leqslant \gamma^{-\frac{10\tau}{\sigma^2(k'-1)^2}}, \\
            & |\nu| \leqslant (1+|n^*|)^{-(\frac{k'}{2}+2\tau)}, \, \text{if } n^{*} \neq 0, \quad \|2\rho(\bar{A})\|_{\mathbb{R}/\mathbb{Z}} \geqslant \frac{\gamma}{4^{\tau}(1+|n^*|)^{2\tau}}.  
		\end{split}
	\end{equation*}
\end{theorem}

\begin{remark}
    In order to obtain the PDL, the conditions in Theorem~\ref{goodeig}\ref{item:res2} require stronger estimates after the reduction, which leads to a stronger initial regularity requirement on $k'$ and lower remaining regularity $k$ in the reduction.
\end{remark}

\begin{proof}
Choose $\epsilon_{2}^{*}<\epsilon_{1}^{*}$. The assumptions of Theorem \ref{Ckalmost} are satisfied.  By Theorem~\ref{Ckalmost}\ref{item:comega} and \ref{item:ck}, for $j \geqslant 1$, there exist sequences $B_{l_j} \in C^{\omega}_{r_{j+1}}(2\mathbb{T}^d, \mathrm{SL}(2, \mathbb{R}))$, $A_{l_j} \in \mathrm{SL}(2,\mathbb{R})$, and $\overline{f}_{l_j} \in C^{k}(\mathbb{T}^d, \mathrm{sl}(2,\mathbb{R}))$ such that
\begin{equation*}
    B_{l_j}^{-1}(x+\alpha)A e^{f(x)} B_{l_j}(x)=A_{l_j} e^{\overline{f}_{l_j}(x)},
\end{equation*}
with the following estimates:
\begin{equation}\label{estb1}
    \begin{split}
        & \|B_{l_j}\|_0, |B_{l_j}|_{r_{j+1}} \leqslant \epsilon_{l_j}^{-\sigma/2}, \quad |\deg B_{l_j}| \leqslant 8l_j |\ln\epsilon_{l_j}|, \\
        & \|\overline{f}_{l_j}\|_0 \leqslant \epsilon_{l_{j+1}}, 
        \quad \|A_{l_j}\| \leqslant 2\|A\|.
    \end{split}
\end{equation}
We claim that there are at most finitely many resonant steps. Otherwise, there exist arbitrarily large resonant indices $j$ such that 
\begin{equation}\label{contradiction}
    \|2\rho(A_{l_{j-1}})-\langle n^{*}_{l_{j}}, \alpha \rangle\|_{\mathbb{R}/\mathbb{Z}}<\epsilon_{l_{j}}^{\sigma}.
\end{equation}
Since $\rho(\alpha, Ae^f) \in \Theta_\gamma^{2\tau}$, it follows from~\eqref{rot2} that for any $n \in \mathbb{Z}^d$,
\begin{equation}\label{rotdc}
    \begin{split}
        \|2\rho(\alpha, A_{l_{j-1}} e^{\overline{f}_{l_{j-1}}})-\langle n, \alpha \rangle\|_{\mathbb{R}/\mathbb{Z}} 
        & =\|2\rho(\alpha, Ae^f)-\langle \deg B_{l_{j-1}}, \alpha \rangle-\langle n, \alpha \rangle\|_{\mathbb{R}/\mathbb{Z}} \\
        & \geqslant \frac{\gamma}{(|n+\deg B_{l_{j-1}}|+1)^{2\tau}}.
    \end{split}
\end{equation}
By \eqref{estb1} and \eqref{rotdc}, for sufficiently large $j$ (depending on $\gamma$), we have
\begin{equation*}
    \|2\rho(\alpha, A_{l_{j-1}} e^{\overline{f}_{l_{j-1}}})-\langle n^{*}_{l_{j}},\alpha \rangle\|_{\mathbb{R}/\mathbb{Z}} 
    \geqslant \frac{\gamma}{(N_{l_{j}}+8l_{j-1} |\ln\epsilon_{l_{j-1}}|+1)^{2\tau}}  \geqslant 2\epsilon_{l_{j}}^{\sigma}.
\end{equation*}
Combining the above estimate with \eqref{con} and \eqref{estb1}, we obtain
\begin{equation*}
    \begin{split}
        & \quad \ \|2\rho( A_{l_{j-1}})-\langle n^{*}_{l_{j}}, \alpha \rangle \|_{\mathbb{R}/\mathbb{Z}} \\
        & \geqslant \|2\rho(\alpha, A_{l_{j-1}} e^{\overline{f}_{l_{j-1}}})-\langle n^{*}_{l_{j}}, \alpha \rangle \|_{\mathbb{R}/\mathbb{Z}}-\|2\rho(\alpha, A_{l_{j-1}} e^{\overline{f}_{l_{j-1}}})-2\rho(A_{l_{j-1}})\|_{\mathbb{R}/\mathbb{Z}} \\
        & \geqslant 2\epsilon_{l_{j}}^{\sigma}-2\tilde{c} \,\epsilon_{l_{j}}^{\frac{1}{2}} \geqslant \epsilon_{l_{j}}^{\sigma}, 
    \end{split}
\end{equation*}
which contradicts to \eqref{contradiction}.

Now we prove the reducibility. If there are no resonant steps, set $n^*=0$, $\widetilde{B}=\mathrm{Id}$. Then the products of non-resonant conjugacies converge in $C^k$ to a conjugacy $B=e^Y$. Moreover, by \eqref{estb1}, $\|\bar{A}\| \leqslant 2\|A\|$. Hence the canonical elliptic logarithm of $\bar{A}$ satisfies $|\nu| \lesssim_{\|A\|} 1$. All the remaining estimates follow directly from the non-resonant estimates. 

Assume that there are $m \in \mathbb{N}^+$ resonant steps, we let $j_m$ be the last resonant step and denote by $n^*=n^*_{l_{j_m}}$ the resonant site for short. Next, we will first consider the general case where $m \geqslant 2$. 

Since $j_m>1$, by Theorem~\ref{Ckalmost}\ref{item:comega}, we have
\begin{equation*}
	B_{l_{j_m-1}}^{-1}(x+\alpha) A e^{f_{l_{j_m-1}}(x)} B_{l_{j_m-1}}(x) = A_{l_{j_m-1}} e^{f'_{l_{j_m-1}}(x)}. 
\end{equation*}
Let $A_{l_{j_m-1}} e^{\tilde{f}_{l_{j_m-1}}(x)} \coloneq A_{l_{j_m-1}} e^{f_{l_{j_m-1}}'(x)}B^{-1}_{l_{j_m-1}}(x)e^{-f_{l_{j_m-1}}(x)}e^{f_{l_{j_m}}(x)}B_{l_{j_m-1}}(x)$. Then
\begin{equation*}
     B^{-1}_{l_{j_m-1}}(x+\alpha) A e^{f_{l_{j_m}}(x)} B_{l_{j_m-1}}(x)
     = A_{l_{j_m-1}} e^{\tilde{f}_{l_{j_m-1}}(x)}.
\end{equation*}
    Lemma~\ref{adjacentres} excludes two consecutive resonant steps, hence the $(j_{m}-1)$-th step is non-resonant. By Theorem \ref{Ckalmost}\ref{item:comega} and \eqref{estb1},
\begin{equation*}
  \begin{split}
     |\tilde{f}_{l_{j_m-1}}|_{r_{{j_m}}}
     & \leqslant 2(|f_{l_{j_m-1}}'|_{r_{{j_m}}}+|B^{-1}_{l_{j_m-1}}|_{r_{j_m}}|f_{l_{{j_m}}}-f_{l_{j_m-1}}|_{r_{j_m}}|B_{l_{j_m-1}}|_{r_{j_m}}) \\
     & \leqslant 2(\epsilon_{l_{j_m-1}}^{2-\frac{7}{2}\sigma}+\epsilon_{l_{j_m-1}}^{-\sigma} 2\epsilon_{l_{1}}l_{j_m-1}^{-(k'-1)}) \\
     & \leqslant 4\epsilon_{l_{j_m-1}}^{\min(2-\frac{7}{2}\sigma,\frac{k'-1}{D\tau}-\sigma)} \leqslant \epsilon_{l_{j_m}}. 
  \end{split}
\end{equation*}
Now we apply Theorem \ref{KAM} to the cocycle $(\alpha,A_{l_{j_m-1}}e^{\tilde{f}_{l_{j_m-1}}})$, there exist $\widetilde{B}_{l_{j_m-1}} \in C^\omega_{r_{j_m+1}}(2\mathbb{T}^d,\mathrm{SL}(2,\mathbb{R}))$, $A_{l_{j_m}} \in \mathrm{SL}(2,\mathbb{R})$ and $f_{l_{j_m}}' \in C^\omega_{r_{j_m+1}}(\mathbb{T}^d,\mathrm{sl}(2,\mathbb{R}))$ such that 
\begin{equation*}
	\widetilde{B}_{l_{j_m-1}}^{-1}(x+\alpha) A_{l_{j_m-1}}e^{\tilde{f}_{l_{j_m-1}}(x)} \widetilde{B}_{l_{j_m-1}}(x)=A_{l_{j_m}} e^{f'_{l_{j_m}}(x)},  
\end{equation*}
which gives 
\begin{equation*}
	B_{l_{j_m}}^{-1}(x+\alpha) Ae^{f_{l_{j_m}}(x)} B_{l_{j_m}}(x)=A_{l_{j_m}} e^{f'_{l_{j_m}}(x)},  
\end{equation*}
where 
\begin{equation*}
    B_{l_{j_m}}=B_{l_{j_m-1}} \cdot \widetilde{B}_{l_{j_m-1}} \in C^{\omega}_{r_{j_m+1}}(2\mathbb{T}^d, \mathrm{SL}(2,\mathbb{R})). 
\end{equation*} Since $j_m$-th step is the last resonant step, then Theorem \ref{KAM}(B) implies that there exist $P_{l_{j_m}}\in \mathrm{SL}(2,\mathbb{R})$, $Y_{l_{j_m}}\in C^{\omega}_{r_{j_m+1}}(\mathbb{T}^{d}, \mathrm{sl}(2,\mathbb{R}))$, $0<|n^{*}|\leqslant N_{l_{j_m}}$ with 
\begin{equation}\label{PY}
	\begin{split}
        \|P_{l_{j_m}}\| \leqslant 4\|A_{l_{j_m-1}}\|^{1/2}{\kappa}^{-1/2}|n^*|^{\tau/2}, \quad |Y_{l_{j_m}}|_{r_{j_m+1}} \leqslant 4\|A_{l_{j_m-1}}\|^{1/2}{\kappa}^{-1/2}|n^*|^{\tau/2} \epsilon_{l_{j_m}}^{1/2}, 
	\end{split} 
\end{equation}
such that 
\begin{equation*}
    \widetilde{B}_{l_{j_{m}-1}}(x)=P_{l_{j_{m}}}^{-1} \cdot e^{Y_{l_{j_{m}}}(x)}\cdot R_{\langle n^*,x \rangle/2}
\end{equation*} 
with $\deg \widetilde{B}_{l_{j_m-1}}=n^{*}$ and $|f_{l_{j_{m}}}'|_{r_{j_{m}+1}} \leqslant \epsilon_{l_{j_{m}}}^{100}$.  
Thus by \eqref{estb1}, we have 
\begin{equation*}
    \begin{split}
        |\tilde{f}_{l_{j_{m}}}|_{r_{j_{m}+1}} & \leqslant 2(|f_{l_{j_m}}'|_{r_{{j_m}+1}}+|B^{-1}_{l_{j_m}}|_{r_{{j_m}+1}}|f_{l_{{j_m}+1}}-f_{l_{j_m}}|_{r_{{j_m}+1}}|B_{l_{j_m}}|_{r_{{j_m}+1}}) \\
        & \lesssim \epsilon_{l_{j_{m}}}^{100}+\epsilon_{l_{j_m}}^{-\sigma} 2\epsilon_{l_{1}}l_{j_m}^{-(k'-1)} \\
        & \lesssim l_{j_{m}}^{-(k'-1-\sigma D\tau)} <\epsilon_{l_{j_{m}+1}}
    \end{split}
\end{equation*}

By Theorem \ref{KAM}(A) and $D>\frac{18}{\sigma^2}$,
\begin{equation}\label{nrej0}
    |\widetilde{B}_{l_{j_{m}}}-\mathrm{Id}|_{r_{j_{m}+2}} \leqslant \epsilon_{l_{j_{m}+1}}^{-2\sigma-\sigma^2/6} |\tilde{f}_{l_{j_{m}}}|_{r_{j_{m}+1}} \lesssim \epsilon_{l_{j_m}}^{\frac{k'-1-\sigma D\tau}{D\tau}-(2\sigma+\sigma^2/6)(1+s)}, 
\end{equation}
we let $Y \in C^{k}(\mathbb{T}^{d}, \mathrm{sl}(2,\mathbb{R}))$ be such that 
\begin{equation}\label{expY}
    e^{Y}=\prod_{j=j_{m}}^{\infty}\widetilde{B}_{l_{j}} \in C^{k}(\mathbb{T}^{d},\mathrm{SL}(2,\mathbb{R})).
\end{equation}
Furthermore, we denote $\widetilde{B}(x)=B_{l_{j_{m}-1}}(x) P_{l_{j_{m}}}^{-1} e^{Y_{l_{j_{m}}}(x)}$ and
\begin{equation*}
    B(x)=\widetilde{B}(x) \cdot R_{\langle n^*,x \rangle/2} \cdot e^{Y(x)}. 
\end{equation*}

Since all the steps after $j_m$ are non-resonant, let $\bar{A}=\lim_{j \rightarrow \infty}A_{l_j} \in \mathrm{SL}(2,\mathbb{R})$, we obtain 
\begin{equation*}
    B^{-1}(x+\alpha)Ae^{f(x)}B(x)=\bar{A}=M^{-1}\exp\begin{pmatrix} 
                it & \nu \\ 
                \bar{\nu} & -it 
            \end{pmatrix}M.
\end{equation*}

In the following, we will evaluate the quantitative estimates for $B$ and $\bar{A}$ via the relation of $\gamma$ and $l_{j_{m}}$. Let $\delta=\frac{\sigma}{100k'D\tau}$. For every resonant index $j \in \{j_{1},\cdots,j_{m}\}$, we have
\begin{equation}\label{Nlj0}
	1+|n^{*}_{l_{j}}| \leqslant 8l_j |\ln \epsilon_{l_j}| \leqslant l_{j}^{1+\delta}. 
\end{equation}
Since $j_{m}$ is the last resonant step, that is
\begin{equation*}
    \|2\rho(A_{l_{j_{m}-1}})-\langle n^{*},\alpha\rangle\|_{\mathbb{R}/\mathbb{Z}} <\epsilon_{l_{j_{m}}}^{\sigma},
\end{equation*}
combining the above inequality with \eqref{con} implies that
\begin{equation}\label{lastle}
    \begin{split}
        &\quad \ \|2\rho(\alpha, A_{l_{j_{m}-1}} e^{\overline{f}_{l_{j_{m}-1}}})-\langle n^{*}, \alpha \rangle \|_{\mathbb{R}/\mathbb{Z}} \\
        &\leqslant \|2\rho(A_{l_{j_{m}-1}})-\langle n^{*},\alpha\rangle\|_{\mathbb{R}/\mathbb{Z}}+\| 2\rho(\alpha, A_{l_{j_{m}-1}} e^{\overline{f}_{l_{j_{m}-1}}}) - 2\rho(A_{l_{j_{m}-1}})\|_{\mathbb{R}/\mathbb{Z}} \\
        &\leqslant \epsilon_{l_{j_{m}}}^{\sigma}+2\tilde{c}\epsilon_{l_{j_{m}}}^{1/2} \leqslant 2\epsilon_{l_{j_{m}}}^{\sigma}.
    \end{split}
\end{equation}
On the other hand, \eqref{rotdc} shows 
\begin{equation}\label{lastge}
    \|2\rho(\alpha, A_{l_{j_{m}-1}} e^{\overline{f}_{l_{j_{m}-1}}}) - \langle n^{*}, \alpha \rangle\|_{\mathbb{R}/\mathbb{Z}} \geqslant \frac{\gamma}{(N_{l_{j_{m}}}+|\deg B_{l_{j_{m}-1}}|+1)^{2\tau}}.
\end{equation}
It follows from \eqref{lastle} and \eqref{lastge} that
\begin{equation}\label{gamma}
    \gamma \leqslant 2\epsilon_{l_{j_{m}}}^{\sigma} (N_{l_{j_{m}}}+ 8l_{j_{m}-1} |\ln \epsilon_{l_{j_{m}-1}}|+1)^{2\tau} \leqslant 2^{2\tau+1}\epsilon_{l_{j_{m}}}^{\sigma}N_{l_{j_{m}}}^{2\tau} \leqslant l_{j_{m}}^{-(\sigma D-2-2\delta)\tau},
\end{equation}
which implies
\begin{equation}\label{lj0}
    l_{j_m} \leqslant \gamma^{-\frac{1}{(\sigma D-2-2\delta)\tau}}.
\end{equation}
Therefore, it follows from \eqref{Nlj0}, \eqref{lj0} and $D>\frac{2(k'-1)}{(4-5\sigma)\tau}$ that 
\begin{equation}\label{lastres}
    |n^{*}|\leqslant N_{l_{j_{m}}} \leqslant l_{j_{m}}^{1+\delta} \leqslant {\gamma}^{-\frac{1+\delta}{(\sigma D-2-2\delta)\tau}} \leqslant \gamma^{-\frac{2}{\sigma(k'-1)}}.
\end{equation}

By \eqref{expY}, \eqref{norm}, and \eqref{nrej0}, we have  
\begin{equation*}
    \|Y\|_{k} \leqslant 2\sum_{j = j_{m}}^{\infty} \|\widetilde{B}_{l_{j}} - \mathrm{Id}\|_{k} 
    \lesssim 4k! \, (l_{j_{m}+2})^{k} \, \epsilon_{l_{j_m}}^{\frac{k'-1-\sigma D\tau}{D\tau}-(2\sigma+\sigma^2/6)(1+s)}.
\end{equation*}
Using $k\leqslant \frac{k'-8\sigma D\tau-4}{4(1+s)^2}$ and \eqref{Nlj0}, it follows that 
\begin{equation*}
    \begin{split}
        \|Y\|_{k} & \lesssim l_{j_{m}}^{(1+s)^2k-(k'-1-\sigma D\tau)+(2+\sigma/6)(1+s)\sigma D\tau} \lesssim l_{j_m}^{-\frac{3k'}{4}+\frac{6+\sigma+\sigma s+12s}{6}\sigma D\tau} \\ 
        & \lesssim (1 + |n^{*}|)^{-\frac{9k'-(12+2\sigma+2\sigma s+24s)\sigma D\tau}{12(1+\delta)}} 
        \leqslant (1+|n^{*}|)^{-\frac{k'}{2}}.
    \end{split}
\end{equation*}

By \eqref{PY} and \eqref{lastres}, we obtain 
\begin{equation}\label{P}
    \|P_{l_{j_{m}}}\| \lesssim |n^*|^{\frac{\tau}{2}} \lesssim \gamma^{-\frac{1+\delta}{2(\sigma D-2-2\delta)}}.
\end{equation} 

Let $j_{i}$ denote the resonant steps, with $n^{*}_{l_{j_i}}$ the associated resonant site ($2 \leqslant i \leqslant m$). On one hand, the resonance condition gives 
\begin{equation*}
    \|2\rho(A_{l_{j_{i}-1}})-\langle n^{*}_{l_{j_i}}, \alpha \rangle\|_{\mathbb{R}/\mathbb{Z}}<\epsilon_{l_{j_i}}^\sigma. 
\end{equation*}
Together with the Diophantine condition
\begin{equation*}
    \|\langle n^{*}_{l_{j_i}}, \alpha \rangle\|_{\mathbb{R}/\mathbb{Z}}>\kappa|n_{l_{j_i}}^{*}|^{-\tau},
\end{equation*}
we deduce that
\begin{equation*}
    \|2\rho(A_{l_{j_i-1}})\|_{\mathbb{R}/\mathbb{Z}}>\frac{2\kappa}{3} |n^{*}_{l_{j_i}}|^{-\tau}.
\end{equation*}
On the other hand, by Theorem~\ref{Ckalmost}\ref{item:re} we have $\|2\rho(A_{l_{j_{i-1}}})\|_{\mathbb{R}/\mathbb{Z}} \leqslant 4\epsilon_{l_{j_{i-1}}}^\sigma$, and Theorem~\ref{Ckalmost}\ref{item:nre} further yields  
\begin{equation*}
    \|2\rho(A_{l_{j_i-1}})\|_{\mathbb{R}/\mathbb{Z}} \leqslant 8\epsilon_{l_{j_{i-1}}}^\sigma \leqslant 8(\kappa|n_{l_{j_{i-1}}}^{*}|^{-\tau})^{18}.
\end{equation*}
Thus,  
\begin{equation*} 
    \quad |n_{l_{j_{i}}}^{*}|^{-\tau}<12\kappa^{17}|n_{l_{j_{i-1}}}^{*}|^{-18\tau} \leqslant |n_{l_{j_{i-1}}}^{*}|^{-17\tau},
\end{equation*}
Choose $0<\lambda<\min\{\frac{\sigma^2}{18}-D^{-1},D^{-1}\}$ in Lemma~\ref{adjacentres} and above inequalities imply  
\begin{equation}\label{lji}
    l_{j_{i-1}} \leqslant l_{j_i}^{\frac{1+\lambda D}{\sigma D}}, \quad 
    |n_{l_{j_{i-1}}}^{*}| \leqslant |n_{l_{j_{i}}}^{*}|^{\frac{1}{17}}.
\end{equation}

From Theorem~\ref{Ckalmost}\ref{item:comega} and \eqref{norm},  
\begin{equation*}
    \|B_{l_{j_{m-1}}}\|_0 \, \|B_{l_{j_{m-1}}}\|_{k} \leqslant \epsilon_{l_{j_{m-1}}}^{-\sigma/2} k! \,(l_{j_{m-1}+1})^{k} |B_{l_{j_{m-1}}}|_{r_{j_{m-1}+1}} \leqslant k! \,(l_{j_{m-1}+1})^{k}\epsilon_{l_{j_{m-1}}}^{-\sigma}.
\end{equation*}
Using $k \leqslant \frac{k'-8\sigma D\tau-4}{4(1+s)^2}$, and \eqref{lj0}, \eqref{lji}, we get 
\begin{equation*}
    \begin{split}
        \|B_{l_{j_{m-1}}}\|_0 \, \|B_{l_{j_{m-1}}}\|_{k} 
        & \lesssim l_{j_{m-1}}^{(1+s)k+\sigma D\tau}  
        \lesssim l_{j_{m}}^{\frac{(1+s)k+\sigma D\tau}{\sigma D}(1+\lambda D)} \\ 
        & \lesssim \gamma^{-\frac{(1+s)k+\sigma D\tau}{\sigma D(\sigma D-2-2\delta)\tau}(1+\lambda D)} \leqslant \gamma^{-\frac{k'-4\sigma D\tau-4}{2\sigma D(\sigma D-2-2\delta)\tau}}, 
    \end{split}
\end{equation*}
and 
\begin{equation*}
    \|Y_{l_{j_m}}\|_{k} \lesssim l_{j_{m}+1}^{k}|n^*|^{\tau/2}\epsilon_{l_{j_m}}^{1/2} \lesssim l_{j_{m}}^{(1+s)k+\frac{(1+\delta)\tau}{2}-\frac{D\tau}{2}} .
\end{equation*}
Here, we want $\|e^{Y_{l_{j_m}}}\|_{k}$ not to have any impact on the order of estimate of $\|\widetilde{B}\|_0 \, \|\widetilde{B}\|_{k}$, thus we need to choose $k \leqslant \frac{k'-8\sigma D\tau-4}{4(1+s)^2}$. 

Therefore, by Theorem~\ref{Ckalmost}\ref{item:nre}, \eqref{P}, and $D>\frac{2(k'-1)}{(4-5\sigma)\tau}$, we have 
\begin{equation*}
    \begin{split}
        \|\widetilde{B}\|_0 \, \|\widetilde{B}\|_{k} 
        & \leqslant 2\|B_{l_{j_{m-1}}}\|_0 \, \|B_{l_{j_{m-1}}}\|_{k} \|P_{l_{j_{m}}}\|^2 \|e^{Y_{l_{j_m}}}\|_{k} \\ 
        & \lesssim \gamma^{-\frac{k'-4\sigma D\tau-4}{2\sigma D(\sigma D-2-2\delta)\tau}-\frac{1+\delta}{\sigma D-2-2\delta}} 
        \leqslant \gamma^{-\frac{2\tau k'}{\sigma^2(k'-1)^2}}. 
    \end{split}
\end{equation*}

Moreover, by Theorem~\ref{Ckalmost}\ref{item:nre}\ref{item:re} together with \eqref{deg1}, \eqref{estb1}, \eqref{Nlj0}, \eqref{lji}, and $D>\frac{2(k'-1)}{(4-5\sigma)\tau}$, we obtain 
\begin{equation*}
    \begin{split}
        |\deg\widetilde{B}|=|\deg B_{l_{j_{m-1}}}| & \leqslant 8\,l_{j_{m-1}} |\ln \epsilon_{l_{j_{m-1}}}| 
        \leqslant l_{j_{m-1}}^{1+\delta} 
        \leqslant l_{j_{m}}^{\frac{(1+\delta)(1+\lambda D)}{\sigma D}} \\ 
        & \leqslant \gamma^{-\frac{(1+\delta)(1+\lambda D)}{\sigma D(\sigma D-2-2\delta)\tau}} \leqslant \gamma^{-\frac{10\tau}{\sigma^2(k'-1)^2}}.
    \end{split}
\end{equation*}

Since $\rho(\alpha, Ae^f) \in \Theta_\gamma^{2\tau}$, Theorem~\ref{Ckalmost}\ref{item:nre}, Theorem~\ref{Ckalmost}\ref{item:re} and \eqref{lji} give 
\begin{equation*}
    \begin{split}
        \|2\rho(\bar{A})\|_{\mathbb{R}/\mathbb{Z}} & =\|2\rho(\alpha,Ae^f)-\langle \deg B, \alpha \rangle\|_{\mathbb{R}/\mathbb{Z}} \\
        & \geqslant \frac{\gamma}{(|\deg B|+1)^{2\tau}} \geqslant \frac{\gamma}{\big(\sum_{i=1}^{m} |n_{l_{j_{i}}}^{*}|+1\big)^{2\tau}} \\
        & \geqslant \frac{\gamma}{(2|n^{*}|+1)^{2\tau}} \geqslant \frac{\gamma}{4^\tau(|n^{*}|+1)^{2\tau}}.
    \end{split}
\end{equation*}

Finally, from Theorem~\ref{Ckalmost}\ref{item:nre}\ref{item:re} we get  
\begin{equation*}
    \begin{split}
        |\nu| & \leqslant |\nu_{j_{m}}|+2\sum_{j=j_{m}}^{\infty} \|A_{l_{j+1}}-A_{l_j}\| \\ 
        & \leqslant 64\|A_{l_{j_{m}-1}}\|\kappa^{-1} |n^*|^\tau \epsilon_{l_{j_{m}}} e^{-2\pi|n^*| r_{j_{m}}}+4\sum_{j=j_{m}}^{\infty} \|A_{l_j}\| \epsilon_{l_{j+1}} \\
        & \lesssim (1+|n^*|)^\tau \epsilon_{l_{j_{m}}}.
    \end{split}
\end{equation*}
By \eqref{Nlj0}, and $D>\frac{2(k'-1)}{(4-5\sigma)\tau}$, 
\begin{equation*}
    |\nu| \lesssim (1+|n^*|)^\tau l_{j_{m}}^{-D\tau} \lesssim (1+|n^{*}|)^{-\frac{D\tau}{1+\delta}+\tau} 
    \leqslant (1+|n^{*}|)^{-(\frac{k'}{2}+2\tau)}.
\end{equation*}

Finally, if $m=1$, the conclusion follows from the same argument with $j_{m}$ replaced by $j_{1}$, where, when $j_{1}=1$, we use directly the resonant estimates for the first step in the proof of Theorem~\ref{Ckalmost} and apply \eqref{con} directly to $(\alpha,Ae^{f})$ and $(\alpha,A)$. 
\end{proof}

One can get the following corollary for the $C^k$-reducibility of the quasi-periodic Schr\"odinger cocycle as well as arithmetic PSL and PDL of dual long-range operator by Theorem \ref{Ckreducibility}, Theorem \ref{APSL} and Theorem \ref{pdl}. 

\begin{Corollary}\label{cor3.1}
	Let $\alpha \in \mathrm{DC}_d(\kappa,\tau)$, $0<\sigma<1/9$, $s \in (0,1-\frac{7}{2}\sigma)$ and $\mathfrak{r} \geqslant \tau$. Suppose that $V \in C^{k'}(\mathbb{T}^d, \mathbb{R})$, where $k'>\frac{18}{\sigma^2}(1+s+\sigma)\mathfrak{r}+1$. Let $D>0$ be a constant satisfying 
    \begin{equation*}
        \max\bigg\{\frac{2(k'-1)}{(4-5\sigma)\mathfrak{r}},\frac{18}{\sigma^2}\bigg\}<D<\min\bigg\{\frac{k'-1}{(1+\sigma+s)\mathfrak{r}},\frac{k'-16d-20}{8\sigma\mathfrak{r}}\bigg\}.
    \end{equation*}  
    Then for any 
    \begin{equation*}
        \max\left\{d,\frac{8\mathfrak{r}^2 k'}{\sigma^2(k'-1)^2-8\sigma\mathfrak{r}(k'-1)-40\mathfrak{r}^2}\right\}<k \leqslant \frac{k'-8\sigma D\mathfrak{r}-4}{4(1+s)^2}, \ k \in \mathbb{N}, 
    \end{equation*}
    there exists $\epsilon^*_3=\epsilon^*_3(\kappa,\mathfrak{r},d,k',\sigma,s,D)$ such that if $\|V\|_{k'} \leqslant \epsilon^*_3$, then for all $E \in \Sigma_{V,\alpha}$ satisfying $\rho(E) \in \Theta^{2\mathfrak{r}}$, the Schr\"odinger cocycle $(\alpha,S_E^V)$ is $C^{k}$-reducible, i.e. 
	\begin{equation}\label{conj2}
		B^{-1}(x+\alpha)S_E^V(x)B(x)=\bar{A}, 
		\end{equation}
	with the estimates of $B(x)$ and $\bar{A}$ hold as in Theorem \ref{Ckreducibility}, with $\tau$ replaced by $\mathfrak{r}$. Moreover, the long-range operator $L_{V,\alpha,\theta}$ has arithmetic $k$-PSL for every $\theta \in \Theta^\mathfrak{r}$ and the family $\{L_{V, \alpha, \theta}\}_{\theta \in \mathbb{T}}$ has arithmetic $k$-PDL. 
\end{Corollary}

\begin{proof}
	Let $A_E=\begin{pmatrix} 
		    E & -1 \\ 
		    1 & 0 
	    \end{pmatrix}$
	and $f(x)=\begin{pmatrix} 
			0 & 0 \\ 
			V(x) & 0 
		\end{pmatrix}$, 
	thus
	\begin{equation*}
		A_Ee^{f(x)}=S_E^V(x) \in C^{k'}(\mathbb{T}^d,\mathrm{SL}(2,\mathbb{R})).
	\end{equation*}
	Since $\mathfrak{r} \geqslant \tau$, one has $\alpha \in {\rm DC}_d(\kappa,\mathfrak{r})$. Moreover, after possibly decreasing $\epsilon_3^*$, we may assume $\|V\|_0 \leqslant 1$. Then, for every $E \in \Sigma_{V,\alpha}$,
    \begin{equation*}
        |E| \leqslant 2+\|V\|_0 \leqslant 3.
    \end{equation*}
    Thus $\|A_E\|$ is bounded uniformly for all $E \in \Sigma_{V,\alpha}$. Choose $\epsilon_3^*=\epsilon_{3}^{*}(\kappa,\mathfrak{r},d,k',\sigma,s,D)$ sufficiently small, so that the smallness condition in Theorem~\ref{Ckreducibility}, with $\tau$ replaced by $\mathfrak{r}$, holds uniformly for all such matrices $A_E$.
        
    Fix $\gamma \in (0,1)$ and $E \in \Sigma_{V,\alpha}$ with $\rho(E) \in \Theta_{\gamma}^{2\mathfrak{r}}$. Applying Theorem~\ref{Ckreducibility} with $\tau$ replaced by $\mathfrak{r}$, $A=A_E$, and this $\gamma$, we obtain \eqref{conj2} together with the corresponding estimates, uniformly for all such $E$. Since $\mathcal{E}^{2\mathfrak{r}}=\bigcup_{\gamma \in (0,1)}\mathcal{E}_{\gamma}^{2\mathfrak{r}}$, the reducibility statement follows.
        
    It remains to verify that the hypotheses of Theorem~\ref{APSL} and Theorem~\ref{pdl} hold with $\tau$ replaced by $\mathfrak{r}$. By direct calculation, if $k'>\frac{18}{\sigma^2}(1+s+\sigma)\mathfrak{r}+1$ and $D<\frac{k'-16d-20}{8\sigma\mathfrak{r}}$, it obviously that  
    \begin{equation*}
        \frac{k'-8\sigma D\mathfrak{r}-4}{4(1+s)^2}>d+1.  
    \end{equation*}
    By Theorem~\ref{Ckreducibility}, the constants in \eqref{est} can be chosen as  
    \begin{equation*}
        c_1=\frac{2}{\sigma(k'-1)}, \ c_2=\frac{k'}{2}, \ c_3=\frac{2\mathfrak{r} k'}{\sigma^2(k'-1)^2}, \ c_4=\frac{10\mathfrak{r}}{\sigma^2(k'-1)^2}, \ c_5=\frac{k'}{2}+2\mathfrak{r}.  
        \end{equation*} 
    By the choice of $k$, 
    \begin{equation*}
        2\mathfrak{r}(c_3+c_4k+c_1k)<k/2, 
    \end{equation*}
    moreover,  
    \begin{equation*}
        c_5>2\mathfrak{r}, \quad \frac{k}{\min\{c_2,c_5-2\mathfrak{r}\}}+2(c_3+c_4k)<1. 
    \end{equation*}
    Then, Theorem~\ref{APSL} and Theorem~\ref{pdl}, with $\tau$ replaced by $\mathfrak r$, imply that the long-range operator $L_{V,\alpha,\theta}$ has arithmetic $k$-PSL for every $\theta \in \Theta^\mathfrak{r}$, and the family $\{L_{V,\alpha,\theta}\}_{\theta \in \mathbb{T}}$ has arithmetic $k$-PDL. 
        
	This finishes the proof of Corollary \ref{cor3.1}. 
\end{proof}

\section{Acknowledgements}
Ao Cai was supported by NSFC grant 12671227, 12271091 and JSTJ-2025-674. 
Yuan Shan was supported by the National Natural Science Foundation of China grant No. 12671214. 
Huihui Lv was supported by Postgraduate Research \& Practice Innovation Program of Jiangsu Province grant 26CXJH5827.

\appendix	
\section{Proof of Lemma \ref{continuous}}
\label{appendixA}
\begin{proof}
Fix $M>0$. Since $\theta,\theta' \in \Theta_{\bar{\gamma}}^{2\tau}$, for any
$|m| \leqslant M$ and any $l \in \mathbb{Z}^d$, one has
\begin{equation*}
    \begin{split}
        \|2T^m\theta-\langle l,\alpha\rangle\|_{\mathbb{R}/\mathbb{Z}}
        &=\|2\theta-\langle 2m+l,\alpha\rangle\|_{\mathbb{R}/\mathbb{Z}} \\
        & \geqslant \frac{\bar{\gamma}}{(|2m+l|+1)^{2\tau}} \geqslant \frac{(2M+1)^{-2\tau}\bar{\gamma}}{(|l|+1)^{2\tau}}.
    \end{split}
\end{equation*}
Thus
\begin{equation}\label{TmDio}
    T^m\theta,\ T^m\theta' \in \Theta_{\gamma_1}^{2\tau}, \quad \gamma_1=(2M+1)^{-2\tau}\bar{\gamma}, \quad |m| \leqslant M.
\end{equation}

We first prove the following claim. If $\theta_j,\theta'_j \in \Theta_{\bar{\gamma}}^{2\tau}$ and $\|\theta_j-\theta'_j\|_{\mathbb{R}/\mathbb{Z}} \to 0$, then for every fixed $|m| \leqslant M$ and every fixed $l \in \mathbb{Z}^d$,
\begin{equation}\label{pointwisecontinuity}
    |u_{T^m\theta_j}(l)|^2-|u_{T^m\theta'_j}(l)|^2 \to 0.
\end{equation}

It suffices to show that every subsequence admits a further subsequence along which \eqref{pointwisecontinuity} holds. Let $\{j_i\}_{i \geqslant 1}$ be an arbitrary subsequence. Since $\mathbb{T}$ is compact, after passing to a further subsequence, still denoted by $\{j_i\}$, we may assume that $T^m\theta_{j_i} \to \phi$ for some $\phi \in \mathbb{T}$. Since $\|T^m\theta_{j_i}-T^m\theta'_{j_i}\|_{\mathbb{R}/\mathbb{Z}} \to 0$, one also has $T^m\theta'_{j_i} \to \phi$. By \eqref{TmDio}, $\phi \in \Theta_{\gamma_1}^{2\tau}$. 

We first consider $\phi \in (0,1/2)$. After passing to a further subsequence, $T^m\theta_{j_i}$ and $T^m\theta'_{j_i}$ lie in the same branch of \eqref{utheta}. By the continuity and monotonicity of the rotation number on the spectrum and by the definition of $E(\cdot)$,
\begin{equation}\label{energyconv}
    E(T^m\theta_{j_i}) \to E(\phi), \quad
    E(T^m\theta'_{j_i}) \to E(\phi).
\end{equation}

By the assumptions \eqref{red}--\eqref{est}, applied with the arithmetic constant $\gamma_1$, the conjugacies used to construct $u_E$ for all $E$ satisfying $\rho(E) \in \Theta_{\gamma_1}^{2\tau}$ can be chosen with uniform bounds depending only on $\alpha,V,d,k,\gamma_1$. More precisely, after the constant diagonalization step in the proof of Theorem~\ref{goodeig}, there exist $B_{E} \in C^k(2\mathbb{T}^d,\mathrm{SL}(2,\mathbb{C}))$, $A_E \in \mathrm{SL}(2,\mathbb{R})$ such that
\begin{equation}\label{uniformBE}
    B_E^{-1}(x+\alpha)S_E^V(x)B_E(x)=
    \begin{pmatrix}
        e^{2\pi i\rho(A_E)} & 0 \\
        0 & e^{-2\pi i\rho(A_E)}
    \end{pmatrix}
\end{equation}
and
\begin{equation}\label{uniformboundBE}
    \|B_E\|_k, |\deg B_E| \leqslant C_0(\alpha,V,d,k,\gamma_1).
\end{equation}

For simplicity, denote $E_{j_i}=E(T^m\theta_{j_i}), \ E'_{j_i}=E(T^m\theta'_{j_i})$. Choose $B_{E_{j_i}}$ and $B_{E'_{j_i}}$ satisfying
\eqref{uniformBE}\eqref{uniformboundBE}. By \eqref{uniformboundBE} and the compact embedding $C^k(2\mathbb{T}^d) \hookrightarrow C^{k-1}(2\mathbb{T}^d)$, after passing to a further subsequence, still denoted by $\{j_i\}$, we have
\begin{equation}\label{Bcompact}
    B_{E_{j_i}} \to B_*, \quad B_{E'_{j_i}} \to B'_*
\end{equation}
in $C^{k-1}$, hence in $C^0$. Moreover, by \eqref{uniformboundBE}, $\deg B_{E_{j_i}}$, $\deg B_{E'_{j_i}}$ take values in the finite subset of $\mathbb{Z}^d$. Thus, by the pigeonhole principle, after passing to a further subsequence, still denoted by $\{j_i\}$, for all $i$, $\deg B_{E_{j_i}}$, $\deg B_{E'_{j_i}}$ are constant respectively.

After passing to a further subsequence if necessary, we may also assume that $\rho(A_{E_{j_i}}) \to \rho_*$ and $\rho(A_{E'_{j_i}}) \to \rho'_*$. Passing to the limit in \eqref{uniformBE}, using \eqref{energyconv}, we obtain that $B_*$ and $B'_*$ are two reducibility conjugacies for the same cocycle $(\alpha,S_{E(\phi)}^V)$. Hence the normalized eigenfunctions constructed
from $B_*$ and $B'_*$ are eigenfunctions of
$L_{V,\alpha,\rho(E(\phi))}$ associated with the same energy $E(\phi)$.

Write
\begin{equation*}
    B_{E_{j_i}}(x)=
    \begin{pmatrix}
        b^{11}_{E_{j_i}}(x) & b^{12}_{E_{j_i}}(x) \\
        b^{21}_{E_{j_i}}(x) & b^{22}_{E_{j_i}}(x)
    \end{pmatrix},
\end{equation*}
and define $B_{E'_{j_i}}$ similarly. Let $z_{E_{j_i}}^{11}(x)=e^{-2\pi i\langle \deg B_{E_{j_i}},x\rangle/2} b_{E_{j_i}}^{11}(x)$ and define $z_{E'_{j_i}}^{11}$ similarly. By \eqref{Bcompact} and $\deg B_{E_{j_i}}$, $\deg B_{E'_{j_i}}$ are constants, for every fixed $l \in \mathbb{Z}^d$, we have 
\begin{equation}\label{coefconv}
    \hat{z}_{E_{j_i}}^{11}(l) \to \hat{z}_*^{11}(l), \quad \hat{z}_{E'_{j_i}}^{11}(l) \to \hat{z}_{*}^{\prime 11}(l),
\end{equation}
and
\begin{equation}\label{normconv}
    \|z_{E_{j_i}}^{11}\|_{L^2} \to \|z_*^{11}\|_{L^2}, \quad \|z_{E'_{j_i}}^{11}\|_{L^2}\to \|z_*^{\prime 11}\|_{L^2}.
\end{equation}
By Lemma~\ref{lem3.1} and \eqref{uniformboundBE}, the $L^2$ norms in \eqref{normconv} are uniformly bounded from below. Thus
\begin{equation*}
    u_{E_{j_i}}(l) \to u_*(l), \quad u_{E'_{j_i}}(l) \to u'_*(l),
\end{equation*}
where $u_*$ and $u'_*$ are the normalized eigenfunctions obtained from $B_*$ and $B'_*$ respectively. By Lemma~\ref{welldefine}, $u_*$ and $u'_*$ are equal up to a unimodular constant. Thus $|u_*(l)|^2=|u'_*(l)|^2$. 

If $\phi \in (1/2,1)$, by $u_{T^m\theta}(l)=\overline{u_{E(T^m\theta)}^{+}(-l)}$ for $T^m\theta$ in this branch, so the same conclusion follows from \eqref{uminus} and \eqref{utheta} by applying the preceding argument to the positive-branch eigenfunctions at the index $-l$.

Consequently, along this further subsequence, we obtain 
\begin{equation*}
    |u_{T^m\theta_{j_i}}(l)|^2-|u_{T^m\theta'_{j_i}}(l)|^2 \to 0.
\end{equation*}
Since the original subsequence $\{j_i\}$ was arbitrary, every subsequence admits a further subsequence along which the convergence in \eqref{pointwisecontinuity} holds. Therefore, \eqref{pointwisecontinuity} holds for the original sequence.

Now we prove the lemma. Suppose, for contradiction, that the conclusion is false. Then there exist $\varepsilon_0>0$ and sequences $\theta_j,\theta'_j \in \Theta_{\bar{\gamma}}^{2\tau}$ and $\|\theta_j-\theta'_j\|_{\mathbb{R}/\mathbb{Z}} \to 0$ such that
\begin{equation}\label{contcontra}
    \left|\mathcal{T}_M\nu_{\theta_j,\delta_n}(\mathcal{E}^{2\tau})-\mathcal{T}_M\nu_{\theta'_j,\delta_n}(\mathcal{E}^{2\tau})\right| \geqslant \varepsilon_0.
\end{equation}
By \eqref{TmDio}, for all $|m| \leqslant M$, one has $T^m\theta_j,\ T^m\theta'_j \in \Theta_{\gamma_1}^{2\tau} \subseteq \Theta^{2\tau}$. Thus, by Definition~\ref{RM} and \eqref{22},
\begin{equation*}
    \begin{split}
        \left|\mathcal{T}_M\nu_{\theta_j,\delta_n}(\mathcal{E}^{2\tau})-\mathcal{T}_M\nu_{\theta'_j,\delta_n}(\mathcal{E}^{2\tau})
        \right| & =\left|\sum_{|m| \leqslant M}|u_{m,\theta_j}(n)|^2-\sum_{|m| \leqslant M}|u_{m,\theta'_j}(n)|^2\right| \\ 
        & \leqslant \sum_{|m| \leqslant M}\left||u_{T^m\theta_j}(n+m)|^2-|u_{T^m\theta'_j}(n+m)|^2\right|
    \end{split} 
\end{equation*}
For each fixed $|m| \leqslant M$, applying \eqref{pointwisecontinuity} with $l=n+m$ gives
\begin{equation*}
    |u_{T^m\theta_j}(n+m)|^2-|u_{T^m\theta'_j}(n+m)|^2 \to 0.
\end{equation*}
Since the set $\{m \in \mathbb{Z}^d \colon |m| \leqslant  M\}$ is finite and all eigenfunctions are normalized, it follows that
\begin{equation*}
    \left|\mathcal{T}_M\nu_{\theta_j,\delta_n}(\mathcal{E}^{2\tau})-\mathcal{T}_M\nu_{\theta'_j,\delta_n}(\mathcal{E}^{2\tau})
    \right| \to 0,
\end{equation*}
which contradicts \eqref{contcontra}. This proves Lemma~\ref{continuous}.
\end{proof}

\section{Analytic almost reducibility}
\begin{theorem}[\cite{MR4512234}]\label{KAM}
     Let $\alpha \in \mathrm{DC}_{d}(\kappa,\tau)$, $\kappa, r \in (0,1)$, $\tau>d$, $\sigma \in (0,1/6)$, and $D>2/\sigma$. Suppose that $A \in \mathrm{SL}(2,\mathbb{R})$ and $f \in C^\omega_r(\mathbb{T}^d,\mathrm{sl}(2,\mathbb{R}))$. Then for any $r' \in (0,r)$, there exist constants $c=c(\kappa,\tau,d,\sigma,D)$ and $\tilde{D}=\tilde{D}(D)>0$ such that if 
     \begin{equation}\label{pertur}
         |f|_r \leqslant \epsilon \leqslant \frac{c}{\|A\|^{\tilde{D}}}(r-r')^{D\tau}, 
     \end{equation}
     then there exist $B \in C^{\omega}_{r'}(2\mathbb{T}^d,\mathrm{SL}(2,\mathbb{R}))$, $A_+ \in \mathrm{SL}(2,\mathbb{R})$ and $f_+ \in C^\omega_{r'}(\mathbb{T}^d,\mathrm{sl}(2,\mathbb{R}))$ such that 
     \begin{equation*}
         B^{-1}(x+\alpha)Ae^{f(x)}B(x)=A_+e^{f_+(x)}. 
     \end{equation*}
     More precisely, let $N=\frac{2}{r-r'}|\ln{\epsilon}|$,  we distinguish two case:
     \begin{enumerate}[label=(\Alph*)]
		\item (Non-resonant case) if for any $n \in \mathbb{Z}^d$ with $0<|n| \leqslant N$, we have 
        \begin{equation}\label{Dio1}
            \|2\rho(A)-\langle n,\alpha \rangle\|_{\mathbb{R}/\mathbb{Z}} \geqslant \epsilon^\sigma, 
        \end{equation}
        then
        \begin{equation*}
            \begin{split}
                & |B-\mathrm{Id}|_{r'} \leqslant \epsilon^{-(2\sigma+3/D)}|f|_r \leqslant \epsilon^{1-2\sigma-3/D}, \\ 
                & |f_+|_{r'} \leqslant \epsilon^{2-2\sigma-3/D}, \quad \|A_{+}- A\| \leqslant 2\|A\|\epsilon.
            \end{split} 
        \end{equation*}

        \item (Resonant case) if there exists $n^*$ with $0<|n^*| \leqslant N$ such that 
        \begin{equation*}
            \|2\rho(A)-\langle n^*,\alpha \rangle\|_{\mathbb{R}/\mathbb{Z}}<\epsilon^\sigma, 
        \end{equation*}
        then
        \begin{equation*}
            B(x)=P^{-1} \cdot e^{Y(x)} \cdot R_{\langle n^*,x \rangle/2}
        \end{equation*}
        where $P \in \mathrm{SL}(2,\mathbb{R}), Y \in C_{r'}^{\omega}(\mathbb{T}^d,\mathrm{sl}(2,\mathbb{R}))$ and $R_{\langle n^*,x \rangle/2}$ is the rotation matrix satisfy 
        \begin{equation*}
            \begin{split}
                & \qquad \|P\| \leqslant 4\|A\|^{1/2}\kappa^{-1/2}|n^*|^{\tau/2}, \quad |Y|_{r'} \leqslant 4\|A\|^{1/2}\kappa^{-1/2}|n^*|^{\tau/2} \epsilon^{1/2}, 
            \end{split}
        \end{equation*}
        and
        \begin{equation*}
            \begin{split}
                & \qquad |B|_{r'} \leqslant 8\|A\|^{1/2}\kappa^{-1/2}|n^{*}|^{\tau/2} e^{\pi|n^{*}|r'}, \\ 
                & \qquad \|B\|_0 \leqslant 8\|A\|^{1/2}\kappa^{-1/2}|n^{*}|^{\tau/2}, \quad |f_+|_{r'} \leqslant \epsilon^{100}.
            \end{split}
        \end{equation*}
    Moreover, there exist $t \in \mathbb{R}$ and $\nu \in \mathbb{C}$ such that
     \begin{equation}\label{KAMnormalform}
        A_{+}=M^{-1}\exp
            \begin{pmatrix}
                it & \nu \\
                \bar{\nu} & -it
            \end{pmatrix}M,
    \end{equation}
    with  
    \begin{equation*}
        \qquad \|2\rho(A_{+})\|_{\mathbb{R}/\mathbb{Z}} \leqslant 4\epsilon^{\sigma}, \quad |t| \leqslant 4\epsilon^\sigma, \quad |\nu| \leqslant 64\|A\|\kappa^{-1}|n^*|^\tau\epsilon e^{-2\pi|n^*|r}.
    \end{equation*} 
    \end{enumerate}
\end{theorem}

\begin{proof}
The resonant case follows from the argument in \cite[Appendix]{MR4512234}, so we only prove the non-resonant case. 

The hyperbolic case is standard, so we give the details for the non-hyperbolic case. Then we can choose $\rho \in \mathbb{R}$ such that
\begin{equation*}
    \rho=\rho(A) \bmod \mathbb{Z}, \quad \mathrm{spec}(A)=\{e^{2\pi i\rho},e^{-2\pi i\rho}\}.
\end{equation*}
Define the truncated operator $\mathcal{T}_{N}$ and residual operator $\mathcal{R}_{N}$ by 
\begin{equation*}
    (\mathcal{T}_{N}f)(x)=\sum_{|k| \leqslant N} \hat{f}(k)e^{2\pi i\langle k, x\rangle}, \quad (\mathcal{R}_{N}f)(x)=\sum_{|k|>N} \hat{f}(k)e^{2\pi i\langle k, x\rangle}.
\end{equation*}

{\bf Step 1: Upper triangularization.}

By Schur's Theorem, there exists $U \in \mathrm{SU}(2)$ such that 
\begin{equation*}
    UAU^{-1}=\begin{pmatrix}
    e^{2\pi i\rho} & p \\ 
    0 & e^{-2\pi i\rho}
    \end{pmatrix} \eqcolon \Lambda. 
\end{equation*}

{\bf Step 2: Solve truncated equation.}

Consider the linearized cohomological equation
    \begin{equation}\label{cohomo}
        A^{-1}Y(x+\alpha)A-Y(x)=(\mathcal{T}_{N}f)(x)-\hat{f}(0).  
    \end{equation}
To solve $Y$, we conjugate the equation to the form
\begin{equation}\label{complex}
    \Lambda^{-1} UY(x+\alpha) U^{-1}\Lambda- UY(x)U^{-1}= U((\mathcal{T}_{N}f)(x)-\hat{f}(0))U^{-1} \eqcolon g(x). 
\end{equation}
We let $UYU^{-1}=\begin{pmatrix}
	    	y_1 & y_2 \\
	    	y_3 & y_4 \\
	    \end{pmatrix}$ and $g=\begin{pmatrix}
	    	g_1 & g_2 \\
	    	g_3 & g_4 \\
	    \end{pmatrix}$,
comparing the Fourier coefficients, one obtains that
\begin{equation}\label{hatyi}
    \begin{split}
        &\hat{y}_{3}(k)=\frac{\hat{g}_3(k)}{e^{2\pi i(\langle k,\alpha\rangle+2\rho)}-1}, \\
        &\hat{y}_{1}(k)=-\hat{y}_{4}(k)=\frac{\hat{g}_1(k)+pe^{2\pi i(\langle k,\alpha \rangle+\rho)}\hat{y}_3(k)}{e^{2\pi i\langle k,\alpha\rangle}-1}, \\ 
        & \hat{y}_{2}(k)=\frac{\hat{g}_2(k)+p^2e^{2\pi i\langle k,\alpha \rangle}\hat{y}_3(k)-2pe^{2\pi i(\langle k,\alpha \rangle-\rho)}\hat{y}_1(k)}{e^{2\pi i(\langle k,\alpha \rangle-2\rho)}-1}. 
    \end{split}
\end{equation}

By \eqref{pertur} and $D>2/\sigma$, for any $0<|n|\leqslant N$ we have 
    \begin{equation}\label{Dio2}
        \|\langle n,\alpha \rangle\|_{\mathbb{R}/\mathbb{Z}} \geqslant \frac{\kappa}{|n|^\tau} \geqslant \frac{\kappa}{N^\tau}.  
    \end{equation}
Use \eqref{Dio1} and \eqref{Dio2}, we can get 
\begin{equation*}
    \begin{split}
        & \frac{1}{|e^{2\pi i(\langle k,\alpha \rangle \pm 2\rho)}-1|}=\frac{1}{2|\sin(\pi (\langle k,\alpha \rangle \pm 2\rho))|} \leqslant \frac{1}{4\|2\rho \pm \langle k,\alpha \rangle\|_{\mathbb{R}/\mathbb{Z}}} \leqslant \frac{1}{4}\epsilon^{-\sigma}, \\ 
        & \frac{1}{|e^{2\pi i\langle k,\alpha \rangle}-1|}=\frac{1}{2|\sin(\pi \langle k,\alpha \rangle)|} \leqslant \frac{1}{4\|\langle k,\alpha \rangle\|_{\mathbb{R}/\mathbb{Z}}} \leqslant \frac{1}{4\kappa}N^\tau. 
    \end{split}
\end{equation*}
This yields, 
\begin{equation*}
    \begin{split}
        |y_3|_{r'} & =|\sum_{k\in \mathbb{Z}^d,0<|k|\leqslant N} \hat{y}_3(k)e^{2\pi i\langle k,x \rangle}|_{r'}=|\sum_{k\in \mathbb{Z}^d,0<|k|\leqslant N} \frac{\hat{g}_3(k)}{e^{2\pi i(\langle k,\alpha\rangle+2\rho)}-1}e^{2\pi i\langle k,x \rangle}|_{r'} \\ 
        & \leqslant \frac{1}{4}\epsilon^{-\sigma}\sum_{k\in \mathbb{Z}^d,0<|k|\leqslant N} |\hat{g}_3(k)e^{2\pi i\langle k,x \rangle}|_{r'} \leqslant \frac{1}{4}\epsilon^{-\sigma}|g|_r \sum_{k\in \mathbb{Z}^d,0<|k|\leqslant N} e^{-2\pi|k|(r-r')}, \\ 
        |y_4|_{r'} & =|y_1|_{r'}=|\sum_{k\in \mathbb{Z}^d,0<|k|\leqslant N} \frac{\hat{g}_1(k)+pe^{2\pi i(\langle k,\alpha \rangle+\rho)}\hat{y}_3(k)}{e^{2\pi i\langle k,\alpha\rangle}-1}e^{2\pi i\langle k,x \rangle}|_{r'} \\ 
        & \leqslant \frac{1}{4\kappa}N^\tau \sum_{k\in \mathbb{Z}^d,0<|k|\leqslant N} (|\hat{g}_1(k)|+|pe^{2\pi i(\langle k,\alpha \rangle+\rho)}\hat{y}_3(k)|)|e^{2\pi i\langle k,x \rangle}|_{r'} \\ 
        & \leqslant \frac{1}{4\kappa}N^{\tau}\epsilon^{-\sigma}|g|_r(1+|p||e^{2\pi i\rho}|) \sum_{k\in \mathbb{Z}^d,0<|k|\leqslant N} e^{-2\pi|k|(r-r')}, \\ 
    \end{split}
\end{equation*}
\begin{equation*}
    \begin{split}        
        |y_2|_{r'} & \leqslant |\sum_{k\in \mathbb{Z}^d,0<|k|\leqslant N}  \frac{\hat{g}_2(k)+p^2e^{2\pi i\langle k,\alpha \rangle}\hat{y}_3(k)-2pe^{2\pi i(\langle k,\alpha \rangle-\rho)}\hat{y}_1(k)}{e^{2\pi i(\langle k,\alpha \rangle-2\rho)}-1}e^{2\pi i\langle k,x \rangle}|_{r'} \\ 
        & \leqslant \frac{1}{16\kappa}N^\tau\epsilon^{-2\sigma}|g|_r(1+|p|^2+2|p||e^{-2\pi i\rho}|+2|p|^2) \sum_{k\in \mathbb{Z}^d,0<|k|\leqslant N} e^{-2\pi|k|(r-r')}. \\ 
    \end{split}
\end{equation*}
Then for $y_j, j=1,2,3,4$, by integrating by parts, we have 
\begin{equation*}
    \begin{split}
        |y_j|_{r'} & \leqslant \frac{1}{4\kappa}N^\tau\epsilon^{-2\sigma}|g|_r(1+|p|^2+2|p||e^{-2\pi i\rho}|+2|p|^2) \sum_{k\in \mathbb{Z}^d,0<|k|\leqslant N} e^{-2\pi|k|(r-r')} \\ 
        & \leqslant \frac{1}{4\kappa}N^\tau\epsilon^{-2\sigma}|g|_r6\|A\|^2 \sum_{k\in \mathbb{Z}^d,0<|k|\leqslant N} e^{-2\pi|k|(r-r')} \\ 
        & \leqslant \frac{1}{4}\epsilon^{-2(\sigma+1/D)}|g|_r, 
    \end{split}
\end{equation*}
Therefore, 
\begin{equation}\label{anaY}
    |Y|_{r'}=|UYU^{-1}|_{r'} \leqslant \epsilon^{-2(\sigma+1/D)}|g|_r \leqslant \epsilon^{-2(\sigma+1/D)}|\mathcal{T}_{N}f-\hat{f}(0)|_r.  
\end{equation} 
And we have $Y \in C^{\omega}_{r'}(\mathbb{T}^d,\mathrm{sl}(2,\mathbb{R}))$. Firstly, \eqref{hatyi} indicates that $\mathrm{tr}Y=\mathrm{tr}(UYU^{-1})=0$. Then, for $k, 0<|k| \leqslant N$, equation \eqref{cohomo} induces a linear equation 
\begin{equation*}
    M_k(\hat{Y}(k)) \coloneq A^{-1}\hat{Y}(k)e^{2\pi i\langle k,\alpha \rangle}A-\hat{Y}(k)=\widehat{\mathcal{T}_Nf}(k). 
\end{equation*}
By \eqref{Dio1} and \eqref{Dio2}, the operator $M_k \colon \mathrm{sl}(2,\mathbb{C}) \rightarrow \mathrm{sl}(2,\mathbb{C})$ is invertible, then there exists a unique solution $Y(x)=\sum_{k \in \mathbb{Z}^d, 0<|k| \leqslant N}\hat{Y}(k)e^{2\pi i\langle k,x \rangle} \in C^{\omega}_{r'}(\mathbb{T}^d,\mathrm{sl}(2,\mathbb{C}))$. On the other hand, the original operator
\begin{equation*}
    \mathcal{L}Y(x)=A^{-1}Y(x+\alpha)A-Y(x)=\mathcal{T}_{N}f-\hat{f}(0)
\end{equation*}
preserves real-valued trigonometric polynomials. Since $\mathcal{T}_{N}f-\hat{f}(0) \in C^{\omega}_{r}(\mathbb{T}^d,\mathrm{sl}(2,\mathbb{R}))$, the solution $Y$ must satisfy $\hat{Y}(-k)=\overline{\hat{Y}(k)}$, then $Y \in C^{\omega}_{r'}(\mathbb{T}^d,\mathrm{sl}(2,\mathbb{R}))$.  

By cohomological equation \eqref{cohomo}, we obtain 
\begin{equation*}
    -Y(x+\alpha)=-AY(x)A^{-1}-A((\mathcal{T}_{N}f)(x)-\hat{f}(0))A^{-1}.
\end{equation*}
Then this yields, 
\begin{equation*}
    \begin{split}
        e^{-Y(x+\alpha)}(Ae^{f(x)})e^{Y(x)} 
	    = & \ e^{-AY(x)A^{-1}-A((\mathcal{T}_{N}f)(x)-\hat{f}(0))A^{-1}}(Ae^{f(x)})e^{Y(x)} \\
	    = & \ Ae^{\hat{f}(0)-(\mathcal{T}_Nf)(x)-Y(x)}e^{f(x)}e^{Y(x)} \\
	    = & \ Ae^{\hat{f}(0)} \left[e^{-\hat{f}(0)}e^{\hat{f}(0)-(\mathcal{T}_Nf)(x)-Y(x)}e^{f(x)}e^{Y(x)}\right] \\ 
	    = & \ A_+e^{f_+(x)}.
    \end{split} 
\end{equation*}
Therefore, we define 
\begin{equation*}
	B=e^{Y}, \quad A_+=Ae^{\hat{f}(0)}, 
\end{equation*}
and 
\begin{equation*}
    f_+ =\ln(e^{-\hat{f}(0)}e^{\hat{f}(0)-(\mathcal{T}_Nf)(x)-Y(x)}e^{f(x)}e^{Y(x)}).
\end{equation*}

{\bf Step 3: Estimate residual terms.}

\begin{Lemma}\label{ert}
    Let $r' \in (0,r)$, $f \in C^{\omega}_r(\mathbb{T}^d,*)$ and let $f_{N}$ be the truncated Fourier series of $f$ at order $N$. We get 
    \begin{equation*}
        |f_N-f|_{r'} \leqslant \tilde{c}|f|_re^{-2\pi N(r-r')}(N+\frac{1}{r-r'})^d.  
    \end{equation*}
    Here, the constant $\tilde{c}$ is independent of $r,r',N$. 
\end{Lemma}
\begin{proof}
    By directly calculation, we have 
    \begin{equation*}
        \begin{split}
            |(f_N-f)(x)|_{r'} & =|\sum_{k\in \mathbb{Z}^d, |k|>N} \hat{f}(k)e^{2\pi i\langle k,x \rangle}|_{r'} \leqslant \sum_{k\in \mathbb{Z}^d, |k|>N} |\hat{f}(k)||e^{2\pi i\langle k,x \rangle}|_{r'} \\ 
            & \leqslant |f|_r \sum_{k\in \mathbb{Z}^d, |k|>N} e^{-2\pi|k|(r-r')} \leqslant C(d)|f|_r \int_N^{+\infty} l^{d-1}e^{-2\pi l(r-r')} dl. 
        \end{split}
    \end{equation*}
    Using integration by parts, this yields
    \begin{equation*}
        \int_N^{+\infty} l^{d-1}e^{-2\pi l(r-r')} dl=e^{-2\pi N(r-r')}\sum_{j=0}^{d-1}\frac{N^j}{(2\pi)^{d-j}(r-r')^{d-j}}\frac{(d-1)!}{j!}, 
    \end{equation*}
    and 
    \begin{equation*}
        \sum_{j=0}^{d-1}\frac{N^j}{(2\pi)^{d-j}(r-r')^{d-j}}\frac{(d-1)!}{j!} \leqslant C'(d)(N+\frac{1}{r-r'})^d. 
    \end{equation*}
    Consequently, we can get 
    \begin{equation*}
        |f_N-f|_{r'} \leqslant \tilde{c}|f|_re^{-2\pi N(r-r')}(N+\frac{1}{r-r'})^d, 
    \end{equation*}
    where the constant $\tilde{c}=\tilde{c}(d)$.  
\end{proof}

By \eqref{anaY}, we have the following estimates 
\begin{equation*}
    \begin{split}
        & |(\mathcal{T}_{N}f)(x)-\hat{f}(0)|_{r} \leqslant |\sum_{k \in \mathbb{Z}^d,0<|k|\leqslant N}\hat{f}(k)e^{2\pi i\langle k,x \rangle}|_{r} \leqslant \sum_{k \in \mathbb{Z}^d,0<|k|\leqslant N}|\hat{f}(k)||e^{2\pi i\langle k,x \rangle}|_{r} \\ 
        & \qquad \qquad \qquad \qquad \leqslant |f|_r\sum_{k \in \mathbb{Z}^d,0<|k|\leqslant N} e^{-2\pi|k|r}|e^{2\pi i\langle k,x \rangle}|_{r} \lesssim N^d|f|_r, \\  
        & |Y|_{r'} \leqslant \epsilon^{-2(\sigma+1/D)}|\mathcal{T}_{N}f-\hat{f}(0)|_{r} \lesssim \epsilon^{-2(\sigma+1/D)}N^d|f|_r \leqslant \frac{1}{2}\epsilon^{-(2\sigma+3/D)}|f|_r. 
    \end{split}
\end{equation*}
Since $N(r-r')=2|\ln\epsilon|$, by Lemma~\ref{ert}, one can get 
\begin{equation*}
    |\mathcal{R}_Nf|_{r'} \leqslant \tilde{c}|f|_re^{-2\pi N(r-r')}(N+\frac{1}{r-r'})^d \lesssim |f|_r e^{-2\pi N(r-r')}N^{d} \lesssim \epsilon^{4\pi}N^{d}|f|_r. 
\end{equation*}
Therefore,     
\begin{equation*}
    \begin{split}
        & |B-\mathrm{Id}|_{r'} \leqslant 2|Y|_{r'} \leqslant \epsilon^{-(2\sigma+3/D)}|f|_r \leqslant \epsilon^{1-2\sigma-3/D}, \\
        & \|A_+-A\| \leqslant \|A\|\|\mathrm{Id}-e^{\hat{f}(0)}\| \leqslant 2\|A\|\epsilon,  
    \end{split}
\end{equation*}
and by Baker-Campbell-Hausdorff formula,  
\begin{equation*}
    \begin{split}
         |f_+|_{r'} & =|\ln(e^{-\hat{f}(0)}e^{\hat{f}(0)-(\mathcal{T}_Nf)(x)-Y(x)}e^{f(x)}e^{Y(x)})|_{r'} \\ 
         & \lesssim |\mathcal{R}_{N}f+\frac{1}{2}([\hat{f}(0),\mathcal{T}_{N}f]+[\hat{f}(0),Y]-[\mathcal{T}_{N}f,f]+2[f,Y]-[\mathcal{T}_{N}f,Y])|_{r'} \\ 
         & \qquad+\text{ higher order terms} \\
         & \lesssim |\mathcal{R}_Nf|_{r'}+|f(\mathcal{T}_Nf)|_{r'}+|fY|_{r'}+|(\mathcal{R}_Nf)Y|_{r'} \\ 
         & \lesssim \epsilon^{4\pi}N^{d}|f|_r+|f|_{r'}N^d|f|_r+|f|_{r'}\epsilon^{-2(\sigma+1/D)}N^d|f|_r+\epsilon^{4\pi}N^d|f|_r\epsilon^{-(2\sigma+3/D)}|f|_r \\
         & \leqslant \epsilon^{2-2\sigma-3/D} . 
    \end{split}
\end{equation*}

This finishes the proof of Theorem \ref{KAM}.
\end{proof}

	\bibliographystyle{abbrv}
	\bibliography{ref}
\end{document}